\documentclass[english]{article}
\usepackage{preamble}

\title{Descent and cyclotomic redshift\\for chromatically localized algebraic $K$-theory}

\begin{document}

\maketitle

\begin{abstract}
    We prove that $\Tnp$-localized algebraic $K$-theory satisfies descent for $\pi$-finite $p$-group actions on stable $\infty$-categories of chromatic height up to $n$, extending a result of Clausen--Mathew--Naumann--Noel for finite $p$-groups.
    Using this, we show that it sends $\Tn$-local Galois extensions to $\Tnp$-local Galois extensions.
    Furthermore, we show that it sends cyclotomic extensions of height $n$ to cyclotomic extensions of height $n+1$, extending a result of Bhatt--Clausen--Mathew for $n=0$.
    As a consequence, we deduce that $\Knp$-localized $K$-theory satisfies hyperdescent along the cyclotomic tower of any $\Tn$-local ring.
    Counterexamples to such cyclotomic hyperdescent for $\Tnp$-localized $K$-theory were constructed by Burklund, Hahn, Levy and the third author, thereby disproving the telescope conjecture.
\end{abstract}

\begin{figure}[ht!]
    \centering
    \includegraphics[width=110mm]{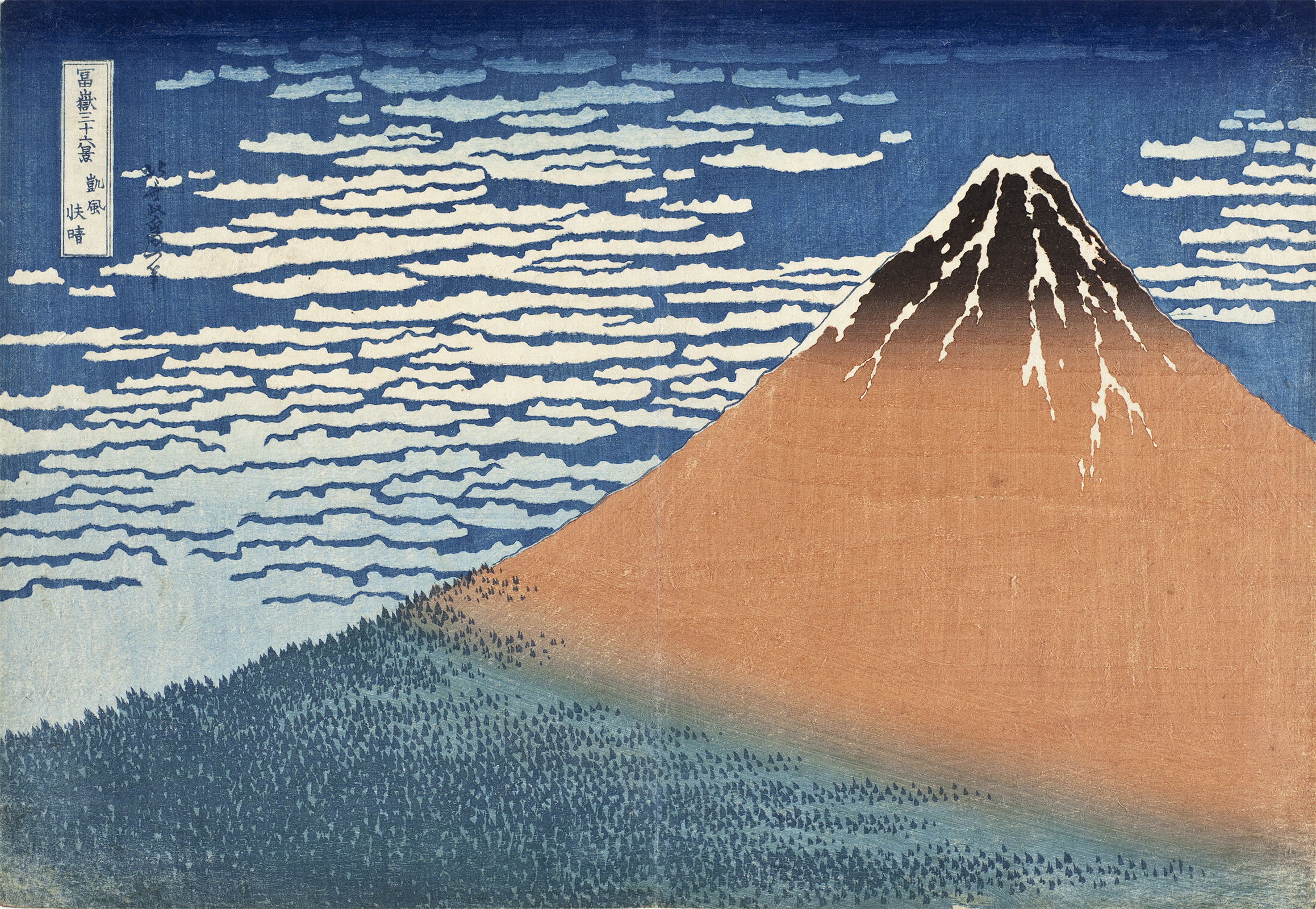}
    \caption*{Fine Wind, Clear Morning (Red Fuji) by Katsushika Hokusai.}
\end{figure}

\newpage
\tableofcontents{}
\newpage

\section{Introduction}

\subsection{Background}

The algebraic $K$-theory of rings, ring spectra and, more generally, of stable $\infty$-categories, is a long-studied invariant, situated at the junction of many fields of mathematics ranging from number theory to differential topology.
It is, however, notoriously difficult to compute, as among other things it possesses fairly weak descent properties. In particular, it fails to satisfy Galois descent even for ordinary fields.
The intricate nature of algebraic $K$-theory stems from the fact that its construction involves categorification -- passing from a ring spectrum to its $\infty$-category of modules. Thus, loosely speaking, the algebraic $K$-theory spectrum does not live in the same characteristic as the original ring spectrum. The field of chromatic homotopy theory provides a precise formalization of the notion of characteristic for spectra in the form of the chromatic height filtration. 
Based on computational evidence in chromatic height $1$, Ausoni--Rognes \cite{AR02, ausoni2008chromatic} formulated the far-reaching redshift conjecture, out of which emerged a wider philosophy, predicting the interaction of algebraic $K$-theory with chromatic height.
The conjecture states roughly that the process of categorification increases chromatic height by one.
Furthermore, Ausoni and Rognes conjectured that algebraic $K$-theory localized at the highest chromatic height posited by the redshift conjecture does satisfy descent for the Galois extensions of commutative ring spectra introduced by Rognes in \cite{RognesGal} and further studied by Mathew in \cite{AkhilGalois}.

Recent years have seen several breakthroughs in the study of such redshift phenomena \cite{hahn-wilson,yuan2021examples,ausoni2022adjunction,Null}.
Of particular relevance to this paper are the results of Clausen--Mathew--Naumann--Noel \cite{clausen2020descent} and Land--Mathew--Meier--Tamme \cite{land2020purity}, establishing Galois descent for chromatically localized algebraic $K$-theory.
Recall that an idempotent complete stable $\infty$-category $\cC$ is called \emph{$\Lnf$-local} if all of its mapping spectra are $\Lnf$-local, which roughly means that they are concentrated in chromatic heights lower or equal $n$.

\begin{thm}[{\cite[Theorem C and Proposition 4.1]{clausen2020descent}}]\label{CMNN1}
    Let $\cC$ be an $\Lnf$-local $\infty$-category acted by a finite $p$-group $G$, then there are canonical isomorphisms
    \[
        \LTnp K(\cC^{hG}) \iso \LTnp K(\cC)^{hG},
        \qquad \LTnp K(\cC)_{hG} \iso \LTnp K(\cC_{hG}).
    \]
\end{thm}

For example, for every $\Lnf$-local ring $R$, the $\infty$-category $\Perf(R)$ is $\Lnf$-local. Thus, combined with Galois descent for perfect modules, this readily implies the following:

\begin{cor}[{\cite[Corollary 4.16]{clausen2020descent}}]\label{CMNN2}
    Let $R \to S$ be a $\Tn$-local $G$-Galois extension where $G$ is a finite $p$-group.
    Then there is a canonical isomorphism
    \[
        \LTnp K(R) \iso (\LTnp K(S))^{hG}.
    \]
\end{cor}

It is known that \cref{CMNN1} may fail for arbitrary finite groups $G$ (e.g.\ of order prime to $p$), but the question of whether \cref{CMNN2} holds in this generality is still open. More generally, one can ask to what extent chromatically localized algebraic $K$-theory satisfies \textit{hyperdescent} for profinite Galois extensions. Of particular interest is the case of the Lubin--Tate spectrum $E_n$, which is a profinite Galois extension of the $K(n)$-local sphere with Morava stabilizer Galois group $\GG_n$. Namely, it is natural to ask whether the functor $U\mapsto \LTnp K(E_n^{hU})$ on open subgroups $U \le \GG_n$ corresponds to a hypersheaf on the site of continuous finite $\GG_n$-sets. As noted in \cite[Example 4.17]{clausen2020descent}, even on the pro-$p$ part of $\GG_n$, where \cref{CMNN2} implies that the resulting functor is a sheaf, the condition of being a \textit{hypersheaf} is not automatically implied. 

\begin{rem}
    To put this discussion in context, we note that for ordinary commutative rings (and schemes), Galois descent is a special case of \'{e}tale descent, which is a much studied subject in algebraic $K$-theory. In particular, the \'{e}tale sheafification of algebraic $K$-theory (known as \'{e}tale $K$-theory) is closely related to its chromatic $L_1$-localization, as observed by Waldhausen \cite[\S4]{waldhausen1984algebraic}. Furthermore, questions of hyperdescent for \'{e}tale $K$-theory are of fundamental importance to the construction of spectral sequences relating algebraic $K$-theory to \'{e}tale cohomology. We refer the reader to \cite{clausen2021hyperdescent} for the state-of-the-art results in this direction. 
\end{rem}

\subsection{Higher Descent}

Our first main result is an extension of \Cref{CMNN1} in a different direction, by considering actions of \textit{higher groups}. However, as we shall explain in the next subsection, this seemingly unrelated generalization also has implications to the above questions about Galois (hyper)descent for ordinary (pro)finite groups.
To state our results, we say that a group $G$ in spaces is an \emph{$m$-finite $p$-group} if it is $m$-truncated and all of its homotopy groups are finite $p$-groups,
and that $G$ is a \emph{$\pi$-finite $p$-group} if it is an $m$-finite $p$-group for some $m$.

\begin{theorem}[{Higher Descent, \cref{main-thm-cor}}]\label{main-thm-intro}
    Let $\cC$ be an $\Lnf$-local $\infty$-category acted by a $\pi$-finite $p$-group $G$, then there are canonical isomorphisms
    \[
        \LTnp K(\cC^{hG}) \iso \LTnp K(\cC)^{hG},
        \qquad \LTnp K(\cC)_{hG} \iso \LTnp K(\cC_{hG}).
    \]
\end{theorem}

From this, we deduce the following:

\begin{cor}[{\cref{galois-comp}}]\label{galois-comp-intro}
    Let $R \to S$ be a $\Tn$-local $G$-Galois extension where $G$ is an $n$-finite\footnote{Recall that $G$ is $n$-finite if and only if $BG$ is $(n+1)$-finite. See the discussion above \cref{galois-comp} for an elaboration on the truncatedness assumption.} $p$-group. 
    Then $\LTnp K(R) \to \LTnp K(S)$ is a $\Tnp$-local $G$-Galois extension.
\end{cor}

Note that the first condition in Rognes' definition of a $G$-Galois extension is that the source is canonically isomorphic to the $G$-fixed points of the target. Thus, \Cref{galois-comp-intro} is simultaneously a generalization and a strengthening of \Cref{CMNN2}. On the other hand, in our case the second condition in Rognes' definition of a $G$-Galois extension is superfluous, via \cite[Proposition 2.27]{barthel2022chromatic}. 

\cref{main-thm-intro} is closely related to the higher semiadditivity of $\Tn$-local spectra (which indeed features in its proof as explained later in the introduction). In \cite{moshe2021higher} the authors have constructed and studied the \textit{universal} higher semiadditive approximation of algebraic $K$-theory. In \cref{sa-K}, we deduce that for $\Lnf$-local $\infty$-categories, 
this higher semiadditive algebraic $K$-theory coincides $\Tnp$-locally with (ordinary) algebraic $K$-theory.

\subsection{Cyclotomic Redshift}

As already mentioned above, aside from their intrinsic interest, \Cref{main-thm-intro} and \Cref{galois-comp-intro}
also have applications to some particular cases of Galois descent with respect to \textit{ordinary} groups. We begin by recalling that in \cite{bhatt2020remarks}, Bhatt--Clausen--Mathew studied $\To$-local $K$-theory, with a focus on discrete (commutative) rings.
In particular, they have shown a result which can be phrased as follows (see also \cite{mitchell2000topological, dwyer1998k} for closely related results):

\begin{thm}[{{\cite[Theorem 1.4]{bhatt2020remarks}}}]
\label{BCM_cycloshift}
    Let $R$ be a commutative ring, then there is a $\ZZ_p^\times$-equivariant isomorphism
    \[
        \LTo K(R[\omega_{p^\infty}]) \simeq \LTo K(R) \otimes \KUp
        \qin \CAlg(\Sp_{\To}),
    \]
    where $\ZZ_p^\times$ acts on the $p^\infty$-th cyclotomic extension of $R$ and on the $p$-complete complex $K$-theory spectrum $\KUp$ by Galois automorphisms and the Adams operations respectively.
\end{thm}

In \cite{carmeli2021chromatic} (later extended in \cite{barthel2022chromatic}) a higher height analogue of cyclotomic extensions was studied in the context of higher semiadditive $\infty$-categories.
Let $\cC$ be a presentably symmetric monoidal stable $\infty$-semiadditive $\infty$-category of height $n$, and let $R \in \CAlg(\cC)$.
Then, there is some idempotent in the group algebra $R[B^n C_{p^\infty}]$, constructed using higher semiadditive integrals.
This idempotent splits off a direct factor called the (height $n$) \emph{$p^\infty$-cyclotomic extension} of $R$, and denoted by $\cyc{R}{p^\infty}{n}$.
As the name suggests, height $0$ cyclotomic extensions reproduce ordinary cyclotomic extensions.

\begin{example}\label{ex:cyclo-Kn}
    In the case $\cC = \Sp_{\Kn}$ and $R = \Sph_{\Kn}$, the $p^\infty$-cyclotomic extension $\cyc{\Sph_{\Kn}}{p^\infty}{n}$ is the ring $R_n$ studied by Westerland in \cite{Westerland}.
    Namely, it is the (continuous) homotopy fixed points of $E_n$ by the kernel of $\det\colon \Morex_n \to \ZZ_p^\times$, whence it is a $\ZZ_p^\times$-Galois extension of $\Sph_{\Kn}$.
    Therefore, the case $\cC = \Sp_{\Tn}$ provides a $\Tn$-local lift of this Galois extension.
\end{example}

Specializing this example to height $n = 1$, we see that $\cyc{\Sph_{\To}}{p^\infty}{1} \simeq \KUp$.
Consequently, for commutative rings $R$ with $p \in R^\times$, \Cref{BCM_cycloshift}
can be rephrased as a $\ZZ_p^\times$-equivariant isomorphism
\[
    \LTo K(\cyc{R}{p^\infty}{0}) \simeq \cyc{\LTo K(R)}{p^\infty}{1}.
\]
Our second main theorem is the extension of this isomorphism to higher chromatic heights.

\begin{theorem}[{Cyclotomic Redshift, \cref{cyclo-main}}]\label{cyclo-main-intro}
    Let $R$ be a $\Tn$-local ring spectrum, then there is a $\ZZ_p^\times$-equivariant isomorphism
    \[
        \LTnp K(\cyc{R}{p^\infty}{n}) \simeq \cyc{\LTnp K(R)}{p^\infty}{n+1}.
    \]
\end{theorem}

Higher height roots of unity, classified by higher cyclotomic extensions, were also used in \cite{barthel2022chromatic} for constructing higher height analogous of the discrete Fourier transform and Kummer theory. We show in \cref{fourier-comp} and \cref{kummer-comp} respectively, that these constructions are also suitably intertwined by the functor $\LTnp K$.    

The proof of \Cref{cyclo-main-intro} proceeds by applying \cref{main-thm-intro} to $G = B^n C_{p^\infty}$ to obtain an isomorphism
\[
    \LTnp K(R)[B^{n+1} C_{p^\infty}] \iso 
    \LTnp K(R[B^n C_{p^\infty}]).
\]
We then show that the idempotents splitting the corresponding cyclotomic extensions agree (equivariantly), implying the result.
In fact, the analogue of \Cref{cyclo-main-intro} also holds for (and follows from) the finite cyclotomic extensions $\cyc{R}{p^r}{n}$ for every integer $r \ge 0$.
In particular, taking $r =1$ and odd prime $p$, we obtain for every height $n$ a non-trivial instance of a $\Tn$-local Galois extension of order \textit{prime to $p$}, which is mapped to a $\Tnp$-local Galois extension by $\LTnp K$.

\subsection{Hyperdescent and the Telescope Conjecture}\label{tel-conj}

Cyclotomic redshift (\cref{cyclo-main-intro}) also has implications for hyperdescent for algebraic $K$-theory and the telescope conjecture.
The starting point of this discussion is the question of faithfulness of the cyclotomic extensions over $\Tn$-local ring spectra.
While each of the finite extensions $\cyc{R}{p^r}{n}$  is faithful over such a ring spectrum $R$, the infinite one $\cyc{R}{p^\infty}{n}$ has a priori no reason to be.
In \cite[\S7.3]{barthel2022chromatic}, the authors constructed the universal localization of $\SpTn$ in which infinite cyclotomic extensions become faithful, called the \textit{cyclotomic completion}.
Moreover, the cyclotomic completion is shown to be smashing and in the case of an odd prime $p$ given by the formula\footnote{The case $p=2$ requires only a slight modification.}
\[
    \cycomp{R} \simeq
    \cyc{R}{p^\infty}{n}^{hG} \qin 
    \Alg(\SpTn)
\]
for the action of the discrete dense subgroup
\[
    G := \Fp^\times \times \ZZ \leq 
    \ZZ_p^\times.
\]
Also observe that by \cref{ex:cyclo-Kn}, we have $\cyc{\Sph_{\Kn}}{p^\infty}{n} \simeq R_n$, and it follows from Devinatz--Hopkins theory   \cite{devinatz2004homotopy} that $R_n^{hG} \simeq \Sph_{\Kn}$  (see also \cite[Theorem 1.1]{mor2023picard}).
Therefore, all $\Kn$-local spectra are cyclotomically complete.

In the present paper, we show that cyclotomic completion is intimately related to hyperdescent.
The compatible system of finite cyclotomic extensions $\cyc{R}{p^r}{n}$ assembles into a sheaf on the site of continuous finite $\ZZ_p^\times$-sets, whose stalk is the $p^\infty$-cyclotomic extension $\cyc{R}{p^\infty}{n}$.
In \cref{cyclo-sheaf-galois}, we show that this sheaf is a hypersheaf if and only if $R$ is cyclotomically complete.

Cyclotomic redshift (\cref{cyclo-main-intro}) allows us to transfer questions about cyclotomic extensions between height $n$ and height $n+1$.
In particular, it shows that the sheaf $\LTnp K(\cyc{R}{p^r}{n})$ is isomorphic to the sheaf $\cyc{\LTnp K(R)}{p^r}{n+1}$.
Using the fact that every $\Knp$-local spectrum is cyclotomically complete, we thus obtain a non-trivial instance of hyperdescent for $\Knp$-localized algebraic $K$-theory.

\begin{theorem}[{Cyclotomic Hyperdescent, \cref{Knp-hyper}}]\label{cyc-descent}
    Let $R$ be a $\Tn$-local ring spectrum, then the sheaf determined by the values $\LKnp K(\cyc{R}{p^r}{n})$ is a hypersheaf.
\end{theorem}

It is natural to ask whether the hyperdescent along the cyclotomic tower holds already on the level of telescopic localizations, or, equivalently, whether $\LTnp K(R)$ is cyclotomically complete for every $\Tn$-local ring spectrum $R$.
This touches upon the subtle distinction between the $\Tn$-local and the $\Kn$-local categories, which is the subject of Ravenel's long-standing telescope conjecture.
Ravenel originally conjectured that  $\SpKn = \SpTn$ for all heights and primes, but while it is known to hold for $n = 0,1$, it was soon suspected to be false for higher chromatic heights.
Burklund, Hahn, Levy and the third author \cite{Telefalse} constructed \emph{counterexamples} to hyperdescent of the $\Tnp$-localized $K$-theory of cyclotomic towers for every $n \geq 1$ and every prime number.
Therefore, combined with \Cref{cyc-descent}, \emph{disproving} the telescope conjecture at height $2$ and above.

For the convenience of the reader, we now give a rough sketch of the strategy of \cite{Telefalse} and its relation to cyclotomic completion.
For simplicity of the exposition, we focus on the case $n = 2$ and odd prime $p$.
In this case we take $R = \Sph_{K(1)}$, and the goal is to show that the sheaf of cyclotomic extensions $L_{T(2)} K(\cyc{\Sph_{K(1)}}{p^r}{1})$ is not a hypersheaf.
As explained above, this is equivalent to showing that $L_{T(2)}K(\Sph_{K(1)})$ is not cyclotomically complete, namely, that the cyclotomic completion map
\[
    L_{T(2)}K(\Sph_{K(1)})
    \too \cycomp{L_{T(2)}K(\Sph_{K(1)})}
\]
is not an isomorphism.
Since the $p^\infty$-cyclotomic extension of $\Sph_{K(1)}$ is $\KUp$, cyclotomic redshift (\Cref{cyclo-main-intro}) provides a $\ZZ_p^\times$-equivariant isomorphism
\[
    \cyc{L_{T(2)}K(\Sph_{K(1)})}{p^\infty}{2} 
    \iso L_{T(2)}K(\KUp).
\]
Hence, the cyclotomic completion is given by 
\[
    \cycomp{L_{T(2)}K(\Sph_{K(1)})} \simeq
    L_{T(2)}K(\KUp)^{hG}.
\]
Moreover, we have $\Sph_{K(1)} \simeq (\KUp)^{hG}$, and it is not hard to see that the cyclotomic completion map identifies with the assembly map of $G$-fixed points
\[
    L_{T(2)}K((\KUp)^{hG}) \too 
    L_{T(2)}K(\KUp)^{hG}.
\]
One can further show that the assembly map for fixed points with respect to the finite subgroup $\Fp^\times \leq G$ is an 
isomorphism in this case, so, writing $L = (\KUp)^{h\Fp^\times}$
for the non-connective $p$-complete Adams summand, we can identify the cyclotomic completion map with the assembly map for $\ZZ$-fixed points (see \cref{iwa-comp})
\[
    L_{T(2)}K(L^{h\ZZ}) \too 
    L_{T(2)}K(L)^{h\ZZ}.
\]
Using trace methods and the seminal computations of \cite{AR02} of the topological cyclic homology of the connective $p$-complete Adams summand $\ell$, this map is shown \textit{not} to be an isomorphism, disproving the telescope conjecture.
The argument for higher heights requires a more involved variant of these ideas, replacing $\ell$ with certain $\mathbb{E}_3$-forms of $\BPn$ constructed by \cite{hahn-wilson}.




\subsection{Outline of the Proof}\label{outline-proof}

Let $\Catperf$ be the $\infty$-category of stable idempotent complete $\infty$-categories, and let $\CatLnf \subset \Catperf$ denote the full subcategory of $\Lnf$-local $\infty$-categories.
\Cref{main-thm-intro} can be rephrased as saying that the functor
\[
    \LTnp K\colon \CatLnf \too \SpTnp
\]
preserves all $\pi$-finite $p$-space indexed colimits and limits.
Observe that if the space is discrete, i.e.\ $0$-finite, then this is precisely the preservation of finite products.
The case of $1$-finite $p$-spaces is precisely \Cref{CMNN1}.
The argument then proceeds inductively on the level of truncatedness $m \geq 2$.
The proof can be divided into three steps, which we now describe.

\subsubsection{Reduction to Constant Colimits of Monochromatic $\infty$-Categories}

A key ingredient in this reduction is the categorical analogue of \emph{monochromatization}.
Recall that a spectrum $X \in \Sp$ is called $n$-monochromatic if it is $\Lnf$-local and $\Lnmf(X) = 0$.
The inclusion of $n$-monochromatic spectra $\Mnf\Sp \into \Lnf\Sp$ admits a right adjoint $\Mnf\colon \Lnf\Sp \too \Mnf\Sp$, which fits into a natural exact sequence
\[
    \Mnf(X) \too X \too \Lnmf(X).
\]
Analogously, we say that an $\infty$-category $\cC \in \Catperf$ is \emph{$n$-monochromatic} if it is $\Lnf$-local and $\Lnmf(\cC) = 0$, where $\Lnmf\colon \Catperf \to \Cat_{\Lnmf}$ is the left adjoint of the inclusion. 
We show that the inclusion $\CatMnf \into \CatLnf$ of the $n$-monochromatic $\infty$-categories admits a right adjoint $\Mnf\colon \CatLnf \too \CatMnf$, which fits into a natural exact (i.e.\ Verdier) sequence
\[
    \Mnf(\cC) \too \cC \too \Lnmf(\cC).
\]
This construction enjoys three key properties:

\begin{enumerate}[ref=Property~(\arabic*)]    
    \item\label{Mnf-lim-colim-intro} The functor $\Mnf\colon \CatLnf \to \CatMnf$ preserves all limits and colimits (see \cref{Mnf-lim-colim}).
    
    \item\label{K-monochrom-iso-intro} For every $\cC \in \CatLnf$, the inclusion $\Mnf(\cC) \hookrightarrow \cC$ induces a ``purity isomorphism'' (see \cref{K-monochrom-iso} and \cite[Theorem C]{clausen2020descent})
    \[
        \LTnp K(\Mnf(\cC)) \iso \LTnp K(\cC).
    \]

    \item\label{CatMnf-inf-sa-intro} The $\infty$-category $\CatMnf$ is $\infty$-semiadditive (see \cref{CatMnf-inf-sa}).

\end{enumerate}

Turning back to the proof of \cref{main-thm-intro}, \ref{Mnf-lim-colim-intro} and \ref{K-monochrom-iso-intro} together imply that it suffices to prove that the functor restricted to $n$-monochromatic $\infty$-categories
\[
    \LTnp K\colon \CatMnf \too \SpTnp
\]
preserves $\pi$-finite $p$-space indexed limits and colimits.
Now, by \ref{CatMnf-inf-sa-intro}, this functor is between two \textit{$\infty$-semiadditive} $\infty$-categories, from which we gain two things:
First, the preservation of $\pi$-finite $p$-space indexed colimits implies the same for limits, allowing us to consider only colimits.
Second, it suffices to consider only \emph{constant} colimits concentrated in a homotopy degree $m$ (see \cref{red-const} and \cref{red-concentrated}).
In other words, we are reduced to showing that for any $n$-monochromatic $\infty$-category $\cC$ and $\pi$-finite $p$-space $A$ concentrated in homotopy degree $m$ the assembly map
\[
    \LTnp K(\cC)[A] \too \LTnp K(\cC[A])
\]
is an isomorphism.

\subsubsection{Reduction to Categories of Modules}

By the Schwede--Shipley theorem \cite{schwede2003stable}, any $\cC \in \Catperf$ is a filtered colimit of $\infty$-categories of the form $\Perf(R)$, where $R$ is the endomorphism ring spectrum of some object of $\cC$ (see \cref{Schwede-Shipley}).
Since algebraic $K$-theory commutes with filtered colimits, we are reduced to considering only such $\infty$-categories.
Moreover, as our $\cC$ is $n$-monochromatic, the ring spectra $R$ are $n$-monochromatic as well.
Thus, we are reduced to showing that for any $n$-monochromatic ring spectrum $R$ and $\pi$-finite $p$-space $A$ concentrated in homotopy degree $m$ the map
\[
    \LTnp K(R)[A] \too \LTnp K(R[\Omega A])
\]
is an isomorphism.

\subsubsection{Proof of the Special Case}

We thank Akhil Mathew for suggesting the following inductive argument.
The key ingredient facilitating the inductive step is \cite[Corollary 4.31]{land2020purity}, stating that the $\Tnp$-localized algebraic $K$-theory of ring spectra preserves sifted colimits when $n \geq 1$.
The base of the induction is $m=1$, which is provided by \Cref{CMNN1}.
When $m \geq 2$, the space $\Omega A$ is connected, therefore, we can present $A$ and $\Omega A$ compatibly via their bar constructions
\[
    A \simeq \colim_{\Delta^\op} (\Omega A)^\bullet,
    \qquad
    \Omega A \simeq \colim_{\Delta^\op} (\Omega^2 A)^\bullet.
\]
Each of the spaces $(\Omega A)^k$ is an ($m-1$)-finite $p$-space, for which the result already holds by the inductive hypothesis, and we finish by the preservation of sifted colimits.

\subsection{Acknowledgements}

We thank the anonymous referee for carefully reading the manuscript and for the numerous helpful comments and suggestions.
We thank Akhil Mathew for suggesting the argument for the proof of the key special case of the main theorem described above. We also thank Dustin Clausen for useful conversations and Maxime Ramzi for several useful comments, suggestions and references. The first and third named authors would like to thank the Massachusetts Institute of Technology for its hospitality, in which part of this paper were written. 
The first and last named authors would like to thank the University of Copenhagen for its hospitality during and after the ``Masterclass on Topological Hochschild Homology and Zeta Values'', in which parts of this paper were written, and to Omer Carmeli for his hospitality and for stimulating conversations.
The second author is partially supported by the Danish National Research Foundation through the
Copenhagen Centre for Geometry and Topology (DNRF151). The third author was supported by ISF1588/18, BSF 2018389 and the ERC under the European Union's Horizon 2020 research and innovation program (grant agreement No.\ 101125896).

\section{Monochromatic Categories}

In this section we set up the categorical framework for the higher descent results in chromatically localized algebraic $K$-theory. That is, we define $\CatLnf$, $\CatMnf$ and other related categories, and study them by means of various characterizations, closure properties, examples, adjunctions etc.
In particular, we establish 
\ref{Mnf-lim-colim-intro}, \ref{K-monochrom-iso-intro} and \ref{CatMnf-inf-sa-intro} discussed above.
Here and in the rest of the paper we use the term `category' to mean `$\infty$-category'.

\subsection{Recollections of Stable Categories}

We begin by setting notation for some categories of stable categories and recalling some well known facts about them. 

\begin{defn} 
    We denote by $\mdef{\Catperf} \subset \Cat$ the (non-full) subcategory of small, idempotent complete stable categories and exact functors between them, endowed with the Lurie tensor product.
\end{defn}

\begin{defn}
    We denote by $\mdef{\Prlstw}\subset \Prlst$ the (non-full) symmetric monoidal subcategory of compactly generated stable categories and colimit preserving functors between them that take compact objects to compact objects (equivalently, that have a colimit preserving right adjoint).
\end{defn}

We recall the following from \cite{HA}:

\begin{prop}\label{ind-w}
    The inclusion $\Prlstw \into \Prlst$ is colimit preserving. Moreover, the ind-completion functor
    \[
    \Ind\colon \Catperf \too \Prlst
    \]
    is colimit preserving and symmetric monoidal, and factors through a symmetric monoidal equivalence
    \[
        \Ind\colon \Catperf
        \adj \Prlstw
        \noloc (-)^\omega.
    \]
\end{prop}

\begin{proof}
    The first part is \cite[Lemma 5.3.2.9]{HA}. The second follows from \cite[Remark 4.8.1.8]{HA},
    and the third follows from \cite[Lemma 5.3.2.11(3)]{HA}.
\end{proof}

\subsubsection{Exact Sequences}

We give a brief recollection of exact sequences in $\Catperf$, also known as \emph{Verdier sequences}.
For a comprehensive treatment we refer the reader to \cite[Appendix A]{hermitian-2}.

The category $\Catperf$ is pointed, with zero object the trivial category $\pt$. Accordingly, we can talk about null sequences in $\Catperf$.

\begin{defn}
    A null sequence
    \[
        \cC \too[F] \cD \too[G] \cE
    \] 
    in $\Catperf$ is called an \mdef{exact sequence} if it is both a fiber and a cofiber sequence.
\end{defn}

\begin{rem}
    Note that, unlike in the case of stable categories such as $\Sp$, it is a \emph{property} of a composable pair of morphisms in $\Catperf$ to be a null sequence: It is equivalent to the property that the composition $GF$ carries every object of $\cC$ to a zero object of $\cE$. 
    Accordingly, it is a property of such a pair to be an exact sequence.
\end{rem}

\begin{rem}\label{verd_seq_equiv_chars}
    There are other equivalent characterizations of exact sequences.
    For example, a sequence as in the above definition is an exact sequence if and only if $F$ is fully faithful and $G$ exhibit $\cE$ as a localization (or ``Verdier quotient'') of $\cD$ with respect to the morphisms whose (co)fibers lie in $\cC$.
    This is also equivalent to $F$ being fully faithful and the sequence being a cofiber sequence.
\end{rem}

\subsubsection{Perfect Modules}

One source of examples of categories in $\Catperf$ is perfect modules over ring spectra. 
We recall the following construction from \cite{HA}:

\begin{prop}\label{LMod}
    The construction taking a ring spectrum to the category of its left module spectra assembles into a symmetric monoidal functor
    $$
        \LMod\colon \Alg(\Sp) \too \Prlstw.
    $$
\end{prop}

\begin{proof}
    By \cite[Remark 4.8.5.17]{HA}, there is a symmetric monoidal functor
    $$
        \LMod\colon \Alg(\Sp) \too \Mod_\Sp(\Catall)
    $$
    where $\Catall$ is the category of cocomplete categories and colimit preserving functors.
    As in \cite[Notation 4.8.5.10]{HA}, $\LMod_R$ is presentable, thus, $\LMod$ factors through $\Mod_\Sp(\Prl)$, which by \cite[Proposition 4.8.2.18]{HA} is equivalent to $\Prlst$.
    By \cite[Proposition 7.2.4.2]{HA}, $\LMod_R$ is compactly generated.
    Moreover, by \cite[Corollary 4.2.3.7(2)]{HA}, for every morphism $R \to S$, the extension of scalars functor $\LMod_R \to \LMod_S$ admits a right adjoint which is itself a left adjoint. Therefore, $\LMod$ factors through the symmetric monoidal subcategory $\Prlstw \subset \Prlst$.
\end{proof}

\begin{defn}
    We define the symmetric monoidal functor
    \[
    \mdef{\Perf}\colon \Alg(\Sp) \to \Catperf
    \]
    as the composition
    \[
        \Alg(\Sp)
        \too[\LMod] \Prlstw
        \too[(-)^\omega]
        \Catperf,
    \]
    which takes a ring spectrum to the category of its perfect (i.e.\ compact) module spectra. 
\end{defn}

We now record the following well known corollary of the Schwede--Shipley theorem \cite{schwede2003stable}.

\begin{prop}\label{Schwede-Shipley}
    Every $\cC \in \Catperf$ is a filtered colimit of categories of the form $\Perf(R)$, where $R$ is the endomorphism spectrum of an object in $\cC$.
\end{prop}

\begin{proof}
    Let $P$ denote the collection of thick subcategories of $\cC$ generated by a single object, which is a filtered poset with respect to inclusion.
    For each thick subcategory $\cC_0 \in P$ generated by an object $X \in \cC_0$, we have $\cC_0 \simeq \Ind(\cC_0)^\omega$ by \cref{ind-w}.
    This implies that $X \in \Ind(\cC_0)$ is a compact generator, which by the Schwede--Shipley theorem \cite{schwede2003stable} (see also \cite[Theorem 7.1.2.1]{HA}) implies that
    \[
        \cC_0
        \simeq \Ind(\cC_0)^\omega
        \simeq \LMod_{\End(X)}^\omega(\Sp)
        =: \Perf(\End(X)).
    \]
    Since every $X \in \cC$ belongs to the thick subcategory generated by $X$, we have
    \[
        \colim_{\cC_0 \in P} \cC_0 \iso \cC,
    \]
    and hence $\cC$ is a filtered colimit of categories of the required form.
\end{proof}

\subsection{Categorical Localization and Monochromatization}

\subsubsection{Localization of Stable Categories}\label{subsec-lcat}

A symmetric monoidal localization $L\colon \Sp \to L\Sp$ exhibits $L\Sp$ as an idempotent algebra in $\Prlst$ in the sense of \cite[Definition 4.8.2.1]{HA}.
By \cite[Proposition 4.8.2.10]{HA}, it is a property of $\cC \in \Prl$ to be a module over $L\Sp$, which is precisely that all the mapping spectra of $\cC$ are in $L\Sp$ (see, e.g., \cite[Proposition 5.2.10]{AmbiHeight}). We can similarly consider the analogous notion for $\Catperf$.

\begin{defn}
    Let $L\colon \Sp \to L\Sp$ be a symmetric monoidal localization.
    We let $\mdef{\Cat_L} \subset \Catperf$ be the full subcategory of those categories all of whose mapping spectra are in $L\Sp$. We refer to objects of $\Cat_L$ as \tdef{$L$-local categories}.
\end{defn}

This notion is closely related to being $L\Sp$-linear in the presentable context.
We recall that since $L\Sp$ is a localization of spectra, it is an idempotent algebra in $\Prlst$, and thus being a module over it is a property.
Namely, $\cD \in \Prlst$ is a module over $L\Sp$ if and only if all mapping spectra of $\cD$ are in $L\Sp$.
In particular, given $\cC \in \Catperf$, if $\Ind(\cC)$ is a module over $L\Sp$ in $\Prlst$, then all the mapping spectra of $\cC \subset \Ind(\cC)$ are $L$-local, i.e.\ $\cC$ is $L$-local.
However, the converse does not always hold:

\begin{example}
    Consider the $p$-completion functor $L := (-)_p\colon \Sp \to \Spp$.
    The category $\Perf(\Sph_p)$ of perfect modules over the $p$-complete sphere has $p$-complete mapping spectra, and hence is $L$-local.
    On the other hand, $\Ind(\Perf(\Sph_p)) \simeq \Mod_{\Sph_p}(\Sp)$ is not $\Spp$-linear, as for example
    \[
        \hom(\Sph_p, \Sph_p[1/p]) \simeq \Sph_p[1/p]
    \]
    which is not $p$-complete.
\end{example}

The connection between $L$-local categories and $L\Sp$-linear categories in $\Prlst$ is tightened when $L$ is further assumed to be \emph{smashing}, via the equivalence between $\Catperf$ and $\Prlstw$ from \cref{ind-w}.
When $L$ is smashing, it preserves compact objects and so exhibits $L\Sp$ as an idempotent algebra in $\Prlstw$, and hence being a module over it in $\Prlstw$ is a property.
In fact, it is the same property as in $\Prlst$, that is, $\cD \in \Prlstw$ is a module over $L\Sp$, if and only if the map
\[
    \cD \otimes \Sp \too \cD \otimes L\Sp 
\]
is an equivalence, which is the same condition whether we view it in $\Prlst$ or $\Prlstw$.

Applying the equivalence $(-)^\omega$ of \cref{ind-w} to $L\Sp \in \Prlstw$, we obtain an idempotent algebra 
\[
    \Perf(L\Sph) = L\Sp^\omega \qin \Catperf.
\]
Again by \cite[Proposition 4.8.2.10]{HA}, it is a property of a category in $\Catperf$ to admit a module structure over $\Perf(L\Sph)$, and we now show that all properties described above are equivalent.

\begin{prop}\label{Char_CatL}
    Let $L\colon \Sp \to L\Sp$ be a smashing localization.
    For $\cC \in \Catperf$ the following are equivalent:
    \begin{enumerate}
        \item $\cC \in \Cat_L$.
    
        \item $\cC$ is a module over $\Perf(L\Sph)$ in $\Catperf$.

        \item $\Ind(\cC)$ is a module over $L\Sp$ in $\Prlst$.
    \end{enumerate}
\end{prop}

\begin{proof}
    We first show that (2) and (3) are equivalent.
    Recall from \cref{ind-w} that there is a symmetric monoidal equivalence
    \[
        \Ind\colon \Catperf
        \adj \Prlstw
        \noloc (-)^\omega.
    \]
    Thus, condition (2) is equivalent to $\Ind(\cC)$ being an $L\Sp$-module in $\Prlstw$, which as explained above, is the same as being an $L\Sp$-module in $\Prlst$, i.e.\ condition (3).

    Assuming condition (3), all the mapping spectra in $\Ind(\cC)$ are in $L\Sp$, and since $\cC \subset \Ind(\cC)$ is a full stable subcategory, the same holds for $\cC$, i.e.\ we get condition (1). 

    Finally, assuming (1), to prove (3) it suffices to show that all the mapping spectra in $\Ind(\cC)$ are in $L\Sp$. 
    Given $X,Y\in \Ind(\cC)$ we can write $X = \colim X_a$ and $Y = \colim Y_b$ as filtered colimits of objects in the essential image of $\cC$. 
    We thus get
    \[
        \hom(X,Y) \simeq 
        \invlim_a \hom(X_a, \colim_b Y_b) \simeq
        \invlim_a \colim_b \hom(X_a, Y_b) \qin \Sp,
    \]
    where the second step uses the fact that the essential image of $\cC$ in $\Ind(\cC)$ consists of compact objects. 
    Since $L\Sp$ is a smashing localization, it is closed under both limits and colimits in $\Sp$ and hence we get condition (3).
\end{proof}

\begin{cor}\label{CatL-closure}
    Let $L\colon \Sp \to L\Sp$ be a smashing localization.
    The full subcategory $\Cat_L \subset \Catperf$ is symmetric monoidal, closed under (small) limits and colimits, and the left adjoint $\Catperf \to \Cat_L$ to the inclusion $\Cat_L \into \Catperf$ is given by tensoring with $\Perf(L\Sph)$ and hence is canonically symmetric monoidal.
\end{cor}

\begin{proof}
    By \Cref{Char_CatL}, the inclusion $\Cat_L \into \Catperf$ identifies with the forgetful functor $\Mod_{\Perf(L\Sph)}(\Catperf) \to \Catperf$.
    The claims about the symmetric monoidal structure follow immediately, and the claims about limits and colimits follow from \cite[Corollary 4.2.3.3 and Corollary 4.2.3.5]{HA}.
\end{proof}

By abuse of notation, we denote by $L\colon \Catperf \to \Cat_L$
also the symmetric monoidal left adjoint constructed in \Cref{CatL-closure}. 
The following example is of principal interest to our study:

\begin{example}
    For every $n$ and an (implicit) prime $p$, we let $\Tn$ be the telescope on some finite spectrum of type $n$ and define the finite chromatic localization
    \[
        \Lnf \colon \Sp \too \Lnf \Sp 
    \]
    to be the Bousfield localization with respect to $T(0)\oplus \dots \oplus \Tn$.
    Following the conventions of \cite{clausen2020descent,land2020purity}, we choose $T(0) = \Sph[1/p]$ to be the telescope of $v_0 = p$ on the sphere spectrum.
    Recall from \cite{miller1992finite} that for all $n\ge 0$, the finite chromatic localizations are smashing. 
    Hence, the full subcategory $\mdef{\CatLnf} \subset \Catperf$ of \tdef{$\Lnf$-local categories} identifies with modules over $\Perf(\Lnf \Sph)$.
\end{example}

\subsubsection{Monochromatization of Stable Categories}

Recall that a  spectrum $X$ is called \emph{$n$-monochromatic} if it is $\Lnf$-local and  $\Lnmf (X) = 0$ (where, by convention, $L_{-1}^f(X)$ is defined to be the zero object, so that every $L_0^f$-local spectrum is $0$-monochromatic).
The inclusion of the full subcategory $\mdef{\Mnf\Sp} \into \Lnf\Sp$ of $n$-monochromatic spectra admits a right adjoint 
\[
    \mdef{\Mnf}\colon \Lnf \Sp \too \Mnf \Sp,
\] 
which fits into a natural exact sequence 
\[
    \Mnf(X) \too X \too \Lnmf(X).
\]
The functor $\Mnf$ is both limit preserving, being a right adjoint, and colimit preserving, as it is given by tensoring with $\Mnf(\Sph)$ and $\Mnf\Sp$ is closed under colimits in $\Sp$. 
We shall now discuss a categorical analogue of this setup. 

\begin{defn}
    A category $\cC \in \Catperf$ is \tdef{$n$-monochromatic} if it is $\Lnf$-local and $\Lnmf(\cC) = 0$. 
    We denote by $\mdef{\CatMnf} \subset \CatLnf$ the full subcategory of $n$-monochromatic categories. 
    For every $\cC \in \CatLnf$, we denote by $\mdef{\Mnf(\cC)} \to \cC$ the fiber of the localization map $\cC \to \Lnmf(\cC)$ formed in $\Catperf$.
\end{defn}

We remark that every $L_0^f$-local category is by convention $0$-monochromatic.
Recall that the localization $\PerfLnf \to \PerfLnmf$ participates in an exact sequence (see for example \cite[Lemma 3.7]{land2020purity}), and its fiber is, by construction, the monochromatization of $\PerfLnf$, which we can describe explicitly.

\begin{prop} \label{ex:monochromatics_in_spectra}
    $\Mnf(\PerfLnf)$ is the thick subcategory of $\PerfLnf$ generated by $\Tn$.
\end{prop}

\begin{proof}
    First recall that by the thick subcategory theorem, the kernel $\cK_n$ of the localization $\Perf(\Sph) \to \PerfLnmf$ is the thick subcategory generated from any finite spectrum of type $n$.
    This participates in an exact sequence
    \[
        \cK_n \too \Perf(\Sph) \too \PerfLnmf.
    \]
    By \cite[Lemma 3.12]{land2019k} and \cref{CatL-closure}, tensoring the above exact sequence with $\PerfLnf$ gives the exact sequence
    \[
        \cK_n \otimes \PerfLnf \too \PerfLnf \too \PerfLnmf.
    \]
    We thus see that $\Mnf(\PerfLnf)$ is generated by the $\Lnf$-localization of any finite spectrum of type $n$, that is, any telescope $\Tn$.
\end{proof}

The monochromatization of a general category reduces to the above basic case using the following:

\begin{prop}\label{mnf-exact}
    Let $\cC \in \CatLnf$.
    The sequence
    \[
        \Mnf(\cC) \too \cC \too \Lnmf (\cC)
    \]
    is exact in $\Catperf$, and
    \[
        \Mnf(\cC) \simeq \Mnf(\PerfLnf) \otimes \cC.
    \]
\end{prop}

\begin{proof}
    As mentioned above, the sequence
    \[
        \Mnf(\PerfLnf) \too \PerfLnf \too \PerfLnmf
    \]
    is exact.
    By \cite[Lemma 3.12]{land2019k} and \cref{CatL-closure} again, tensoring it with $\cC$ gives an exact sequence
    \[
        \Mnf(\PerfLnf) \otimes \cC \too \cC \too \Lnmf(\cC),
    \]
    which also exhibits the first term as $\Mnf(\cC)$.
\end{proof}

The notation $\Mnf(\cC)$ is justified by the following:

\begin{prop}\label{monochrom}
    Let $\cC \in \CatLnf$.
    The category $\Mnf(\cC)$ is $n$-monochromatic, and the canonical map $\Mnf(\cC) \to \cC$ exhibits it as the coreflection of $\cC$ into $\CatMnf$.
\end{prop}

\begin{proof}
    For the first part, observe that $\Lnmf(\cC)$ is in particular $\Lnf$-local and hence $\Mnf(\cC)$, being the fiber of $\Lnf$-local categories, is $\Lnf$-local.
    By \cref{mnf-exact}, the sequence
    \[
        \Mnf(\cC) \too \cC \too \Lnmf(\cC)
    \]
    is exact.
    Thus, by \cite[Lemma 3.12]{land2019k} and \cref{CatL-closure} once more, tensoring it with $\PerfLnmf$ gives an exact sequence
    \[
        \Lnmf(\Mnf(\cC)) \too \Lnmf(\cC) \iso \Lnmf(\cC),
    \]
    and in particular $\Lnmf(\Mnf(\cC)) = 0$.
    That is, $\Mnf(\cC)$ is $n$-monochromatic.

    We now prove the second part.
    Let $\cD$ be any $n$-monochromatic category.
    Applying $\Map(\cD, -)$ to the exact sequence
    \[
        \Mnf(\cC) \too \cC \too \Lnmf(\cC)
    \]
    gives a fiber sequence
    \[
        \Map(\cD, \Mnf(\cC)) \too 
        \Map(\cD, \cC) \too 
        \Map(\cD, \Lnmf(\cC)).
    \]
    Since $\Lnmf(\cD) = 0$, we get 
    \[
        \Map(\cD, \Lnmf(\cC)) \simeq 
        \Map(\Lnmf(\cD), \Lnmf(\cC)) \simeq
        \Map(0, \Lnmf(\cC)) \simeq \pt.
    \]
    Thus, the map 
    $\Map(\cD, \Mnf(\cC)) \to \Map(\cD, \cC)$
    is an isomorphism, exhibiting $\Mnf(\cC)$ as the coreflection of $\cC$ into $\CatMnf$.
\end{proof}

\begin{cor}\label{Mnf-lim-colim}
     The full subcategory $\CatMnf \subseteq \CatLnf$ is closed under colimits and the coreflection $\Mnf \colon \CatLnf \to \CatMnf$ preserves all (small) limits and colimits.
\end{cor}

\begin{proof}
    By \Cref{monochrom}, the functor $\Mnf$ is the right adjoint to the inclusion $\CatMnf \into \CatLnf$. 
    Hence, the inclusion preserves colimits and $\Mnf$ preserves limits. 
    By \Cref{mnf-exact}, the composition of the functor $\Mnf$ with the colimit preserving fully faithful inclusion $\CatMnf \into \CatLnf$ is given by tensoring with $\Mnf(\PerfLnf)$ and hence preserves colimits. It follows that $\Mnf$ preserves colimits as well.
\end{proof}

\begin{prop} \label{equiv_char_monochrom}
For $\cC\in \Catperf$, the following are equivalent:
\begin{enumerate}
\item $\cC$ belongs to $\CatMnf$. 
\item $F(n+1)\otimes X = 0$ for every $X\in \cC$, and $\cC$ is generated as a stable idempotent complete category from objects of the form $F(n)\otimes X$ for $X\in \cC$.  
\item All the mapping spectra between objects of $\cC$ are $n$-monochromatic. 
\end{enumerate}
\end{prop}

\begin{proof} 

\underline{$(1)\implies (2)$}: For the first condition, if $\cC$ is $n$-monochromatic then it is in particular $\Lnf$-local, and hence tensored over $\PerfLnf$. Since 
$\Lnf(F(n+1)) = 0$, this implies that $F(n+1)\otimes X = 0$ for all $X\in \cC$. For the second condition,
by \Cref{monochrom} and \Cref{mnf-exact} we have 
\[
\cC\simeq \Mnf(\cC) \simeq \Mnf(\PerfLnf)\otimes \cC.
\]
Consequently, $\cC$ is generated from tensor products of the form $M\otimes X$ for $M\in \Mnf(\PerfLnf)$ and $X\in \cC$. Since by \Cref{ex:monochromatics_in_spectra}, $\Mnf(\PerfLnf)$ is the thick subcategory of $\PerfLnf$ generated by $\Tn \simeq \Lnf(F(n))$, we deduce that $\cC$ is generated by the objects of the form $F(n)\otimes X$ as claimed.   

\underline{$(2)\implies (3)$}: We wish to show that for every $X,Y\in 
\cC$, the mapping spectrum $\hom(X,Y)$ is $n$-monochromatic. First, since $Y\otimes F(n+1) = 0$ and $F(n+1)$ is finite, we see that
\[
\hom(X,Y)\otimes F(n+1) \simeq \hom(X,Y\otimes F(n+1)) \simeq \hom(X,0) = 0. 
\]
This implies that $\hom(X,Y)$ is $\Lnf$-local.
It remains to show that it is also $\Lnmf$-acyclic.

For fixed $X\in \cC$, the objects $Y\in \cC$ for which $\hom(X,Y)$ is $\Lnmf$-acyclic is thick.
Since by assumption $\cC$ is generated from objects of the from $Y\otimes F(n)$, it suffices to show that $\hom(X,Y\otimes F(n))$ is $\Lnmf$-acyclic.
Since $F(n)$ is finite, we get 
\[
\hom(X,Y\otimes F(n))\simeq \hom(X,Y)\otimes F(n),
\]
which is $\Lnmf$-acyclic since $F(n)$ is. 

\underline{$(3)\implies (1)$}: First, since the mapping spectra of $\cC$ are in $\Mnf\Sp\subseteq \Lnf\Sp$, the category $\cC$ is $\Lnf$-local. It remains to show that it is also $\Lnmf$-acyclic, namely, that $\Lnmf(\cC) = 0$. 
Since $\Lnmf(\cC)$ is generated by the image of the localization functor $\Lnmf\colon \cC \to \Lnmf(\cC)$, it would suffice to show that $\Lnmf(X) = 0$ for every $X\in \cC$. 
For every such $X$, we have a fully faithful exact embedding $\Perf(\End(X))\to \cC$ sending $\End(X)$ to $X$, so by the commutativity of the square 
\[
\xymatrix{
\Perf(\End(X))\ar[r]\ar[d] & \cC\ar[d] \\
\Lnmf(\Perf(\End(X)))\ar[r] & \Lnmf(\cC) \\
}
\]
it would suffice to show that $\Lnmf(\Perf(\End(X))) =0$. Finally, by our assumption on $\cC$, the ring spectrum $\End(X)$ is $n$-monochromatic. Thus, using  \Cref{CatL-closure} we get that 
\begin{align*}
    \Lnmf(\Perf(\End(X))) 
    &\simeq \PerfLnmf\otimes \Perf(\End(X))\\
    &\simeq \Perf(\Lnmf\Sph \otimes \End(X))\\
    &\simeq \Perf(\Lnmf(\End(X)))\\
    &= 0.
\end{align*}
\end{proof}

\subsection{Monochromatization and Purity}

We now study the connection between categorical monochromatization and chromatically localized algebraic $K$-theory, starting with some recollections of algebraic $K$-theory.

\begin{defn}
    We let $\mdef{\KTnp}\colon \Catperf \to \SpTnp$ denote the composition
    \[
        \Catperf \too[\ \ K\ \ ] \Sp \too[\LTnp] \SpTnp,
    \]
    where $K$ is the algebraic $K$-theory functor.
\end{defn}

\begin{rem}
    In the definition of $\KTnp$, it does not matter whether one works with connective or non-connective K-theory, since $\Tnp$-localization vanishes on bounded above spectra.
\end{rem}

We remind the reader that the functor $K$ preserves filtered colimits, sends exact sequences in $\Catperf$ to exact sequences in $\Sp$, and admits a canonical lax symmetric monoidal structure (see for example \cite{blumberg2013universal,blumberg2014uniqueness}).
Since $\LTnp$ is symmetric monoidal, colimit preserving and exact, $\KTnp$ inherits these properties and structure.

Recall that for $R \in \Alg(\Sp)$, one defines
\[
    K(R) := K(\Perf(R)).
\]
Similarly, we denote
\[
    \KTnp(R) := \KTnp(\Perf(R)),
\]
which we consider as a functor $\Alg(\Sp) \to \SpTnp$.
Note that for $R \in \CAlg(\SpTn)$, by \cite[Proposition 4.15 and Theorem C]{clausen2020descent}, the inclusion $\Perf(R) \into \cMod_R^\dbl$ to dualizable $\Tn$-local modules induces an isomorphism
\[
    \KTnp(R) \iso \KTnp(\cMod_R^\dbl).
\]

\begin{rem}
    The argument works verbatim for $R \in \Alg(\SpTn)$ (i.e., in the non-commutative case) when one replaces the category $\cMod_R^\dbl$ of dualizable $\Tn$-local modules by the category $\cLMod_R^\ldbl$ of left dualizable left $\Tn$-local modules.    
\end{rem}

One of the key results of \cite{clausen2020descent} and \cite{land2020purity} is the ``purity theorem'', which implies that if $\cC$ is $\Lnmf$-local then $\KTnp(\cC) = 0$.
As an immediate consequence, we obtain the following categorical analogue:

\begin{prop}\label{K-monochrom-iso}
    Let $\cC \in \CatLnf$. Then the inclusion $\Mnf(\cC) \into \cC$ induces an isomorphism 
    \[
        \KTnp(\Mnf(\cC)) \iso \KTnp(\cC).
    \]
\end{prop}

\begin{proof}
    By \Cref{mnf-exact}, there is an exact sequence 
    \[
        \Mnf(\cC) \too \cC \too \Lnmf(\cC). 
    \]
    Since $\KTnp$ preserves exact sequences, we get an exact sequence
    \[
        \KTnp(\Mnf(\cC)) \too \KTnp(\cC) \too \KTnp(\Lnmf(\cC)).
    \]
    Finally, by \cite[Theorem C]{clausen2020descent} the right term $\KTnp(\Lnmf(\cC))$ vanishes and hence the left morphism becomes an isomorphism $\KTnp(\Mnf(\cC)) \iso \KTnp(\cC)$.
\end{proof}

In other words, the restriction of the functor $\KTnp\colon \Catperf \to \SpTnp$ to the full subcategory $\CatLnf \sseq \Catperf$ factors through the limit and colimit preserving  reflection $\Mnf\colon \CatLnf \to \CatMnf$.

\subsection{Higher Semiadditivity}

One advantage of the category $\CatMnf$ over $\CatLnf$ is that it is \textit{$\infty$-semiadditive}. To establish that, we shall compare $\CatMnf$ to the category of compactly generated $\Tn$-local categories, where we can exploit the natural symmetric monoidal structure.

\subsubsection{Compactly Generated $\Tn$-local Categories}

Recall from the discussion in the beginning of \cref{subsec-lcat} that $\SpTn \in \Prlst$ is an idempotent algebra classifying the property of having $\Tn$-local mapping spectra.
We also recall that in this situation, the category of modules over $\SpTn$ in $\Prlst$ has a symmetric monoidal structure with a different unit, but the same tensor product.

\begin{defn}
    Let $\mdef{\PrlTn} \subset \Prlst$ denote the full subcategory on those categories whose mapping spectra are $\Tn$-local, endowed with the symmetric monoidal structure of $\Mod_{\SpTn}(\Prlst)$.
    Similarly, we let $\mdef{\PrlTnw} := \Prlstw \cap \PrlTn$ denote the full subcategory of $\Prlstw$ on those categories whose mapping spectra are $\Tn$-local.
\end{defn}

We now recall the following facts about $n$-monochromatization of spectra.

\begin{prop}[{\cite[Theorem 3.3]{bousfield2001telescopic}}]\label{mnf-sptn}
    The composition
    \[
        \Mnf\Sp \too \Sp \too \SpTn
        \qin \Prl
    \]
    is an equivalence.
\end{prop}

\begin{prop}\label{CatMnf_PrlTnw}
    A category $\cC \in \CatLnf$ is $n$-monochromatic if and only if $\Ind(\cC)$ has $\Tn$-local mapping spectra.
    Namely, the equivalence of \cref{ind-w} restricts to an equivalence
    \[
        \Ind\colon \CatMnf
        \adj \PrlTnw
        \noloc (-)^\omega.
    \]
\end{prop}

\begin{proof}
    For any $\cC \in \CatLnf$, by \cref{mnf-exact}, we have an exact sequence
    \[
        \Mnf(\cC) \too \cC \too \Lnmf (\cC).
    \]
    Taking $\Ind$, by the $\infty$-categorical analogue of Thomason--Neeman localization theorem \cite[A.3.11. Theorem]{hermitian-2}, we get an exact sequence (of not necessarily small categories)
    \[
        \Ind(\Mnf(\cC)) \too \Ind(\cC) \too \Ind(\Lnmf (\cC)).
    \]

    Note that in particular for $\cC = \PerfLnf$, the fact that this is a fiber sequence implies that $\Ind(\Mnf(\PerfLnf)) \simeq \Mnf\Sp$.

    Going back to the general case, we see that $\cC$ is $n$-monochromatic if and only if $\Ind(\Mnf(\cC)) \to \Ind(\cC)$ is an equivalence.
    By \cref{monochrom} we have
    \[
        \Mnf(\cC) \simeq \Mnf(\PerfLnf) \otimes \cC.
    \]
    Using the fact that $\Ind$ is symmetric monoidal and the result for $\PerfLnf$, we see that $\cC$ is $n$-monochromatic if and only if
    \[
        \Mnf\Sp \otimes \Ind(\cC) \too \Ind(\cC)
    \]
    is an equivalence.
    By tensoring the equivalence from \cref{mnf-sptn} with $\Ind(\cC)$, we get that the composition
    \[
        \Mnf\Sp \otimes \Ind(\cC) \too \Ind(\cC) \too \SpTn \otimes \Ind(\cC)
        \qin \Prl
    \]
    is an equivalence.
    Thus, by $2$-out-of-$3$, we see that $\cC$ is $n$-monochromatic, if and only if $\Ind(\cC) \to \SpTn \otimes \Ind(\cC)$ is an equivalence.
    Since $\SpTn \in \Prl$ is the idempotent algebra classifying the property of having $\Tn$-local mapping spectra (see \cite[Proposition 5.2.10]{AmbiHeight}), we obtain the result.
\end{proof}

\begin{prop}\label{prltnw-prltn}
    $\PrlTnw \subset \PrlTn$ is a symmetric monoidal subcategory closed under all colimits, and in particular its symmetric monoidal structure distributes over all colimits.
\end{prop}

\begin{proof}
    To show that $\PrlTnw \subset \PrlTn$ is a symmetric monoidal subcategory, it suffices to check that it contains the unit, and that morphisms are closed under the tensor product.
    Note that $\Tn$ is a compact generator of $\SpTn$ because $\Tn \simeq \LTn F(n)$ where $F(n)$ is a finite type $n$ spectrum, which by the thick subcategory theorem is a compact generator of $\Lnmf$-acyclic spectra of which $\SpTn$ is a smashing localization.
    Thus, $\SpTn \in \PrlTnw$.
    For closure of morphisms under tensor product, recall that $\Prlstw$ and $\PrlTn$ have the same tensor product as $\Prlst$, and $\PrlTnw$ is their intersection.

    To see that $\PrlTnw \subset \PrlTn$ is closed under colimits, consider the following commutative diagram:
    \[\begin{tikzcd}
    	\Catperf & \Prlstw & \Prlst \\
    	\CatMnf & \PrlTnw & \PrlTn
    	\arrow["\sim", from=1-1, to=1-2]
    	\arrow["\sim", from=2-1, to=2-2]
    	\arrow[hook, from=2-1, to=1-1]
    	\arrow[hook, from=2-2, to=1-2]
    	\arrow[hook, from=2-2, to=2-3]
    	\arrow[hook, from=1-2, to=1-3]
    	\arrow[hook, from=2-3, to=1-3]
    \end{tikzcd}\]
    Both horizontal morphisms in the left square are equivalences by \cref{ind-w} and \cref{CatMnf_PrlTnw}.
    Colimits in $\PrlTn$ are computed as in $\Prlst$ because it is a smashing localization, and the same holds for $\Prlstw$ by \cref{ind-w}.
    Thus, by the commutativity of the right square, it suffices to check that $\PrlTnw \subset \Prlstw$ is closed under colimits.
    This follows from the commutativity of the left square and \cref{Mnf-lim-colim}.
\end{proof}

\begin{thm}\label{CatMnf-inf-sa}
    For all $n\ge 0$, the category $\CatMnf$ is $\infty$-semiadditive.
\end{thm}

\begin{proof}
    By \Cref{CatMnf_PrlTnw}, we have $\CatMnf \simeq \PrlTnw$, so it suffices to prove that $\PrlTnw$ is $\infty$-semiadditive.
    We shall prove that $\PrlTnw$ is $m$-semiadditive by induction on $m$, starting with $m=-2$ where the condition is vacuous. 
    Assume $\PrlTnw$ is $(m-1)$-semiadditive.
    By \cref{prltnw-prltn}, it is symmetric monoidal and the tensor product distributes over colimits, so by \cite[Proposition 2.3.4]{TeleAmbi}, for every $m$-finite space $A$ we have a canonical ambidexterity pairing
    \[
        \varepsilon\colon 
        \SpTn[A] \otimes \SpTn[A] \too 
        \SpTn \qin \PrlTnw.
    \]
    Furthermore, $A$ is $\PrlTnw$-ambidextrous if and only if $\varepsilon$ is non-degenerate in the sense that it exhibits $\SpTn[A]$ as self-dual. 
    Now, again by \cref{prltnw-prltn}, the inclusion $\PrlTnw \into \PrlTn$ is colimit preserving and exhibits the source as a symmetric monoidal (non-full) subcategory of the target, which is $\infty$-semiadditive by \cite[Example 4.3.11]{AmbiKn} and \cite[Remark 5.3.2]{AmbiHeight}.
    In particular, $\varepsilon$ is also the canonical ambidexterity pairing in $\PrlTn$, and there exists 
    \[
        \eta \colon 
        \SpTn \too 
        \SpTn[A] \otimes \SpTn[A] \qin \PrlTn,
    \]
    which satisfies together with $\varepsilon$ the zigzag identities.
    It thus suffices to show that $\eta$ belongs to $\PrlTnw$.
    Namely, that its right adjoint itself admits a further right adjoint. 

    The ambidexterity pairing of any algebra in $\PrlTn$ is the same as in $\Prl$ itself, which under the identification
    \[
        \SpTn[A]\otimes \SpTn[A] \simeq 
        \SpTn[A \times A] \simeq 
        \SpTn^{A \times A}
    \]
    is given by (e.g., see  \cite[Proposition 3.17]{harpaz2020ambidexterity})
    \[
        \varepsilon \colon 
        \SpTn^{A \times A} \too[\Delta^*]
        \SpTn^A \too[A_!]
        \SpTn
    \]
    and 
    \[
        \eta \colon 
        \SpTn \too[A^*]
        \SpTn^A \too[\Delta_!]
        \SpTn^{A \times A},
    \]
    where $\Delta \colon A \to A \times A$ is the diagonal. Since $A$ is $\SpTn$-ambidextrous, the right adjoint $A_*$ of $A^*$ is isomorphic to $A_!$ and hence admits a further right adjoint (note that by induction, we already know this for $\Delta^*$, which is $(m-1)$-finite). It follows that $\eta$ belongs to $\PrlTnw$, which completes the inductive step and hence the proof. 
\end{proof}

\begin{rem}
    As a consequence of the arguments appearing above, $\pi$-finite limits in   $\PrlTnw$ are computed as in $\CAT_\infty$. By the equivalence of $\PrlTnw$ with $\CatMnf$, we obtain an explicit formula for computing $\pi$-finite limits in $\CatMnf$. Namely, for a $\pi$-finite space $A$ and an $A$-local system $\{\cC_a\}_{a \in A}$ in $\CatMnf$, the limit over $A$ can be computed as 
    \[
        \invlim_{a \in A} \cC_a 
        \simeq \big(\invlim_{a \in A} \Ind (\cC_a)\big)^\omega.
    \]
    However, the forgetful functor $\CatMnf \to \Cat_\infty$ does \textit{not} preserve $\pi$-finite limits. Indeed, in light of the above formula, if the forgetful functor was $\pi$-finite limit preserving, the limit of an $A$-shaped diagram of compact objects in a compactly generated $T(n)$-local $\infty$-category $\cC$, would have been itself compact. 
    A counterexample to this was communicated to us by Robert Burklund. Consider $\cC = \cMod_{E_1}$ and $A = BC_p$.
    Starting with $E_1$ with the trivial $C_p$-action, by the chromatic Fourier transform, its $C_p$-equivariant self maps are given by
    \[
        \pi_0\Map_{C_p}(E_1 , E_1) \simeq 
        \pi_0\Map(E_1,E_1^{BC_p}) \simeq 
        \pi_0(E_1^{BC_p}) \simeq
        \pi_0(E_1[C_p]) \simeq
        \ZZ_p[x]/(x^p - 1).
    \]
    We consider the $C_p$-equivariant self map $f\colon E_1 \to E_1$ corresponding to $1+x+\dots+x^{p-1}$ and take $X$ to be its cofiber in $C_p$-equivariant $K(1)$-local $E_1$-modules. On the one hand, as $x$ is mapped to the identity under the forgetful map
    \[
        \Map_{C_p}(E_1 , E_1) \too \Map(E_1 , E_1),
    \]
    as a non-equivariant $E_1$-module we have $X \simeq E_1/p$, which is compact. On the other hand, the homotopy fixed points module $X^{hC_p}$ is the cofiber of the map $E_1[C_p] \to E_1[C_p]$ given by multiplication by $1+x+\dots+x^{p-1}$, which is $E_1^{p-1} \oplus \Sigma E_1$. Since it is not a $p$-torsion object it is not compact.
    
\end{rem}

\subsubsection{Consequences of Higher Semiadditivity}

To illustrate the usefulness of \Cref{CatMnf-inf-sa}, we shall record the following general reductions for proving preservation of $\pi$-finite colimits for a functor between higher semiadditive categories, which will be used in the proof of \Cref{main-thm-intro}.

\begin{prop}\label{red-const}
    Let $F \colon \cC \to \cD$ be a functor between $p$-typically $m$-semiadditive categories.
    If $F$ preserves constant $m$-finite $p$-space indexed (co)limits, then it preserves all $m$-finite $p$-space indexed (co)limits (i.e., it is $p$-typically $m$-semiadditive).
\end{prop}

\begin{proof}
    For $m = -2$ there is nothing to prove. Assume by induction that $F$ is $p$-typically $(m-1)$-semiadditive. Hence, by \cite[Theorem 3.2.3]{TeleAmbi}, for every $m$-finite $p$-space $A$ we have a commutative norm diagram
    \[\begin{tikzcd}
    	{FA_!} && {FA_*} \\
    	{A_!F} && {A_*F},
    	\arrow["\beta_!", from=2-1, to=1-1]
    	\arrow["\beta_*", from=1-3, to=2-3]
    	\arrow["{\Nm^{\cD}}", from=2-1, to=2-3]
    	\arrow["{\Nm^{\cC}}", from=1-1, to=1-3]
    	\arrow["\sim"', from=2-1, to=2-3]
    	\arrow["\sim"', from=1-1, to=1-3]
    \end{tikzcd}\]
    in which the horizontal maps are isomorphisms because $\cC$ and $\cD$ are $p$-typically $m$-semiadditive. It follows that $\beta_!$ admits a \emph{left homotopy inverse}. On the other hand, using the wrong way adjunction $A^* \dashv A_!$ in $\cC$, we have the following diagram:
    \[\begin{tikzcd}
    	&& {A_!F A^*A_!} && {A_!F} \\
    	{FA_!} && {FA_!A^*A_!} && {FA_!.}
    	\arrow["{\mu_!}", from=2-1, to=2-3]
    	\arrow["{\nu_!}", from=2-3, to=2-5]
    	\arrow["{\beta_!}", from=1-5, to=2-5]
    	\arrow["{\beta_!}", from=1-3, to=2-3]
    	\arrow["\wr"', from=1-3, to=2-3]
    	\arrow["{\nu_!}", from=1-3, to=1-5]
    	\arrow[curve={height=18pt}, Rightarrow, no head, from=2-1, to=2-5]
    \end{tikzcd}\]
    Hence, $\beta_!$ admits a \emph{section}. Therefore, $\beta_!$ and $\beta_*$ are isomorphisms. 
\end{proof}

\begin{prop}\label{red-concentrated}
    Let $\cC, \cD$ be categories admitting $m$-finite $p$-space indexed colimits, and let $F\colon \cC \to \cD$ be a functor commuting with $(m-1)$-finite $p$-space indexed colimits.
    Then, $F$ commutes with (constant) $m$-finite $p$-space indexed colimits if and only if it commutes with (constant) colimits indexed by $m$-finite $p$-spaces concentrated in homotopy degree $m$.
\end{prop}

\begin{proof}
    The only if part is clear.
    For the other direction, let $A$ be some $m$-finite $p$-space, and we shall show that $F$ commutes with $A_!$.
    Let $B := A_{\leq m-1}$, and $f\colon A \to B$ be the canonical map.
    Since $A_! \simeq B_! f_!$, it suffices to check that $F$ commutes with both of them.
    By the inductive hypothesis, $F$ indeed commutes with $B_!$, thus it remains to show that it commutes with $f_!$.
    This can be checked fiber-wise on the target, namely, for any $b \in B$, we need to show that $b^* f_! F \to b^* F f_! \simeq F b^* f_!$ is an isomorphism.
    Consider the pullback:
    \[\begin{tikzcd}
    	{A_b} & \pt \\
    	A & B
    	\arrow["f"', from=2-1, to=2-2]
    	\arrow["b", from=1-2, to=2-2]
    	\arrow["{i_b}"', from=1-1, to=2-1]
    	\arrow[from=1-1, to=1-2]
    \end{tikzcd}\]
    The Beck--Chevalley condition guarantees that $b^* f_! \simeq A_{b!} i_b^*$.
    Under this isomorphism, the condition we need to check is that $A_{b!} i_b^* F \to F A_{b!} i_b^*$ is an isomorphism.
    Since $F$ commutes with $i_b^*$, it suffices to show that it commutes with $A_{b!}$, which indeed holds since by construction $A_b$ is $m$-finite and concentrated in homotopy degree $m$, so this follows by the hypothesis.
    
    The case of constant colimits follows by pre-composition with $A^*$.
\end{proof}

\section{Higher Descent}

\subsection{Higher Categorical Descent}

We are finally in position to prove \Cref{main-thm-intro}, namely, that $\KTnp$ preserves $\pi$-finite $p$-space (co)limits of $\Lnf$-local categories. This is an immediate corollary of the following:

\begin{thm}\label{main-thm}
    The functor
    \[
        \KTnp\colon \CatMnf \too \SpTnp
    \]
    is $p$-typically $\infty$-semiadditive. That is, it preserves $\pi$-finite $p$-space indexed limits and colimits.
\end{thm}

\begin{proof}
First,
by \Cref{CatMnf-inf-sa} and \cite[Theorem A]{TeleAmbi} the categories $\CatMnf$ and $\SpTnp$ are both $\infty$-semiadditive. Therefore, by (the $p$-typical analogue of) \cite[Corollary 3.2.4]{AmbiHeight} it suffices to show that $\KTnp$ preserves $\pi$-finite $p$-space indexed colimits. By \Cref{red-const} and \Cref{red-concentrated} it is enough to consider constant colimits indexed by $\pi$-finite $p$-spaces concentrated in a fixed homotopical degree $m$. Namely, we have to show that for such a space $A$ the assembly map 
\[
\KTnp(\cC)[A]\too \KTnp(\cC[A])
\]
is an isomorphism.

For the case $m = 0$ observe that $\KTnp$ is exact and in particular commutes with finite coproducts.
For $m=1$, since $\cC$ is in $\CatMnf$, and in particular it is $\Lnf\Sph$-linear, the result follows from \cite[Theorem 4.12(1)]{clausen2020descent} applied to $R=\Lnf\Sph$. We proceed by induction on $m\ge 2$. 
    
   First, we reduce to the case where $\cC$ is of the form $\Perf(R)$ for a ring spectrum $R\in \Mnf\Sp$.
    Indeed, by \cref{Schwede-Shipley} we may write $\cC$ as a filtered colimit of the form 
    \[
    \cC \simeq \colim\Perf(R_i) \qin \Catperf,
    \]
    where the ring spectra $R_i$ are all endomorphism rings of objects of $\cC$. 
    Since $\cC \in \CatMnf$, we know by \Cref{equiv_char_monochrom}(3) that each $R_i$ belongs to $\Mnf\Sp$. Since each of the categories $\Perf(R_i)$ is a full subcategory of $\cC$, by \Cref{equiv_char_monochrom}(3) again we deduce that they all belong to $\CatMnf$. 
    Also, recall that
    \[
        \KTnp\colon \Catperf \too \SpTnp
    \]
    preserves filtered colimits.
    Thus, by naturality of the assembly map, to show that
    \[
        \KTnp(\cC)[A] \too 
        \KTnp(\cC[A])
        \qin \SpTnp
    \]
    is an isomorphism for every $\cC \in \CatMnf$ it suffices to show that
    \[
        \KTnp(\Perf(R))[A] \too
        \KTnp(\Perf(R)[A])
        \qin \SpTnp
    \]
    is an isomorphism for every $n$-monochromatic ring spectrum $R$. 
    
    Note that once $m \geq 1$ and $A$ is concentrated in homotopy degree $m$, it is connected, whence by \cref{perf-R-A} this map assumes the form
    \[
        \KTnp(R)[A] \too
        \KTnp(R[\Omega A])
        \qin \SpTnp.
    \]
    Moreover, when $m\ge 2$ the space $\Omega A$ is also connected. 
    
    First we deal with the case where $R$ is of height $n = 0$.
    In this case, since $\Omega A$ is a connected $\pi$-finite space we have $R[\Omega A] \simeq R$.
    Additionally, by \cite[Theorem D]{AmbiHeight}, the constant colimits over $A$ in $\Sp_{\To}$ do not change the object, i.e.\ $\colim_A X \simeq X$ naturally in $X\in \Sp_{\To}$ (see also the discussion under \cite[Definition 2.4.1]{AmbiHeight}).
    Hence, our assembly map identifies with the identity map of $K_{T(1)}(R)$ and in particular it is an isomorphism.
    
    From now on we assume that $n \geq 1$.
    Using the bar construction, we can write $A \simeq \colim_{\Delta^\op} A_k$ where $A_k := (\Omega A)^k$.
    Consider the commutative diagram:
    \[\begin{tikzcd}
    	{\colim_{\Delta^\op} K_{T(n+1)}(R)[A_k]} & {\colim_{\Delta^\op} K_{T(n+1)}(R[\Omega A_k])} \\
    	& {K_{T(n+1)}(\colim_{\Delta^\op} R[\Omega A_k])} \\
    	{K_{T(n+1)}(R)[\colim_{\Delta^\op} A_k]} & {K_{T(n+1)}(R[\Omega \colim_{\Delta^\op} A_k])}
    	\arrow[from=1-1, to=1-2]
    	\arrow[from=1-1, to=3-1]
    	\arrow[from=1-2, to=2-2]
    	\arrow[from=3-1, to=3-2]
    	\arrow[from=2-2, to=3-2]
    \end{tikzcd}\]
    The functor $\Spc_* \to \Spc$ which forgets the pointing preserves sifted colimits by \cite[Proposition 4.4.2.9]{HA}, and the inclusion $\Spc^{\geq 1} \into \Spc$ also preserves sifted colimits as explained in the proof of \cite[Proposition 1.4.3.9]{HA}, thus so does the functor $\Spc^{\geq 1}_* \to \Spc$.
    Therefore, there is no ambiguity as to where the colimit in the two bottom objects are computed.
    
    Our goal is to show that the bottom map is an isomorphism.
    We show that all other morphisms are isomorphisms, which implies the result by the commutativity of the diagram.
    The top map is an isomorphism by the inductive hypothesis, as $A_k$ is an $(m-1)$-finite $p$-space.
    The left map is an isomorphism because the functor $\KTnp(R)[-]$ preserves all colimits.
    Since $n \geq 1$, by \cite[Corollary 4.31]{land2020purity}, the functor
    \[
        \KTnp\colon \Alg(\Sp) \too \SpTnp
    \]
    preserves sifted colimit, showing that the upper right morphism is an isomorphism.
    Finally, for the bottom right morphism, note that the forgetful functor $\Alg(\Sp) \to \Sp$ and the functor $\Omega(-)$ both preserve sifted colimits (see \cite[Proposition 3.2.3.1]{HA} and \cite[Corollary 5.2.6.18]{HA}), and $R[-]$ commutes with all colimits, so that $R[\Omega(-)]$ preserves sifted colimits.
\end{proof}

\begin{cor}\label{main-thm-cor}
    The functor 
    \[
        \KTnp\colon \CatLnf \too \SpTnp
    \] 
    preserves $\pi$-finite $p$-space indexed limits and colimits. 
\end{cor}

\begin{proof}
By \Cref{K-monochrom-iso} we can present the functor $\KTnp\colon \CatLnf\to \SpTnp$ as the composition 
\[
    \CatLnf \too[\Mnf] \CatMnf \too[\KTnp] \SpTnp. 
\]
Now, the first functor preserves all limits and colimits by \Cref{Mnf-lim-colim} and the second preserves all $\pi$-finite $p$-space indexed limits and colimits by \Cref{main-thm}. Hence, the composition preserves these limits and colimits as well. 
\end{proof}

\begin{rem}
    By \Cref{Char_CatL}, an equivalent formulation of \Cref{main-thm-cor} is that the functor
    \[
        \KTnp((-)^\omega) \colon 
        \Prl_{\Lnf,\omega} \too \SpTn
    \]
    preserves $\pi$-finite $p$-space indexed limits and colimits. 
\end{rem}

\begin{cor}\label{KTnp-R-A-Lnf}
    There is a map
    \[
        \KTnp(R)[A] \too \KTnp(R[\Omega A])
        \qin \SpTnp,
    \]
    lax symmetric monoidally natural in $A \in \Spc^{\geq 1}_*$ and $R \in \Alg(\Lnf\Sp)$.
    Furthermore, when $A$ is a sifted colimit of pointed connected $\pi$-finite $p$-spaces, it is an isomorphism.
\end{cor}

\begin{proof}
    As explained in the proof of \cref{main-thm}, the map is the assembly map of the lax symmetric monoidal functor $\KTnp$, together with the equivalence
    \[
        \Perf(R[\Omega A]) \simeq \Perf(R)[A]
    \]
    from \cref{perf-R-A}.
    By \cref{main-thm}, it is an isomorphism for $\pi$-finite $p$-spaces.
    Since $\KTnp$ preserves sifted colimits of ring spectra by \cite[Corollary 4.31]{land2020purity}, and $R[\Omega(-)]$ preserves sifted colimits of pointed connected spaces, we get that the assembly map is an isomorphism for sifted colimits of pointed connected $\pi$-finite $p$-spaces as well.
\end{proof}

\begin{rem}\label{KTnp-cMod}
    Combining \cref{Mod-R-dbl-F}, \cite[Proposition 4.15]{clausen2020descent}, and the fact that for modules in spectra perfect objects and dualizable objects coincide, for $R \in \CAlg(\SpTn)$ and $M \in \Spcn$, the assembly map
    \[
        \KTnp(R)[\Sigma M] \too \KTnp(R[M])
    \]
    from \cref{KTnp-R-A-Lnf} is equivalent to the composition
    \begin{align*}
        \KTnp(\cMod_R^\dbl)[\Sigma M]
        &\too \KTnp(\cMod_R^\dbl[\Sigma M])\\
        &\too \KTnp(\cMod_R[\Sigma M]^\dbl)\\
        &\iso \KTnp(\cMod_{\LTn R[M]}^\dbl).
    \end{align*}
    Thus, when $M$ is a filtered colimit of $\pi$-finite $p$-spectra, all three maps are isomorphisms.
    Indeed, the first morphism is an isomorphism since $\KTnp$ preserves both filtered colimits and $\pi$-finite $p$-space shaped colimits by \cref{main-thm-cor}.
    Moreover, by \cref{KTnp-R-A-Lnf}, the composition is an isomorphism, so that the second morphism is an isomorphism as well.
\end{rem}

\begin{cor}\label{KTnp-CAlg-lim}
    The functor
    \[
        \KTnp\colon \Alg(\SpTn) \too \Alg(\SpTnp)
    \]
    preserves $n$-finite $p$-space indexed limits.
    In particular, for every $R \in \Alg(\SpTn)$ and an $n$-finite $p$-space $A$ there is an isomorphism
    \[
        \KTnp(R^A) \iso \KTnp(R)^A.
    \]
\end{cor}

\begin{proof}
    Let $R\colon A \to \Alg(\SpTn)$ be a diagram indexed by some space $A$.
    The assembly map is the following composition
    \begin{align*}
        \KTnp(\invlim_A R)
        &\iso \KTnp(\cLMod_{\invlim_A R}^\ldbl)\\
        &\too[(1)] \KTnp((\invlim_A \cLMod_R)^\ldbl)\\
        &\iso \KTnp(\invlim_A \cLMod_R^\ldbl)\\
        &\too[(2)] \invlim_A \KTnp(\cLMod_R^\ldbl)\\
        &\iso \invlim_A \KTnp(R).
    \end{align*}
    By \cite[Theorem 7.29]{barthel2022chromatic}, the space $A$ is $\SpTn$-affine, namely, the map denoted by $(1)$ is an isomorphism.
    By \cref{main-thm}, the map denoted by $(2)$ is an isomorphism.
\end{proof}

\cref{main-thm} asserts that a certain version of $\Tnp$-localized $K$-theory is $\infty$-semiadditive.
In \cite{moshe2021higher}, the universal $p$-typically $m$-semiadditive version of algebraic $K$-theory was studied.
This functor has a $\Tnp$-localized version, which assigns to any stable category admitting $m$-finite $p$-space indexed colimits $\cC$ a $\Tnp$-local spectrum $\KmTnp(\cC)$.

\begin{cor}\label{sa-K}
    Let $\cC \in \CatLnf$ be an $\Lnf$-local category admitting $m$-finite $p$-space indexed colimits, then
    \[
        \KmTnp(\cC) \simeq \KTnp(\cC)
        \qin \SpTnp.
    \]
    In particular, for any $R \in \Alg(\SpTn)$, we have
    \[
        \KmTnp(R) \simeq \KTnp(R)
        \qin \SpTnp.
    \]
\end{cor}

\begin{proof}
    By \cite[Corollary 6.14]{moshe2021higher}, we need to check that the functor
    \[
        \Span(\Smfin) \too \SpTnp,
        \qquad A \mapsto \KTnp(\cC^A)
    \]
    satisfies the $m$-Segal condition.
    Namely, that
    \[
        \KTnp(\cC^A) \too \KTnp(\cC)^A
    \]
    is an isomorphism for every $m$-finite $p$-space $A$, which holds by \cref{main-thm}.
    
    For $R \in \Alg(\SpTn)$, recall from \cite[Remark 6.22]{moshe2021higher} that 
    \[
        \KmTnp(R) \simeq \KmTnp(\cLMod_R^\ldbl),
    \]
    and recall that 
    \[
        \KTnp(R) \simeq \KTnp(\cLMod_R^\ldbl),
    \] 
    so this is the special case $\cC = \cLMod_R^\ldbl$.
\end{proof}

\subsection{Higher Galois Descent}

We now turn to the implications of higher descent for $\KTnp$ to preservation of Galois extensions and in particular prove \Cref{galois-comp-intro} from the introduction (as \Cref{galois-comp}). 

Given a symmetric monoidal category $\cC$ and a weakly $\cC$-ambidextrous (pointed) connected space $BG \oto{q} \pt$ with diagonal denoted by $BG \oto{\Delta} BG \times BG$, a $G$-equivariant commutative algebra $R \colon BG \to \CAlg(\cC)$ is said to be \emph{Galois} if it satisfies the two conditions of \cite[Definition 4.1.3]{RognesGal} (see also \cite[Definition 2.25]{barthel2022chromatic} for the present formulation):
\begin{enumerate}
    \item The mate of the unit $\one \to q_*R =: R^{hG}$ is an isomorphism.
    \item The mate of the multiplication $R\otimes R \to \Delta_*R =: \coind{G}{R}$ is an isomorphism.
\end{enumerate}

Here we are using the notation $\coind{A}{X}$ to mean the constant $A$-shaped limit on $X$, where $A$ is an arbitrary space.

\begin{rem}
    In \cite{RognesGal}, Rognes develops the theory of $G$-Galois extensions under the assumption that $\one_\cC[G]$ is dualizable in $\cC$. 
    By \cite[Corollary 2.7]{carmeli2021chromatic},
    this condition is implied by the weak $\cC$-ambidexterity of $BG$ (i.e., the $\cC$-ambidexterity of $G$). 
    We also note that for $\cC$ semiadditive (e.g., stable), the space $BG$ is weakly $\cC$-ambidextrous for all finite discrete groups $G$.  
\end{rem}

By applying the functor
\[
    \Mod_{(-)}(\cC) \colon 
    \CAlg(\cC) \too 
    \CAlg_\cC(\Prl),
\]
to $R$, we get a $G$-equivariant structure on the $\cC$-linear symmetric monoidal category $\Mod_R(\cC)$ and an induced symmetric monoidal functor, which we denote by a slight abuse of notation by
\[
    R\otimes(-) \colon 
    \cC \too 
    \Mod_R(\cC)^{hG}.
\]
It is a classical fact, commonly referred to as \emph{Galois descent}, that for an ordinary finite group $G$ and an ordinary $G$-Galois extension $R$, this functor is an equivalence. We shall show that it is in fact always an equivalence provided $R$ is faithful. Similar results under somewhat different assumptions can be found in \cite{AkhilGalois,banerjee2017galois,gepner2021brauer}.
We begin by generalizing the standard maneuver of identifying the category of ``descent data for $G$'' $\Mod_R(\cC)^{hG}$  with that of ``$G$-semilinear $R$-modules'' $\Mod_R(\cC^{BG})$. This is a general fact unrelated to Galois extensions. 

\begin{prop}\label{Shay}
    Let $\cC \in \CAlg(\Prl)$ and let $R \in \CAlg(\cC)^A$ for some space $A$. There is a natural equivalence of $\cC$-linear symmetric monoidal categories
    \[
        \invlim_A \Mod_{a^*R}(\cC) \simeq \Mod_R(\cC^A).
    \]
\end{prop}
\begin{proof}
    First, we write $\cC^A$ as the limit of the constant $A$-shaped diagram on $\cC$ in $\CAlg(\Prl)$. The cone maps exhibit for each $a \in A$ the corresponding copy of $\cC$ in the diagram as a commutative $\cC^A$-algebra, via the symmetric monoidal functor $a^* \colon \cC^A \to \cC$. Consequently, we have $\cC^A = \invlim_A \cC$ in $\CAlg_{\cC^A}(\Prl)$.
    We claim that the tensor product in $\CAlg_{\cC^A}(\Prl)$ commutes with space-shaped limits.
    Indeed, since space-shaped limits and colimits in $\Mod_{\cC^A}(\Prl)$ coincide, they are preserved by the tensor product \cite[Example 4.26]{cnossen2022characters}, and the forgetful $\CAlg_{\cC^A}(\Prl) \to \Mod_{\cC^A}(\Prl)$ is symmetric monoidal and preserves limits, the same holds in $\CAlg_{\cC^A}(\Prl)$.
    We can thus compute,
    \[
        \Mod_R(\cC^A) \simeq 
        \Mod_R(\cC^A) \otimes_{\cC^A} \cC^A \simeq 
        \Mod_R(\cC^A) \otimes_{\cC^A} \invlim_A \cC \simeq
        \invlim_A \big(\Mod_R(\cC^A) \otimes_{\cC^A} \cC \big).
    \]
    Finally, when $\cC$ is viewed as a commutative $\cC^A$-algebra via the symmetric monoidal functor $a^* \colon \cC^A \to \cC$, we have by \cite[Theorem 4.8.5.16]{HA},
    \[
        \Mod_R(\cC^A) \otimes_{\cC^A} \cC \simeq 
        \Mod_{a^* R}(\cC)
    \]
    and therefore 
    \[
        \Mod_R(\cC^A) \simeq  \invlim_A \Mod_{a^* R}(\cC).
    \]
\end{proof}

In addition to the maps $BG \oto{q} \pt$ and $BG \oto{\Delta} BG \times BG$, it will be handy to have a name for the basepoint $\pt \oto{e} BG$ as well. However, for a $G$-equivariant object $X$ in $\cC^{BG}$, we shall denote the underlying non-equivariant object $e^*X$ in $\cC$ also by $\underline{X}$.
In the notation above for a $G$-Galois extension $R$, we have the following:
\begin{lem}\label{Galois_Coind}
    Let $\cC \in \CAlg(\Prl)$, let $BG$ be a weakly $\cC$-ambidextrous pointed connected space, and let $R\colon BG \to \CAlg(\cC)$ be a faithful $G$-Galois extension. 
    For every $X \in \Mod_R(\cC^{BG})$ we have a natural $\cC$-linear isomorphism
    \[
        \underline{R}\otimes X \simeq
        \coind{G}{\underline{X}} \qin \cC^{BG},
    \]
    Namely, $G$ acts on the left hand side on the $X$ coordinate alone, and on the right hand side simply 
    by permuting the $G$-factors (i.e.\ the co-induced object).
\end{lem}
\begin{proof}
    An $R$-module structure on $X$ provides a map 
    \[
        \Delta^*(R\boxtimes X) \simeq 
        R\otimes X 
        \too X
        \qin \cC^{BG},
    \]
    whose mate is a natural transformation in $X$,
    \[
        f_X \colon R\boxtimes X \too 
        \Delta_* X
        \qin \cC^{BG \times BG}.
    \]
    Forgetting the action of $G$ on the $R$-coordinate of $R\otimes X$ corresponds to applying the functor
    \[
        (e\times \Id)^* \colon
        \cC^{BG \times BG} \too 
        \cC^{BG}.
    \]
    Using the Beck--Chevalley condition for the pullback square of spaces
    \[\begin{tikzcd}
    	\pt & BG \\
    	BG & {BG\times BG},
    	\arrow["\Delta", from=2-1, to=2-2]
    	\arrow["e"', from=1-1, to=2-1]
    	\arrow["e", from=1-1, to=1-2]
    	\arrow["{e\times \Id}", from=1-2, to=2-2]
    \end{tikzcd}\]
    we get the map
    \[
        g_X \colon 
        \underline{R} \otimes X \iso
        (e\times \Id)^*(R\boxtimes X) \too[f_X]
        (e\times \Id)^*\Delta_* X \iso
        e_*e^* X =: \coind{G}{\underline{X}}
    \]
    of functors
    \[
        \Mod_R(\cC^{BG}) \too
        \cC^{BG}.
    \]
    It remains to show that $g_X$ is an isomorphism for all $X$.
    Since $e$ and $\Delta$ are $\cC$-ambidextrous, by \cite[Corollary 2.34]{barthel2022chromatic}, $e_*$ is $\cC^{BG}$-linear and colimit preserving, $\Delta_*$ is $\cC^{BG\times BG}$-linear, and colimit preserving and $g_X$ is a natural transformation of $\cC^{BG}$-linear colimit preserving functors. Consequently, it suffices to show that it is an isomorphism for $X = R$ for which it follows from the fact that $f_R$ is an isomorphism, as it is exactly the second Rognes--Galois condition.
\end{proof}

From this we deduce descent for Galois extensions. 

\begin{prop}\label{Galois_Descent}
    Let $\cC \in \CAlg(\Prl)$ and let $BG$ be a weakly $\cC$-ambidextrous pointed connected space. Every faithful $G$-Galois extension $R\colon BG \to \CAlg(\cC)$ induces a symmetric monoidal equivalence 
    \[
        \cC \iso \Mod_R(\cC)^{hG}.
    \]
\end{prop}

\begin{proof}
    Under the equivalence in \Cref{Shay}, the functor 
    \[
        R \otimes (-) \colon 
        \cC \too 
        \Mod_R(\cC)^{hG},
    \]
    corresponds to the composition of left adjoints
    \[
        \cC \too[\ q^*\ ]
        \cC^{BG} \too[R\otimes(-)]
        \Mod_R(\cC^{BG}).
    \]
    We shall show that the right adjoint of this composition is conservative and that the unit of the adjunction is an isomorphism.
    
    The unit map is given for every $X \in \cC$ by the composition
    \[
        u_X \colon 
        X \too
        q_*q^* X \too
        q_*(R \otimes q^* X) =: (R\otimes X)^{hG}.
    \]
    Since $R$ is faithful, it suffices to show that $u_X$ becomes an isomorphism after post-composition with the functor of tensoring with $e^*R$. We thus have, 
    \[
        e^*R \otimes q_*(R \otimes q^*X) \simeq
        q_*(q^*e^*R \otimes  (R \otimes q^*X)) \simeq
        q_*(e_*e^*(R \otimes q^*X)) \simeq
        e^* R \otimes X,
    \]
    where the first isomorphism follows from the fact that tensoring with the dualizable object $e^*R$ (Galois extensions are dualizable by \cite[Proposition 6.2.1]{RognesGal}) commutes with $q_*$, the second is \Cref{Galois_Coind} and the third follows from $qe = \Id$. 
    We get that $e^*R \otimes u$ is a natural $\cC$-linear endomorphism of the functor $e^*R \otimes -$. Thus, to show that it is an isomorphism, it suffices to show that it is an isomorphism on $X = \one$, but $u_\one$ is already an isomorphism by the first Rognes--Galois condition. 
    
    To show that the right adjoint is conservative, it again suffices to do so after post-composition with the functor of tensoring with $e^*R$.
    By the same argument as before we have for every $M\in \Mod_R(\cC^{BG})$ a natural isomorphism
    \[
        e^*R \otimes q_*M \simeq
        q_*(q^*e^*R \otimes M) \simeq
        q_*(e_*e^*M) \simeq
        e^* M,
    \]
    and $e^*$ is conservative. 
\end{proof}
This 
Galois descent result has the following consequence for the functoriality of Galois extensions under symmetric monoidal colimit preserving functors.  

\begin{prop}\label{colim_sym_mon_preserve_Gal}
Let $F\colon \cC \to \cD$ be a morphism in $\calg(\Prl)$ and let $BG$ be a weakly $\cC$-ambidextrous pointed connected space. Then, the functor $F\colon \calg(\cC)^{BG}\to \calg(\cD)^{BG}$ carries $G$-Galois extensions to $G$-Galois extensions. 
\end{prop}

\begin{proof}
Let $R$ be a $G$-Galois extension in $\cC$. We wish to show that $F(R)$ is a $G$-Galois extension in $\cD$. First, since $F$ is symmetric monoidal and colimit preserving, by \cite[Corollary 3.3.2]{TeleAmbi} the space $BG$ is also weakly $\cD$-ambidextrous and by \cite[Proposition 2.1.8]{AmbiHeight} the functor $F$ commutes with $G$-shaped limits. As a result, the functor  $F$ takes the map 
\[
R\otimes R \too \coind{G}{R}
\]
to the map
\[
F(R)\otimes F(R) \too \coind{G}{F(R)}.
\]
Since the former is an isomorphism by the assumption that $R$ is $G$-Galois, we deduce that the latter is also an isomorphism. 

It remains to show that the unit $\one_\cD\to F(R)^{hG}$ is an isomorphism. Regarding $\cD$ as a $\cC$-algebra via $F$, by \Cref{Galois_Descent} and the fact that the tensor product in  $\Mod_\cC$ preserves $BG$-shaped limits in each coordinate, we obtain a symmetric monoidal equivalence 
\[
\Mod_{F(R)}(\cD^{BG})\simeq \Mod_{F(R)}(\cC^{BG}\otimes_{\cC}\cD) \simeq \Mod_R(\cC^{BG})\otimes_{\cC}\cD \simeq \cC\otimes_{\cC} \cD \simeq \cD.
\]
Taking the endomorphism objects of the units from both sides we get the desired identification 
\[
F(R)^{hG}\simeq \one_\cD.
\]
\end{proof}

We turn to the main result of this section, regarding preservation of $T(n)$-local $G$-Galois extensions $R \to S$ under the functor $\KTnp$ for $G$ a $\pi$-finite $p$-group. By \cite[Proposition 3.2.3]{AmbiHeight}, we have canonical equivalences
\[
    \Sp_{\Tn}^{BG} \simeq 
    \Sp_{\Tn}^{\tau_{\le n+1} (BG)} \simeq
    \Sp_{\Tn}^{B(\tau_{\le n}G)}.
\] 
Thus, without loss of generality, we may restrict attention to \textit{$n$-finite} $p$-groups $G$. In case $G$ is further assumed to be \textit{$(n-1)$-finite},  we get by \Cref{KTnp-CAlg-lim} an isomorphism
\[
    \KTnp(R) \iso \KTnp(S)^{hG}.
\]
Namely, $\KTnp(R) \to \KTnp(S)$ satisfies the first Rognes--Galois condition. The second Rognes--Galois condition is satisfied automatically, by \cite[Corollary 7.31]{barthel2022chromatic}\footnote{This is part of the theory of \textit{affineness} developed in \cite[\S2]{barthel2022chromatic}.}, making $\KTnp(S)$ a $G$-Galois extension of $\KTnp(R)$. Using our results on Galois descent, we can bring the above reasoning to bear on the general $n$-finite case.

\begin{thm}\label{galois-comp}
    Let $n \geq 0$, and let $G$ be an $n$-finite $p$-group. For every $\Tn$-local $G$-Galois extension $R \to S$, the induced $\Tnp$-local $G$-extension 
    \[
        \KTnp(R) \too \KTnp(S)
    \]
    is Galois. 
\end{thm}
\begin{proof}
    By \Cref{Galois_Descent}, we have a canonical equivalence
    \[
        \Mod_R(\SpTn) \iso \Mod_S(\SpTn)^{hG}.
    \]
    Passing to dualizable objects we get an equivalence 
    \[
        \Mod_R(\SpTn)^\dbl \iso (\Mod_S(\SpTn)^{hG})^{\dbl}.
    \]
    By \cite[Proposition 4.6.1.11]{HA}, there is a canonical equivalence 
    \[
        (\Mod_S(\SpTn)^{hG})^\dbl \simeq
        (\Mod_S(\SpTn)^\dbl)^{hG}.
    \]
    Therefore
    \begin{align*}
        \KTnp(R)
        &\simeq \KTnp(\Mod_R(\SpTn)^\dbl)\\ 
        &\simeq \KTnp((\Mod_S(\SpTn)^\dbl)^{hG})\\
        &\simeq \KTnp(\Mod_S(\SpTn)^\dbl)^{hG}\\
        &\simeq \KTnp(S)^{hG}  
    \end{align*}
    where the first and last isomorphisms follow from \cite[Proposition 4.15]{clausen2020descent} and the third isomorphism follows from \Cref{main-thm}.
    This is the first Rognes--Galois condition for $\KTnp(S)$ as a $G$-extension of $\KTnp(R)$ in $\SpTnp$.
    Finally, by \cite[Corollary 7.31]{barthel2022chromatic}, the second Rognes--Galois condition is automatically satisfied.
\end{proof}


\section{Cyclotomic Redshift}

In this section we apply the higher descent results of the previous section to establish the compatibility of the chromatically localized algebraic $K$-theory functor with higher cyclotomic extensions (see \cref{cyclo-main}), the chromatic Fourier transform (see \cref{fourier-comp}) and Kummer theory (see \cref{kummer-comp}).
Roughly put, these constructions depend on the height parameter, and the functor $\KTnp$ intertwines the corresponding construction in height $n$ and in height $n+1$.\

Since all ring spectra in this section will be $\Tn$-local (or $\Tnp$-local), for
brevity we use the notation $R[\Omega A]$ for the group algebra computed in the $\Tn$-local (or $\Tnp$-local) category rather than in spectra.

\subsection{Cyclotomic Extensions}

In \cite{carmeli2021chromatic}, for any $R \in \Alg(\SpTn)$ the authors constructed a height $n$ analogue of the $p^r$-cyclotomic extension $\cyc{R}{p^r}{n}$, which is a direct factor of $R[B^n C_{p^r}]$. In this subsection we prove \cref{cyclo-main}, saying that there is an isomorphism
\[
    \cyc{\KTnp(R)}{p^r}{n+1} \iso \KTnp(\cyc{R}{p^r}{n}).
\]
For the convenience of the reader, we now give an informal account of the argument detailed in the rest of this subsection.
First, the statement easily reduces to the case $r = 1$ and commutative algebra $R$ (in fact, to $R = \Sph_{\Tn}$).
In the commutative case, the group algebra $R[B^n C_p]$ decomposes into the product of $\cyc{R}{p}{n}$ and $R$, with the projection to the second factor given by applying $R[-]$ to the map $B^n C_p \to \pt$.
Second, the isomorphism
\[
    \KTnp(R)[A] \iso \KTnp(R[\Omega A])
\]
from \cref{KTnp-R-A-Lnf} shows that the projection map $\KTnp(R)[B^{n+1} C_p] \to \KTnp(R)$ identifies with the image of the projection map $R[B^n C_p] \to R$ under $\KTnp$.
Finally, by the uniqueness of direct complements we get the required isomorphism
\[
    \cyc{\KTnp(R)}{p}{n+1} \iso \KTnp(\cyc{R}{p}{n}).
\]

\subsubsection{Generalities on Decompositions}

We begin with some general facts regarding decompositions of commutative algebras into a product. First, we observe that each of the two factors in the decomposition determines the other.

\begin{lem}\label{idemp-unique}
    Let $\cC$ be a symmetric monoidal stable category, and let $R \in \CAlg(\cC)$.
    Assume that we are given two decompositions
    \[
        R \simeq R_1 \times R_2,
        \quad R \simeq \overline{R}_1 \times \overline{R}_2
    \]
    in $\calg(\cC)$ and an isomorphism of $R_1$ and $\overline{R}_1$ under $R$.
    Then, there is an isomorphism of $R_2$ and $\overline{R}_2$ under $R$ as well, namely an isomorphism of decompositions.
\end{lem}

\begin{proof}
    Let $\varepsilon_1, \varepsilon_2, \overline{\varepsilon}_1, \overline{\varepsilon}_2 \in \pi_0 R$ be the idempotents splitting $R_1, R_2, \overline{R}_1, \overline{R}_2$ respectively.
    Taking $\pi_0$ to the isomorphism of $R_1$ and $\overline{R}_1$ under $R$ gives a commutative diagram of commutative rings:
    \[\begin{tikzcd}
    	& {\pi_0 R} \\
    	{\pi_0 R[\varepsilon_1^{-1}]} && {\pi_0 R[\overline{\varepsilon}_1^{-1}]}
    	\arrow[from=1-2, to=2-1]
    	\arrow[from=1-2, to=2-3]
    	\arrow["\sim", from=2-1, to=2-3]
    \end{tikzcd}\]
    Since the image of $1 - \varepsilon_1$ in $\pi_0 R[\varepsilon_1^{-1}]$ is $0$, commutativity shows that the same is true in $\pi_0 R[\overline{\varepsilon}_1^{-1}]$.
    Thus, we have 
    \[
    (1 - \varepsilon_1)\overline{\varepsilon}_1^n = 0 \qin \pi_0 R
    \] 
    for some $n \geq 1$.
    Since $\overline{\varepsilon}_1$ is idempotent, we conclude that $\overline{\varepsilon}_1 = \varepsilon_1 \overline{\varepsilon}_1$.
    Symmetrically, $\varepsilon_1 = \varepsilon_1 \overline{\varepsilon}_1$.
    Thus, $\varepsilon_1 = \overline{\varepsilon}_1$.

    Since $\varepsilon_2 = 1-\varepsilon_1$ and $\overline{\varepsilon}_2 = 1-\overline{\varepsilon}_1$, we get that $\varepsilon_2 = \overline{\varepsilon}_2$.
    Thus, we get an isomorphism 
    \[
        R_2 \simeq R[\varepsilon_2^{-1}] = R[\overline{\varepsilon}_2^{-1}] \simeq \overline{R}_2
    \]
    of commutative $R$-algebras, as required.
\end{proof}

We next show that the datum of a decomposition of a commutative algebra $R$ into a product is equivalent to the datum of a map $\one_\cC^2 := \one_\cC \times \one_\cC \to R$ where $\one_\cC$ is the unit of $\cC$.

\begin{lem}\label{dec-unit-2}
    Let $\cC \in \CAlg(\Catall)$ be $0$-semiadditive.
    There is a commutative diagram
    \[\begin{tikzcd}
    	{\CAlg(\cC)^2} && {\CAlg_{\unit_\cC^2}(\cC)} \\
    	& {\CAlg(\cC)}
    	\arrow["\sim", from=1-1, to=1-3]
    	\arrow[from=1-1, to=2-2]
    	\arrow[from=1-3, to=2-2]
    \end{tikzcd}\]
    where the left diagonal map is given by the binary product and the right diagonal map is the forgetful functor.
\end{lem}

\begin{proof}
   As in the proof of \cite[Proposition 5.1.11]{AmbiHeight} we have a commutative diagram  of symmetric monoidal categories 
    \[\begin{tikzcd}
    	{\cC\times \cC} & {\Mod_{\unit_{\cC}}(\cC)\times \Mod_{\unit_{\cC}}(\cC)} & {\Mod_{\unit_{\cC}^2}(\cC)} \\
    	& \cC
    	\arrow["\sim", from=1-1, to=1-2]
    	\arrow["\sim", from=1-2, to=1-3]
    	\arrow[from=1-3, to=2-2]
    	\arrow[from=1-2, to=2-2]
    	\arrow[from=1-1, to=2-2]
    \end{tikzcd}\]
    where the upper-right isomorphism is by the $0$-semiadditivity of $\cC$.
    The result then follows by applying $\CAlg$.
\end{proof}

Let $\cC \in \Catall$, and let $I$ be a category with an initial object denoted $i$.
Consider the adjunction
\[
    i_!\colon \cC \adj \Fun(I, \cC) \noloc i^*
\]
between evaluation at $i$ and left Kan extension along $\pt \xrightarrow{i} I$.
Observe that since $i$ is initial, for every $X \in \cC$ we have
\[
    (i_! X)(j) = \colim_{i \to j} X = X,
\]
namely $i_! X$ is the constant functor $\underline{X}\colon I \to \cC$.
We have the following corollary for decompositions of commutative algebras in the functor category $\Fun(I, \cC)$ with respect to the point-wise structure:

\begin{lem}\label{dec-initial}
    Let $I$ be a category with an initial object $i$ and let $\cC \in \CAlg(\Catall)$ be $0$-semiadditive.
    Then, a decomposition of a functor $F\colon I \to \CAlg(\cC)$ into a product is the same data as a map $\unit_\cC^2 \to F(i) \in \CAlg(\cC)$.
    That is, the following diagram commutes
    \[\begin{tikzcd}
    	{\Fun(I, \CAlg(\cC)^2)} && {\Fun(I, \CAlg(\cC))_{\unit_\cC^2 /}} \\
    	& {\Fun(I, \CAlg(\cC))}
    	\arrow["\sim", from=1-1, to=1-3]
    	\arrow[from=1-1, to=2-2]
    	\arrow[from=1-3, to=2-2]
    \end{tikzcd}\]
    where $\Fun(I, \CAlg(\cC))_{\unit_\cC^2 /}$ is the slice category of the evaluation at $i \in I$.
\end{lem}

\begin{proof}
    Endow $\Fun(I, \cC)$ with the point-wise symmetric monoidal structure, so that
    \[
        \Fun(I, \CAlg(\cC)) \simeq \CAlg(\Fun(I, \cC)),
    \]
    whose unit is the constant functor $\underline{\unit}_\cC$.
    Thus, by \cref{dec-unit-2}, we get that
    \[
        \Fun(I, \CAlg(\cC)^2) \simeq \CAlg_{\underline{\unit}_\cC^2}(\Fun(I, \cC))
    \]
    over $\Fun(I, \CAlg(\cC))$.
    Note that the functor $i^*$ is symmetric monoidal, which endows $i_!$ with an oplax symmetric monoidal structure.
    By the discussion above, $i_!$ is given by the formation of constant functor, and we see that the oplax structure is strong.
    Hence, the adjunction $i_! \dashv i^*$ lifts to an adjunction on commutative algebras.
    Since $\underline{\unit}_\cC \simeq i_! \unit_\cC$, by the adjunction, we get an equivalence of slice categories
    \[
        \CAlg_{\underline{\unit}_\cC^2}(\Fun(I, \cC))
        \simeq \CAlg(\Fun(I, \cC))_{\unit_\cC^2 /}
    \]
    over $\Fun(I, \CAlg(\cC))$, concluding the proof.
\end{proof}

Similarly, we have the following corollary for decompositions of commutative algebras in the functor category $\Fun(I, \cC)$ with respect to the \emph{Day convolution}:

\begin{lem}\label{dec-unit}
    Let $I$ be a symmetric monoidal category whose unit $\one_I$ is initial, and let $\cC \in \CAlg(\Catall)$ be $0$-semiadditive.
    Then, a decomposition of a lax symmetric monoidal functor $F\colon I \to \cC$ into a product is the same data as a map $\unit_\cC^2 \to F(\unit_I) \in \CAlg(\cC)$.
    That is, the following diagram commutes
    \[\begin{tikzcd}
    	{\Fun^\lax(I, \cC^2)} && {\Fun^\lax(I, \cC)_{\unit_\cC^2 /}} \\
    	& {\Fun^\lax(I, \cC)}
    	\arrow["\sim", from=1-1, to=1-3]
    	\arrow[from=1-1, to=2-2]
    	\arrow[from=1-3, to=2-2]
    \end{tikzcd}\]
\end{lem}

\begin{proof}
    The argument is analogous to the proof of \cref{dec-initial}, for the Day convolution.
    Recall that
    \[
        \Fun^\lax(I, \cC) \simeq \CAlg(\Fun(I, \cC)_\Day).
    \]
    Since $\unit_I$ is assumed to be initial, we see that the unit of the Day convolution is
    \[
        \Map(\unit_I, -) \otimes \unit_\cC \simeq \underline{\unit}_\cC.
    \]
    Thus, by \cref{dec-unit-2}, we get that
    \[
        \Fun^\lax(I, \cC^2) \simeq \CAlg_{\underline{\unit}_\cC^2}(\Fun(I, \cC)_\Day)
    \]
    over $\Fun^\lax(I, \cC)$.
    Since $\unit_I\colon \pt \to I$ is symmetric monoidal, by \cite[Proposition 3.34]{linskens2022global} or \cite[Proposition 3.6]{moshe2021higher}, the adjunction $(\unit_I)_! \dashv (\unit_I)^*$ is symmetric monoidal, and hence lifts to an adjunction on commutative algebras.
    By the discussion above, $\underline{\unit}_\cC \simeq (\unit_I)_! \unit_\cC$, so by adjunction, we get an equivalence of slice categories
    \[
        \CAlg_{\underline{\unit}_\cC^2}(\Fun(I, \cC)_\Day)
        \simeq \CAlg(\Fun(I, \cC)_\Day)_{\unit_\cC^2 /}
    \]
    over $\Fun^\lax(I, \cC)$, which concludes the proof.
\end{proof}

\subsubsection{Cyclotomic Extensions}\label{Sec_Cyc_Ext}

Let $\cC$ be a presentably symmetric monoidal stable $\infty$-semiadditive category with all objects of semiadditive height $n$ (e.g.\ $\SpTn$).
By \cite[Proposition 4.5]{carmeli2021chromatic} (see also \cite[Definition 4.7]{carmeli2021chromatic} and \cite[Corollary 6.7]{barthel2022chromatic}), the map
\[
    \Sigma^n \Fp \too 0
    \qin \Spcn^{B\Fp^\times}
\]
induces a decomposition
\[
    \unit_\cC[\Sigma^n \Fp] \simeq \unit_\cC \times \mdef{\cyc{\unit_\cC}{p}{n}}
    \qin \CAlg(\cC)^{B\Fp^\times},
\]
where $\cyc{\unit_\cC}{p}{n}$ is called the (height $n$) \tdef{$p$-th cyclotomic extension}.
For every $M \in \Spcn$ equipped with a map $\Sigma^n \Fp \to M$, the induced map
\[
    \unit_\cC[\Sigma^n \Fp] \too \unit_\cC[M]
    \qin \CAlg(\cC),
\]
makes $\unit_\cC[M]$ into a $\unit_\cC[\Sigma^n \Fp]$-algebra.

\begin{defn}
    For $X \in \cC$ and $M \in \Spcn$ equipped with a map $\Sigma^n \Fp \to M$, we let
    \[\begin{gathered}
        \mdef{X[M]_0} := X \otimes \unit_\cC[M] \tensunder{\unit_\cC[\Sigma^n \Fp]} \unit_\cC,
        \qquad
        \mdef{X[M]_\omega} := X \otimes \unit_\cC[M] \tensunder{\unit_\cC[\Sigma^n \Fp]} \cyc{\unit_\cC}{p}{n}
    \end{gathered}\]
    which gives a decomposition 
    \[
    X[M] \simeq X[M]_0 \times X[M]_\omega \qin \cC.
    \]
\end{defn}

Observe that this decomposition is natural in $M \in (\Spcn)_{\Sigma^n \Fp/}$ and lax symmetric monoidally natural in $X \in \cC$.

\begin{example}\label{cyclo-decs}
    Let $R \in \Alg(\SpTn)$ and let $M \in \Sp_{\geq 1}$ be equipped with a map $\Sigma^{n+1} \Fp \to M$.
    \begin{enumerate}
        \item Applying the above to $X = \KTnp(R)$ we get a decomposition
        \[
            \KTnp(R)[M] \simeq \KTnp(R)[M]_0 \times \KTnp(R)[M]_\omega
            \qin \SpTnp.
        \]
        \item Applying the above to $X = R$, and looping $M$, we get a decomposition
        \[
            R[\Omega M] \simeq R[\Omega M]_0 \times R[\Omega M]_\omega
            \qin \Alg(\SpTn),
        \]
        and since $\KTnp$ preserves finite products, we get a decomposition
        \[
            \KTnp(R[\Omega M]) \simeq \KTnp(R[\Omega M]_0) \times \KTnp(R[\Omega M]_\omega)
            \qin \SpTnp.
        \]
    \end{enumerate}
    Observe that these decompositions are natural in $M$ and lax symmetric monoidally natural in $R$. 
\end{example}

Our goal is to compare the two decompositions in the above example. 
Recall from \cref{KTnp-R-A-Lnf} that there is a map
\[
    \KTnp(R)[A] \too \KTnp(R[\Omega A])
    \qin \SpTnp,
\]
lax symmetric monoidally natural in $A \in \Spc^{\geq 1}_*$ and $R \in \Alg(\SpTn)$, which is an isomorphism when $A$ is a sifted colimit of $\pi$-finite $p$-spaces.
By pre-composing with $\Omega^\infty\colon \Sp_{\geq 1} \to \Spc^{\geq 1}_*$ and passing to objects under $\Sigma^{n+1} \Fp$, we get the following:

\begin{defn}
    Using the above, we define the map
    \[
        \KTnp(R)[M] \too \KTnp(R[\Omega M])
        \qin \SpTnp,
    \]
    natural in $M \in (\Sp_{\geq 1})_{\Sigma^{n+1} \Fp /}$ and lax symmetric monoidally natural in $R \in \Alg(\SpTn)$.
\end{defn}

Since $\Omega^\infty$ preserves filtered colimits, this map is an isomorphism when $M$ is a filtered colimit of $\pi$-finite $p$-spectra.

\begin{prop}\label{cyclo-for-cnsp}
    The map
    \[
        \KTnp(R)[M] \too \KTnp(R[\Omega M])
        \qin \SpTnp
    \]
    respects the decompositions of the source and the target from \cref{cyclo-decs}.
\end{prop}

\begin{proof}
    We begin by reducing to the initial case, namely $R = \Sph_{\Tn}$ and $M = \Sigma^{n+1} \Fp$.
    The map in question is a natural transformation of functors
    \[
        (\Sp_{\geq 1})_{\Sigma^{n+1} \Fp /} \too \Fun^\lax(\Alg(\SpTn), \SpTnp)).
    \]
    Since $\Sigma^{n+1} \Fp$ is the initial object of $(\Sp_{\geq 1})_{\Sigma^{n+1} \Fp /}$, by \cref{dec-initial}, it suffices to check that the map
    \[
        \KTnp(R)[\Sigma^{n+1} \Fp] \too \KTnp(R[\Sigma^n \Fp])
    \]
    of lax symmetric monoidal functors $\Alg(\SpTn) \to \SpTnp$ respects the decompositions.
    Since the unit $\Sph_{\Tn} \in \Alg(\SpTn)$ is the initial object, by \cref{dec-unit}, it suffices to check that the map
    \[
        \KTnp(\Sph_{\Tn})[\Sigma^{n+1} \Fp] \too \KTnp(\Sph_{\Tn}[\Sigma^n \Fp])
        \qin \CAlg(\SpTnp)
    \]
    respects the decompositions, which we now show.

    Applying the map in question to $R = \Sph_{\Tn}$ and $M = 0$, we get the following commutative square in $\CAlg(\SpTnp)$:
    \[\begin{tikzcd}
    	{\KTnp(\Sph_{\Tn})[\Sigma^{n+1} \Fp]} & {\KTnp(\Sph_{\Tn}[\Sigma^n \Fp])} \\
    	{\KTnp(\Sph_{\Tn})} & {\KTnp(\Sph_{\Tn})}
    	\arrow["\sim", from=1-1, to=1-2]
    	\arrow[from=1-1, to=2-1]
    	\arrow[from=1-2, to=2-2]
    	\arrow["\sim", from=2-1, to=2-2]
    \end{tikzcd}\]
    where both horizontal maps are isomorphisms because $\Sigma^{n+1} \Fp$ and $0$ are $\pi$-finite $p$-spectra.
    Thus, the two decompositions coincide by \cref{idemp-unique}.
\end{proof}

\begin{rem}\label{rem:cyclo-equiv}
    Recall that the spectrum $\Sigma^{n+1} \FF_p$ is acted by the group $\FF_p^\times$ and the splitting
    \[
        \one_\cC[\Sigma^{n+1} \FF_p] \simeq
        \one_\cC \times \cyc{\one_\cC}{p}{n+1}
    \]
    is $\FF_p^\times$-equivariant.
    Now, let $G$ be a group acting on $M \in \Sp_{\geq 1}$, equipped with a map $G \to \Fp^\times$ and a $G$-equivariant map $\Sigma^{n+1} \FF_p \to M$.
    Since the decompositions of \cref{cyclo-decs} are natural in $M$, the induced map
    \[
        \KTnp(R)[M]_\omega \too 
        \KTnp(R[\Omega M]_\omega) \qin 
        \SpTnp
    \]
    is canonically $G$-equivariant.
\end{rem}

We now specialize to the case of cyclotomic extensions. We begin by recalling some definitions and notations from \cite{barthel2022chromatic} regarding Brown--Comenetz duality.  
Let $I_{\QQ_p/\ZZ_p}$ be the $p$-local Brown--Comenetz spectrum, and let
\[
    \Dual{}{n}\colon 
    (\Sp^{[0,n]})^\op \too \Spcn
\]
be the functor sending $M$ to 
\[
    \Dual{M}{n} :=
    \hom_{\Spcn}(M, \tau_{\ge 0} \Sigma^n I_{\QQ_p/\ZZ_p}).
\]
Since $\Dual{\Fp}{n} \simeq \Sigma^n \Fp$, we get a further induced functor
\[
    \Dual{}{n}\colon 
    (\Sp^{[0,n]}_{/ \Fp})^\op \too 
    (\Sp^{[0,n]})_{\Sigma^n \Fp /}.
\]

\begin{defn}
    For $X \in \cC$ and $M \in \Sp^{[0,n]}_{/ \Fp}$ we let
    \[
        \mdef{\cyc{X}{M}{n}} := X[\Dual{M}{n}]_\omega
        \qin \cC.
    \]
\end{defn}

By \cite[Remark 6.10]{barthel2022chromatic}, when $M = \OR$ is a commutative algebra in $\Sp^{[0,n]}$ augmented over $\FF_p$, this agrees with $\cyc{X}{\OR}{n}$ of \cite[Definition 4.19]{barthel2022chromatic}.

\begin{thm}\label{cyclo-main}
    There is a map
    \[
        \cyc{\KTnp(R)}{M}{n+1} \too \KTnp(\cyc{R}{M}{n})
        \qin \SpTnp,
    \]
    natural in $M \in \Sp^{[0,n]}_{/ \Fp}$ and lax symmetric monoidally natural in $R \in \Alg(\SpTn)$.
    Moreover, when the homotopy groups of $\Dual{M}{n}$ are torsion $p$-groups, the map is an isomorphism.
\end{thm}

\begin{proof}
    The first part follows immediately from \cref{cyclo-for-cnsp} by pre-composing with
    \[\begin{gathered}
        \Dual{}{n+1}\colon 
        (\Sp^{[0,n]}_{/ \Fp})^\op \too 
        (\Sp_{\geq 1})_{\Sigma^{n+1} \Fp /}
    \end{gathered}\]
    and projecting to the $(-)_\omega$ coordinates.
    For the second part, when the homotopy groups of $\Dual{M}{n}$ are torsion $p$-groups, the same is true for $\Dual{M}{n+1} \simeq \Sigma \Dual{M}{n}$, so by \cite[Lemma 6.57]{barthel2022chromatic} it is a filtered colimit of $\pi$-finite $p$-spectra, concluding the proof.
\end{proof}

\begin{rem}\label{rem:cyclo-equiv-ring}
    By \Cref{rem:cyclo-equiv}, if a group $G$ acts on $M$ and is mapped to $\FF_p^\times$ in a way making the map $M \to \FF_p$ equivariant, then the comparison map in \Cref{cyclo-main} is $G$-equivariant.   
\end{rem}

\begin{example}\label{wpr}
    Consider the case $M = \ZZ/p^r$, and recall that by \cite[Corollary 6.7]{barthel2022chromatic} the extension $\cyc{R}{\ZZ/p^r}{n} \simeq \cyc{R}{p^r}{n}$ is the higher cyclotomic extension defined in \cite[Definition 4.7]{carmeli2021chromatic}.
    Thus, we get an isomorphism
    \[
        \cyc{\KTnp(R)}{p^r}{n+1} \iso \KTnp(\cyc{R}{p^r}{n})
        \qin \SpTnp^{B(\ZZ/p^r)^\times}.
    \]
\end{example}

\begin{example}\label{wpinfty}
    For every $r \in \NN$, we have a natural action of $\ZZ_p^\times$ on $\ZZ/p^r$, and the projection maps $\ZZ/p^{r+1} \onto \ZZ/p^r$ are $\ZZ_p^\times$-equivariant. By naturality, we get an isomorphism of sequential diagrams in $\SpTnp^{B\ZZ_p^\times}$:
    \[\begin{tikzcd}
    	{\cyc{\KTnp(R)}{p}{n+1}} & {\cyc{\KTnp(R)}{p^2}{n+1}} & \cdots & {\cyc{\KTnp(R)}{p^r}{n+1}} & \cdots \\
    	{\KTnp(\cyc{R}{p}{n})} & {\KTnp(\cyc{R}{p^2}{n})} & \cdots & {\KTnp(\cyc{R}{p^r}{n})} & \cdots
    	\arrow["\wr"', from=1-4, to=2-4]
    	\arrow["\wr"', from=1-2, to=2-2]
    	\arrow["\wr"', from=1-1, to=2-1]
    	\arrow[from=1-1, to=1-2]
    	\arrow[from=1-2, to=1-3]
    	\arrow[from=1-3, to=1-4]
    	\arrow[from=1-4, to=1-5]
    	\arrow[from=2-2, to=2-3]
    	\arrow[from=2-3, to=2-4]
    	\arrow[from=2-4, to=2-5]
    	\arrow[from=2-1, to=2-2]
    \end{tikzcd}\]
    Taking the colimit of the sequences, by \cite[Definition 4.10]{carmeli2021chromatic} (see also \cite[Corollary 6.18]{barthel2022chromatic}) and the preservation of filtered colimits under $\KTnp$, we get an isomorphism
    \[
        \cyc{\KTnp(R)}{p^\infty}{n+1} \iso \KTnp(\cyc{R}{p^\infty}{n})
        \qin \SpTnp^{B\ZZ_p^\times}.
    \]
\end{example}

Our next goal is to extend the last two examples from cyclotomic extensions to any intermediate extension.
Let $0 \leq r \leq \infty$, and let $G = (\ZZ/p^r)^\times$ if $r < \infty$ or $G = \ZZ_p^\times$ if $r = \infty$, with its canonical action on $\cyc{R}{p^r}{n}$ over $R$.
For any subgroup $H \leq G$, the assembly map for $H$-fixed points together with the (inverses of) the isomorphisms from the examples above give us a map
\[
    \KTnp(\cyc{R}{p^r}{n}^{hH})
    \too \KTnp(\cyc{R}{p^r}{n})^{hH}
    \simeq \cyc{\KTnp(R)}{p^r}{n+1}^{hH}
    \qin \SpTn^{B(G/H)}
\]
lax symmetric monoidally natural in $R \in \Alg(\SpTn)$.

We also remind the reader that $\ZZ_p^\times$ can be decomposed as
\[
    \ZZ_2^{\times} \simeq (\ZZ/4)^\times \times (1 + 4\ZZ_2),
    \qquad \ZZ_p^{\times} \simeq \Fp^\times \times (1 + p\ZZ_p).
\]
In other words, $\ZZ_p^\times \simeq T_p \times \ZZ_p$ where $T_p$ is $(\ZZ/4)^\times$ for $p = 2$ and $\Fp^\times$ for odd primes.

\begin{prop}\label{inter-galois}
    Let $0 \leq r \leq \infty$, and let $H$ be a finite subgroup of $G$, then the map
    \[
        \KTnp(\cyc{R}{p^r}{n}^{hH}) \too \cyc{\KTnp(R)}{p^r}{n+1}^{hH}
    \]
    is an isomorphism.
\end{prop}

\begin{proof}
    We begin with the case of finite $r$.
    If $r = 0$ then $G$ is trivial and there is nothing to prove, so we assume that $r \geq 1$.
    Denote $S = \cyc{R}{p^r}{n}$.
    It suffices to show that the map
    \[
        \KTnp(S^{hH}) \too \KTnp(S)^{hH}
    \]
    is an isomorphism.
    By \cite[Proposition 5.2]{carmeli2021chromatic}, the cyclotomic extension $\cyc{\Sph_{\Tnp}}{p^r}{n+1}$ is faithful. Thus, we can check if the above assembly map is an isomorphism after tensoring with it. Namely, after adding $p^r$-th roots of unity. The resulting map is the top arrow in the following diagram:
    \[\begin{tikzcd}
    	{\cyc{\KTnp(S^{hH})}{p^r}{n+1} } &&&& {\cyc{\KTnp(S)^{hH}}{p^r}{n+1}} \\
    	{\KTnp(\cyc{S^{hH}}{p^r}{n}) } &&&& {\cyc{\KTnp(S)}{p^r}{n+1}^{hH}} \\
    	{\KTnp(\cyc{S}{p^r}{n}^{hH})} &&&& {\KTnp(\cyc{S}{p^r}{n})^{hH}}
    	\arrow[from=1-1, to=1-5]
    	\arrow["\wr"', from=1-1, to=2-1]
    	\arrow["\wr"', from=1-5, to=2-5]
    	\arrow["\wr"', from=2-5, to=3-5]
    	\arrow["\wr"', from=2-1, to=3-1]
    	\arrow[from=3-1, to=3-5]
    	\arrow[from=1-1, to=2-5]
    	\arrow[from=2-1, to=3-5]
    \end{tikzcd}\]
    To see that the diagram commutes, note that the upper and lower triangles commute by the compatibility of assembly maps under composition, and the middle square commutes by the naturality of the cyclotomic redshift map (\cref{cyclo-main}) applied to the assembly map.
    The bottom left and top right vertical maps are isomorphism since $\SpTn$ and $\SpTnp$ are $1$-semiadditive, hence their tensor products commute with $1$-finite limits in each coordinate.
    The top left and bottom right vertical maps are isomorphisms again by cyclotomic redshift (\cref{cyclo-main}).
    Thus, the top map is an isomorphism if and only if the bottom map is an isomorphism. 
    By the second Rognes--Galois condition we have  
    \[
        \cyc{S}{p^r}{n} \simeq
        \prod_{(\ZZ/p^r)^\times} S \qin
        \Alg(\SpTnp)^{B(\ZZ/p^r)^\times}
    \]
    where the action of $(\ZZ/p^r)^\times$, hence of $H$, is given by freely permuting the factors.
    Finally, $\KTnp$ preserves finite products, so the bottom map identifies with the isomorphism
    \[
        \KTnp(\prod_{(\ZZ/p^r)^\times/H} S) \iso \prod_{(\ZZ/p^r)^\times/H} \KTnp(S).
    \]

    Now assume that $r = \infty$.
    Since $H$ is finite, it is contained in $T_p$.
    Note that for $k \geq 2$ if $p$ is even and for $k \geq 1$ if $p$ is odd, $T_p$ also injects to the Galois group of $\cyc{R}{p^k}{n}$, and the isomorphism $\cyc{R}{p^\infty}{n} \simeq \colim \cyc{R}{p^k}{n}$ is $T_p$-equivariant.
    Since $\SpTn$ and $\SpTnp$ are $1$-semiadditive we get that $(-)^{hH}$ commutes with colimits, and $\KTnp$ commutes with filtered colimits, so the result follows from the finite case.
\end{proof}

\subsubsection{Cyclotomic Completion}

Recall that given $R \in \CAlg(\SpTn)$, the infinite cyclotomic extension $\cyc{R}{p^\infty}{n}$ is not guaranteed to be faithful over $R$.
In \cite{barthel2022chromatic}, the authors introduced the \textit{cyclotomic completion} functor $\cycomp{(-)} \colon \SpTn \to \SpTn$, which is the universal localization making all infinite cyclotomic extensions faithful, namely, the $\cyc{\Sph_{\Tn}}{p^\infty}{n}$-localization.
Moreover, in \cite[Proposition 6.19]{barthel2022chromatic} the authors gave a formula for the cyclotomic completion.
Recall that $\ZZ_p^\times \simeq T_p \times \ZZ_p$.
The cyclotomic completion is given by taking fixed points with respect to the (dense) subgroup $T_p \times \ZZ$, i.e.,
\[
    \cycomp{R} \simeq 
    \cyc{R}{p^\infty}{n}^{h(T_p \times \ZZ)} \qin
    \Alg(\SpTn).
\]
By cyclotomic redshift (\cref{cyclo-main}), we have a $\ZZ_p^\times$-equivariant isomorphism
\[
    \cyc{\KTnp(R)}{p^\infty}{n+1} \iso \KTnp(\cyc{R}{p^\infty}{n}).
\]
We thus get the following formula for the cyclotomic completion of $\KTnp(R)$:
\[
    \cycomp{\KTnp(R)} \iso \KTnp(\cyc{R}{p^\infty}{n})^{h(T_p\times \ZZ)}.
\]

One can take fixed points with respect to $T_p\times \ZZ$ in two steps, by first taking the fixed points with respect to the finite group $T_p$, and then the fixed points with respect to the residual action of $\ZZ$.
This leads to consider the following intermediate extensions:

\begin{defn}\label{iwa-def}
    For $X \in \SpTn$ we let
    \[
        \mdef{X\iwa{r}{n}} := \cyc{X}{p^r}{n}^{hT_p}
        \qin \SpTn.
    \]
\end{defn}

Observe that when $p = 2$ we have $X\iwat{2}{n} \simeq X$ and $X\iwat{r}{n}$ carries a residual $(\ZZ/2^{r-2})$-action, while for odd $p$ we have $X\iwa{1}{n} \simeq X$ and $X\iwa{r}{n}$ carries a residual $(\ZZ/p^{r-1})$-action.
We also note that since $\SpTn$ is $1$-semiadditive, we have
\[
    X\iwa{\infty}{n}
    = \cyc{X}{p^\infty}{n}^{hT_p}
    = (\colim \cyc{X}{p^r}{n})^{hT_p}
    \iso \colim \cyc{X}{p^r}{n}^{hT_p}
    = \colim X\iwa{r}{n}
    \qin \SpTn^{B\ZZ_p}.
\]

\begin{prop}\label{iwa-comp}
    For every $R \in \Alg(\Sp_{\Tn})$ the map 
    \[
        \KTnp(R) \too \KTnp(R\iwa{\infty}{n})^{h\ZZ}
    \]
    exhibits the target as the cyclotomic completion of the source.
    In particular, if $R$ itself is cyclotomically complete, the assembly map
    \[
        \KTnp(R\iwa{\infty}{n}^{h\ZZ})
        \too \KTnp(R\iwa{\infty}{n})^{h\ZZ}.
    \]
    exhibits the target as the cyclotomic completion of the source.
\end{prop}

\begin{proof}
    For the first part, by \cite[Proposition 6.19]{barthel2022chromatic}, for any $X \in \SpTn$ the map
    \[
        X
        \iso \cyc{X}{p^\infty}{n}^{h(T_p \times \ZZ)}
        \simeq X\iwa{\infty}{n}^{h\ZZ}
    \]
    exhibits the target as the cyclotomic completion of $X$.
    By \cref{inter-galois}, we have
    \[
        \KTnp(R)\iwa{\infty}{n+1} \iso \KTnp(R\iwa{\infty}{n})
        \qin \SpTnp^{B\ZZ_p}.
    \]
    Taking $\ZZ$-fixed points and applying the above to $X = \KTnp(R)$, we conclude that
    \[
        \KTnp(R) \too \KTnp(R\iwa{\infty}{n})^{h\ZZ}
    \]
    exhibits the target as the cyclotomic completion of the source, as required.
    
    We move on to the second part.
    By the naturality of \cref{cyclo-main} in $M$ and \cref{rem:cyclo-equiv-ring}, the map from the first part is given by applying $\KTnp$ to the $\ZZ$-equivariant map $R \to R\iwa{\infty}{n}$, where $R$ is endowed with the trivial action, and passing to the mate.
    Namely, it is the bottom composition in the following diagram:
    \[\begin{tikzcd}
    	{\KTnp(R)} & {\KTnp(R^{h\ZZ})} & {\KTnp(R\iwa{\infty}{n}^{h\ZZ})} \\
    	& {\KTnp(R)^{h\ZZ}} & {\KTnp(R\iwa{\infty}{n})^{h\ZZ}}
    	\arrow[from=1-1, to=2-2]
    	\arrow[from=2-2, to=2-3]
    	\arrow[from=1-1, to=1-2]
    	\arrow[from=1-2, to=1-3]
    	\arrow[from=1-3, to=2-3]
    	\arrow[from=1-2, to=2-2]
    \end{tikzcd}\]
    The left triangle clearly commutes, and the right square commutes by the naturality of the assembly map.
    Since $R$ is cyclotomically complete we know that $R \iso R\iwa{\infty}{n}^{h\ZZ}$, so that the top composition is an isomorphism.
    Therefore, by the first part, the assembly map appearing on the right exhibits the target as the cyclotomic completion of the source.
\end{proof}

\subsection{Fourier Transform}

We shall now explain the compatibility of the functor $\KTnp$ with the chromatic Fourier transform of \cite{barthel2022chromatic}.
With notation as in \Cref{Sec_Cyc_Ext}, we recall the construction of this Fourier transform.

\begin{notation}
    To avoid cumbersome notation, whenever we have a connective spectrum $M$ and an object $X$ of some category $\cC$, we shall denote by $X^M$ the constant $\Omega^\infty M$-shaped limit on $X$, instead of $X^{\Omega^\infty M}$.
\end{notation}

\begin{defn}[{\cite[Definition 3.6 and Definition 3.11]{barthel2022chromatic}}]
    Let $\cC \in \CMon(\Cat)$ and $\OR \in \CAlg(\Spcnp)$.
    The space of \tdef{$\OR$-pre-orientations} of height $n$ of $\cC$ is defined to be
    \[
        \mdef{\POr{\OR}{\cC}{n}} := 
        \Map_{\Spcn}(\Dual{\OR}{n}, \unit_\cC^\times).
    \]
    By \cite[Proposition 3.10]{barthel2022chromatic}, for $\cC \in \CAlg(\Prl)$ the space $\POr{\OR}{\cC}{n}$ is equivalent to the space of natural transformations 
     \[
        \unit_\cC[-] \too \unit_\cC^{\Dual{(-)}{n}}
    \] 
    between functors $\Mod_{\OR}^{[0,n]}\to \CAlg(\cC)$.
    We denote the natural transformation corresponding to $\omega \in \POr{\OR}{\cC}{n}$ by 
    \[
        \mdef{\Four_\omega}\colon \unit_\cC[-] \too \unit_\cC^{\Dual{(-)}{n}}
    \]
    and refer to $\Four_\omega$ as the \tdef{Fourier transform} associated with $\omega$. Finally, we say that $\omega$ is an \tdef{$\OR$-orientation}, if $\Four_\omega$ is an isomorphism for all \textit{$\pi$-finite} $M \in \Mod_{\OR}^{[0,n]}$.
\end{defn}

The main result of this section is \Cref{fourier-comp}, which shows (under certain assumptions on $\omega$ and $M$) that the image of the Fourier transform 
\[
    \Four_\omega\colon R[M]\too R^{\Dual{M}{n}}
\] 
under the functor $\KTnp\colon \calg(\SpTn) \to \calg(\SpTnp)$ is a Fourier transform with respect to a corresponding pre-orientation. 
Recall that the functor $\KTnp$ can be written as a composition of three functors: 
\[
    \calg(\SpTn) \too[\cMod_{(-)}] \calg_{\SpTn}(\Prl) \too[(-)^{\dbl}] \calg(\CatLnf) \too[\KTnp] \calg(\SpTnp). 
\]
Thus, we are led to consider the compatibility of three different operations with the formation of Fourier transforms:

\begin{enumerate}
    \item For $R\in \calg(\SpTn)$, we shall relate the Fourier transforms of $R$ and of $\cMod_R$.
    This is essentially done in \cite[\S5]{barthel2022chromatic} by discussing more generally the interaction of Fourier theory with categorification, and we review and slightly expand this discussion.  
    \item For $\cC \in \calg_{\SpTn}(\Prl)$, we shall relate the Fourier transforms of $\cC$ and of its small subcategory $\cC^\dbl \in \calg(\CatLnf)$. 
    \item For $\cD \in \calg(\CatLnf)$, we relate the Fourier transforms of $\cD$ with that of $\KTnp(\cD)$.
    More generally, we shall show that the Fourier transform is compatible with higher semiadditive lax symmetric monoidal functors. 
\end{enumerate}

We now deal with each of these points, and later combine them to prove our main result, \cref{fourier-comp}.

\subsubsection{Fourier and Categorification}

For $\cC\in \calg(\Prl)$, recall from \cite[\S5.2]{barthel2022chromatic} that taking loop spaces gives a natural map
\[
    \Omega \colon \POr{\OR}{\Mod_\cC(\Prl)}{n+1} \too \POr{\OR}{\cC}{n},
\]
taking a pre-orientation $\omega \colon \Dual{\OR}{n+1} \to  \cC^\times$ to the pre-orientation given by 
\[
\Dual{\OR}{n} \simeq \Omega \Dual{\OR}{n+1} \too[\ \Omega \omega\ ] \Omega \cC^\times \simeq \unit_\cC^\times. 
\]
The corresponding Fourier transforms are related by the following result:

\begin{prop}\label{fourier-dec}
    Let $\cC \in \CAlg(\Prl)$, let $\OR \in \CAlg(\Spcnp)$ and let $\omega \colon \Dual{\OR}{n+1}\to \cC^\times$ be a pre-orientation of $\Mod_{\cC}(\Prl)$ of height $n+1$.
    Then, we have the following commutative diagram in $\CAlg_\cC(\Prl)$
    \[\begin{tikzcd}
    	{\Mod_{\unit_\cC[M]}(\cC)} & {\Mod_{\unit_\cC^{\Dual{M}{n}}}(\cC)} \\
    	{\cC[\Sigma M]} & {\cC^{\Dual{\Sigma M}{n+1}}}
    	\arrow["{\Four_{\Omega\omega}}", from=1-1, to=1-2]
    	\arrow["{\Four_\omega}", from=2-1, to=2-2]
    	\arrow["\wr", from=2-1, to=1-1]
    	\arrow[from=1-2, to=2-2]
    \end{tikzcd}\]
    naturally in $M \in \Mod_{\OR}^{[0,n]}$, where the right vertical map is the left adjoint of the global sections functor as in \cite[\S2.1]{barthel2022chromatic}.
\end{prop}

\begin{proof}
    Taking ($\cC$-linear) endomorphisms of the units from the categorical Fourier transform 
    \[
    \Four_\omega \colon \cC[\Sigma M] \to \cC^{\Dual{\Sigma M}{n+1}}  
    \]
    gives a morphism 
    \[
    \Dec\Four_\omega \colon \unit_\cC[\Omega \Sigma M] \simeq \End(\unit_{\cC[\Sigma M]})\to \End(\unit_{\cC^{\Dual{\Sigma M}{n+1}}}) \simeq \unit^{\Dual{M}{n}}.  
    \]
    By \cite[Proposition 5.13]{barthel2022chromatic}, this morphism fits into a commutative diagram in $\CAlg(\SpTn)$, natural in $M$:
    \[
    \begin{tikzcd}
    	{\unit_\cC[M]} & {\unit_\cC^{\Dual{M}{n}}} \\
    	{\unit_\cC[\Omega\Sigma M]} & {\unit_\cC^{\Dual{\Sigma M}{n+1}}}
    	\arrow["\wr", from=2-1, to=1-1]
    	\arrow["{\Four_{\Omega\omega}}", from=1-1, to=1-2]
    	\arrow["\wr"', from=1-2, to=2-2]
    	\arrow["{\Dec\Four_\omega}", from=2-1, to=2-2]
    \end{tikzcd}
    \]
    Now, consider the following diagram in $\CAlg_\cC(\Prl)$:
    \[\begin{tikzcd}
    	{\Mod_{\unit_\cC[M]}(\cC)} & {\Mod_{\unit_\cC^{\Dual{M}{n}}}(\cC)} \\
    	{\Mod_{\unit_\cC[\Omega\Sigma M]}(\cC)} & {\Mod_{\unit_\cC^{\Dual{\Sigma M}{n+1}}}(\cC)} \\
    	{\cC[\Sigma M]} & {\cC^{\Dual{\Sigma M}{n+1}}}
    	\arrow["{\Four_{\Omega\omega}}", from=1-1, to=1-2]
    	\arrow["{\Four_\omega}", from=3-1, to=3-2]
    	\arrow["\wr"', from=2-1, to=3-1]
    	\arrow[from=2-2, to=3-2]
    	\arrow["\wr", from=1-2, to=2-2]
    	\arrow["\wr", from=2-1, to=1-1]
    	\arrow["{\Dec\Four_\omega}", from=2-1, to=2-2]
    \end{tikzcd}\]
    The upper part is obtained by applying $\Mod_{(-)}(\cC)$ to the previous commutative diagram.
    The bottom commutative square is obtained by applying the counit map of the symmetric monoidal adjunction
    \[
        \LMod_{(-)}(\cC)\colon \Alg(\cC) \rightleftarrows \Mod_\cC(\Prl)_* \noloc\End(-)
    \]
    to the categorical Fourier transform, and the bottom left map is an isomorphism due to \cref{C-A}. Since the outer rectangle of this diagram is the diagram from the statement of the proposition, the result follows.
\end{proof}

\begin{rem}\label{rem:affine-four-dec}
    In the setting of \Cref{fourier-dec}, when $\Dual{\Sigma M}{n+1}$ is $\cC$-affine in the sense of \cite[Definition 2.15]{barthel2022chromatic}, the right map in the commutative square is also an isomorphism.
    Thus, in this case we get an equivalence between the functor
    \[
        \Four_{\omega}\colon \cC[\Sigma M] \too \cC^{\Dual{\Sigma M}{n+1}}
    \]
    and the functor induced from 
    \[
        \Four_{\Omega \omega}\colon \unit_\cC[M] \too \unit_\cC^{\Dual{M}{n}} 
    \]
    on module categories. 
\end{rem}

\subsubsection{Passing to Dualizable Objects}

Recall that for a symmetric monoidal category $\cC$, the natural inclusion $i\colon \cC^\dbl \into \cC$ induces a natural isomorphism $i\colon (\cC^\dbl)^\times \iso \cC^\times$.
Taking maps from $\Dual{\OR}{n}$ into this isomorphism, we obtain a natural equivalence 
\[
\POr{\OR}{\cC}{n}\simeq \POr{\OR}{\cC^\dbl}{n}.
\]

\begin{defn}
Let $\cC\in \calg(\Cat)$ and $\OR\in \calg(\Spcnp)$. For a pre-orientation $\omega \colon \Dual{\OR}{n} \to \cC^\times$, we denote by 
\[
\mdef{\omega^\dbl} \colon \Dual{\OR}{n}\too (\cC^\dbl)^{\times}
\] 
the pre-orientation corresponding to $\omega$ under the above equivalence. 
\end{defn}

The Fourier transforms of $\omega$ and $\omega^\dbl$ are related by the following result, where the left morphism in the commutative diagram is the one from \cref{Cdbl-in-out}.

\begin{prop}\label{fourier-dbl}
    Let $\cC \in \CAlg(\Prlst)$, let $\OR \in \CAlg(\Spcnp)$ and let $\omega \colon \Dual{\OR}{n} \to \cC^\times$ be an $\OR$-pre-orientation.
    Then, the following diagram commutes naturally in $M\in \Mod_\OR^{[0,n]}$:
    \[\begin{tikzcd}
    	{\cC[M]^\dbl} & {(\cC^{\Dual{M}{n}})^\dbl} \\
    	{\cC^\dbl[M]} & {(\cC^\dbl)^{\Dual{M}{n}}}
    	\arrow["{\Four_\omega^\dbl}", from=1-1, to=1-2]
    	\arrow["{\Four_{\omega^\dbl}}", from=2-1, to=2-2]
    	\arrow[from=2-1, to=1-1]
    	\arrow["\wr", from=1-2, to=2-2]
    \end{tikzcd}\]
\end{prop}

\begin{proof}
    First, the target of both paths in the diagram is right Kan extended along
    \[
        \Dual{\OR}{n}\colon \pt \too \Mod_{\OR}^{[0,n]}.
    \]
    Therefore, it suffices to show that the diagram commutes after evaluation at $M = \Dual{\OR}{n}$ and post-composing with the canonical augmentation $(\cC^\dbl)^{\Map(\Dual{\OR}{n},\Dual{\OR}{n})} \to \cC^\dbl$ (see \cite[Remark 3.12]{barthel2022chromatic}). Namely, we have to show that the following diagram commutes:
    \[\begin{tikzcd}
    	{\cC[\Dual{\OR}{n}]^\dbl} & {\cC^\dbl} \\
    	{\cC^\dbl[\Dual{\OR}{n}]} & {\cC^\dbl}
    	\arrow["{\varepsilon_\omega^\dbl}", from=1-1, to=1-2]
    	\arrow["{\varepsilon_{\omega^\dbl}}", from=2-1, to=2-2]
    	\arrow[from=2-1, to=1-1]
    	\arrow[Rightarrow, no head, from=1-2, to=2-2]
    \end{tikzcd}\]
    where $\varepsilon_\omega$ is the augmentation of $\cC[\Dual{\OR}{n}]$ corresponding to $\omega$ and similarly for $\omega^\dbl$. 
    To show that this square commutes, we can take the mates with respect to the adjunction
    \[
        \cC^\dbl[-]\colon\Spcn \adj \calg_{\cC^\dbl}(\Catperf)\noloc (-)^\times.
    \]
    
    By construction, the mate of the Beck--Chevalley map $\cC^\dbl[\Dual{\OR}{n}] \to \cC[\Dual{\OR}{n}]^\dbl$ is the map
    \[
        \Dual{\OR}{n}
        \too \cC[\Dual{\OR}{n}]^\times
        \too[i^{-1}] (\cC[\Dual{\OR}{n}]^\dbl)^\times.
    \]
    where the first map is the unit of the adjunction above.
    Consider now the diagram
    \[\begin{tikzcd}
    	{(\cC[\Dual{\OR}{n}]^\dbl)^\times} & {(\cC^\dbl)^\times} \\
    	{\cC[\Dual{\OR}{n}]^\times} & {\cC^\times} \\
    	{\Dual{\OR}{n}} & {(\cC^\dbl)^\times}
    	\arrow["{\varepsilon_\omega^\dbl}", from=1-1, to=1-2]
    	\arrow["{i^{-1}}", from=2-1, to=1-1]
    	\arrow[from=3-1, to=2-1]
    	\arrow["{\varepsilon_\omega}", from=2-1, to=2-2]
    	\arrow["{i^{-1}}", from=2-2, to=1-2]
    	\arrow["i", from=3-2, to=2-2]
    	\arrow["{\omega^\dbl}", from=3-1, to=3-2]
    	\arrow["\omega", from=3-1, to=2-2]
    \end{tikzcd}\]
    in $\Spcn$.
    The outer rectangle is the mate of the above square, so we need to show that this diagram commutes.
    The upper square commutes by the naturality of $i$.
    The left triangle commutes by the definition of $\varepsilon_\omega$ and the right triangle commutes by the definition of $\omega^\dbl$. 
\end{proof}

\subsubsection{Fourier and Lax Functors}

We now provide a slight generalization of \cite[Proposition 3.18]{barthel2022chromatic}.
Let $F\colon \cC \to \cD$ be a lax symmetric monoidal functor between symmetric monoidal categories.
As explained in \cref{units-lax-fun}, we have a natural map $F^\times\colon R^\times \to F(R)^\times$ for $R\in \calg(\cC)$.

\begin{defn}
    For an $\OR$-pre-orientation of $\cC$, we define an $\OR$-pre-orientation of $\Mod_{F(\unit_\cC)}(\cD)$
    \[
        \mdef{F(\omega)}\colon \Dual{\OR}{n} \too[\omega] \unit_\cC^\times \too[F^\times] F(\unit_\cC)^\times.
    \]
\end{defn}

The Fourier transforms of $\omega$ and $F(\omega)$ are compatible in the following sense.

\begin{prop}\label{fourier-lax-ambi}
    Let $F\colon \cC \to \cD$ be a lax symmetric monoidal functor between presentably symmetric monoidal categories.
    Let $n \geq 0$, let $\OR \in \CAlg(\Spcnp)$, and let $\omega\colon \Dual{\OR}{n} \to \unit_\cC^\times$ be an $\OR$-pre-orientation of $\cC$.
    Then, the following diagram in $\CAlg_{F(\unit_\cC)}(\cD)$ commutes naturally in $M \in \Mod_{\OR}^{[0,n]}$
    \[\begin{tikzcd}
    	{F(\unit_\cC[M])} & {F(\unit_\cC^{\Dual{M}{n}})} \\
    	{F(\unit_\cC)[M]} & {F(\unit_\cC)^{\Dual{M}{n}}}
    	\arrow[from=2-1, to=1-1]
    	\arrow["{F(\Four_\omega)}", from=1-1, to=1-2]
    	\arrow["{\Four_{F(\omega)}}", from=2-1, to=2-2]
    	\arrow[from=1-2, to=2-2]
    \end{tikzcd}\]
\end{prop}

\begin{proof}
    We essentially repeat the proof of \cite[Proposition 3.18]{barthel2022chromatic}, we give the details for the convenience of the reader.
    First, note that the diagram is indeed natural in $M$, as the assembly maps and Fourier transforms are natural in $M$.
    To see that it commutes, recall from \cite[Proposition 3.10]{barthel2022chromatic} that the space of natural maps
    \[
        F(\unit_\cC)[-] \too F(\unit_\cC)^{(-)}
    \]
    between functors $\Mod_{\OR}^{[0,n]} \to \CAlg_{F(\unit_\cC)}(\cD)$ is equivalent to the space of $\OR$-pre-orientations of $\Mod_{F(\unit_\cC)}(\cD)$.
    Thus, it suffices to prove that the pre-orientations corresponding to the two maps coincide.
    
    By construction, the pre-orientation corresponding to $\Four_{F(\omega)}$ is $F(\omega)$.
    The pre-orientation corresponding to the other composition is the lower composition in following diagram:
    \[\begin{tikzcd}
    	& {\unit_\cC[\Dual{\OR}{n}]^\times} & {(\unit_\cC^{\Map(\Dual{\OR}{n},\Dual{\OR}{n})})^\times} & {\unit_\cC^\times} \\
    	{\Dual{\OR}{n}} & {F(\unit_\cC[\Dual{\OR}{n}])^\times} & {F(\unit_\cC^{\Map(\Dual{\OR}{n},\Dual{\OR}{n})})^\times} & {F(\unit_\cC)^\times}
    	\arrow[from=2-1, to=1-2]
    	\arrow[from=1-2, to=1-3]
    	\arrow[from=1-3, to=1-4]
    	\arrow[from=1-4, to=2-4]
    	\arrow[from=1-2, to=2-2]
    	\arrow[from=2-2, to=2-3]
    	\arrow[from=2-3, to=2-4]
    	\arrow[from=2-1, to=2-2]
    	\arrow[from=1-3, to=2-3]
    \end{tikzcd}\]
    This diagram commutes because of the naturality of the map $F^\times\colon R^\times \to F(R)^\times$.
    Again, by construction of $\Four_\omega$, the composition from $\Dual{\OR}{n}$ to $\unit_\cC^\times$ is $\omega$, which when composed with $\unit_\cC^\times \to F(\unit_\cC)^\times$, gives $F(\omega)$.
\end{proof}

\subsubsection{Fourier and $K$-theory}

We turn to prove the compatibility of the functor $\KTnp$ with the Fourier transform. 

\begin{thm}\label{fourier-comp}
    Let $R \in \calg(\SpTn)$, let $\OR \in \CAlg(\Spcnp)$ and let $\omega$ be an $\OR$-pre-orientation of $\cMod_R \in \CAlg(\PrlTn)$ of height $n+1$.
    \begin{enumerate}[ref=(\arabic*)]    
        \item\label{fourier-comp-K} The following diagram commutes naturally in $M \in \Mod_{\OR}^{[0,n]}$:
        \[\begin{tikzcd}[column sep=large]
        	{\KTnp(R[M])} && {\KTnp(R^{\Dual{M}{n}})} \\
        	{\KTnp(R)[\Sigma M]} && {\KTnp(R)^{\Dual{\Sigma M}{n+1}}}
        	\arrow["{\KTnp(\Four_{\Omega\omega})}", from=1-1, to=1-3]
        	\arrow[from=1-3, to=2-3]
        	\arrow[from=2-1, to=1-1]
        	\arrow["{\Four_{\KTnp(\omega^\dbl)}}", from=2-1, to=2-3]
        \end{tikzcd}\]
        \item\label{M-pi-fin} When $M$ is $\pi$-finite, the vertical maps in the above diagram are isomorphisms, giving an equivalence
        \[
        \KTnp(\Four_{\Omega\omega}) \simeq \Four_{\KTnp(\omega^\dbl)}
        \]
        of maps in $\calg(\SpTnp)$.
        \item\label{Omega-omega} If $\Omega \omega$ is an orientation, then $\KTnp(\omega^\dbl)$ is an orientation. That is, for every $\pi$-finite $M \in \Mod_{\OR}^{[0,n+1]}$ the Fourier transform
        \[
            \Four_{\KTnp(\omega^\dbl)}\colon \KTnp(R)[M] \iso \KTnp(R)^{\Dual{M}{n+1}}
        \]
        is an isomorphism.
    \end{enumerate}
\end{thm}

\begin{proof}
    We begin by establishing part \ref{fourier-comp-K}.
    By \cref{fourier-dec} applied to $\cC = \cMod_R$, we have the following commutative diagram in $\CAlg(\PrlTn)$, natural in $M$:
    \[\begin{tikzcd}
    	{\cMod_{R[M]}} & {\cMod_{R^{\Dual{M}{n}}}} \\
    	{\cMod_R[\Sigma M]} & {\cMod_R^{\Dual{\Sigma M}{n+1}}}
    	\arrow["{\Four_{\Omega\omega}}", from=1-1, to=1-2]
    	\arrow["{\Four_\omega}", from=2-1, to=2-2]
    	\arrow[from=1-2, to=2-2]
    	\arrow["\wr", from=2-1, to=1-1]
    \end{tikzcd}\]
    Taking dualizable objects and pasting with the commutative square from \cref{fourier-dbl}, we get the following commutative diagram in $\CAlg(\Catperf)$:
    \[\begin{tikzcd}
    	{\cMod^\dbl_{R[M]}} & {\cMod^\dbl_{R^{\Dual{M}{n}}}} \\
    	{(\cMod_R[\Sigma M])^\dbl} & {(\cMod_R^{\Dual{\Sigma M}{n+1}})^\dbl} \\
    	{\cMod^\dbl_R[\Sigma M]} & {(\cMod^\dbl_R)^{\Dual{\Sigma M}{n+1}}}
    	\arrow["{\Four_{\Omega\omega}^\dbl}", from=1-1, to=1-2]
    	\arrow["{\Four_\omega^\dbl}", from=2-1, to=2-2]
    	\arrow[from=3-1, to=2-1]
    	\arrow["\wr", from=2-2, to=3-2]
    	\arrow["{\Four_{\omega^\dbl}}", from=3-1, to=3-2]
    	\arrow["\wr", from=2-1, to=1-1]
    	\arrow[from=1-2, to=2-2]
    \end{tikzcd}\]
    Applying $\KTnp$ and pasting the commutative square from \cref{fourier-lax-ambi} for the lax symmetric monoidal functor $\KTnp$, we get the following commutative diagram in $\CAlg(\SpTnp)$, in which the outer rectangle is the required diagram:
    \[\begin{tikzcd}
    	{\KTnp(R[M])} && {\KTnp(R^{\Dual{M}{n}})} \\
    	{\KTnp((\cMod_R[\Sigma M])^\dbl)} && {\KTnp((\cMod_R^{\Dual{\Sigma M}{n+1}})^\dbl)} \\
    	{\KTnp(\cMod^\dbl_R[\Sigma M])} && {\KTnp((\cMod^\dbl_R)^{\Dual{\Sigma M}{n+1}})} \\
    	{\KTnp(R)[\Sigma M]} && {\KTnp(R)^{\Dual{\Sigma M}{n+1}}}
    	\arrow["{\KTnp(\Four_{\Omega\omega}^\dbl)}", from=1-1, to=1-3]
    	\arrow["{\KTnp(\Four_\omega^\dbl)}", from=2-1, to=2-3]
    	\arrow[from=4-1, to=3-1]
    	\arrow[from=3-1, to=2-1]
    	\arrow["{\Four_{\KTnp(\omega^\dbl)}}", from=4-1, to=4-3]
    	\arrow["{\KTnp(\Four_{\omega^\dbl})}", from=3-1, to=3-3]
    	\arrow["\wr", from=2-3, to=3-3]
    	\arrow[from=3-3, to=4-3]
    	\arrow["\wr", from=2-1, to=1-1]
    	\arrow[from=1-3, to=2-3]
    \end{tikzcd}\]

    For part \ref{M-pi-fin}, recall that when $M$ is $\pi$-finite, the left vertical composition is an isomorphism by \cref{KTnp-R-A-Lnf} (see also \cref{KTnp-cMod}), and similarly the right vertical composition is an isomorphism by \cref{KTnp-CAlg-lim}.

    Finally, we prove part \ref{Omega-omega}.
    By assumption, $\Omega\omega$ is an orientation, so that the upper map in the diagram from part \ref{fourier-comp-K} is an isomorphism for every $\pi$-finite $M \in \Mod_{\OR}^{[0,n]}$.
    By part \ref{M-pi-fin}, we conclude that for every $\pi$-finite $N := \Sigma M \in \Mod_{\OR}^{[1,n+1]}$ the Fourier transform
    \[
        \Four_{\KTnp(\omega^\dbl)}\colon \KTnp(R)[N] \iso \KTnp(R)^{\Dual{N}{n+1}}
    \]
    is an isomorphism.
    It remains to show that this holds for every $\pi$-finite $N \in \Mod_{\OR}^{[0,n+1]}$.
    This follows from the case of $1$-connective $N$ as in the proof of \cite[Theorem 5.15]{barthel2022chromatic}, which we briefly recall for the readers convenience.
    Consider the exact sequence
    \[
        \tau_{\geq 1} N \too N \too \tau_{\leq 0} N
        \qin \Mod_{\OR}^{[0,n+1]}.
    \]
    By the $1$-connective case, the Fourier transform is an isomorphism at $\tau_{\geq 1} N$.
    By the duality from \cite[Proposition 4.9]{barthel2022chromatic}, the Fourier transform is also an isomorphism at $\pi$-finite objects of $\Mod_{\OR}^{[0,n]}$, and in particular at $\tau_{\leq 0} N$.
    Since $\cMod_{\KTnp(R)}$ is semiadditive, the underlying space of $\tau_{\leq 0} N$, which is a finite set, is $\cMod_{\KTnp(R)}$-affine by \cite[Example 2.35]{barthel2022chromatic}.
    Thus, by \cite[Proposition 4.12]{barthel2022chromatic}, the Fourier transform is an isomorphism at $N$, as required.
\end{proof}

\subsection{Kummer Theory} 

Let $R \in \calg(\SpTn)$.
As shown in \cite[Proposition 4.32]{barthel2022chromatic},
for $\OR \in \CAlg(\Spcnp)$, an $\OR$-orientation $\omega\colon \Dual{\OR}{n} \to R^\times$ provides a natural equivalence 
\[
    \GalExt{R; \SpTn}{\Omega^\infty M}
    \simeq \Map_{\Spcn}(\Dual{M}{n}, R^\times)
\]
for $\pi$-finite $M\in \Mod_\OR^{[1,n]}$, 
referred to as the ``Kummer equivalence''. 
Here, the left-hand side is the space of local systems $\Omega^\infty M\to \calg_R(\cC)$ which are Galois extensions of $R$ (with Galois group $\Omega^\infty \Sigma^{-1}M$). 

Our goal is to compare the Kummer equivalences of $R$ and $\KTnp(R)$.
We first relate their sources and targets, starting with the sources.
As we have shown in \cref{galois-comp}, the functor
\[
    \KTnp \colon \calg_R(\SpTn) \to \calg_{\KTnp(R)}(\SpTnp)
\]
carries Galois extensions to Galois extensions for $n$-finite $p$-groups, and in particular restricts to a natural map 
\[
    \KTnp\colon \GalExt{R; \SpTn}{\Omega^\infty M} \too \GalExt{\KTnp(R); \SpTnp}{\Omega^\infty M}.
\]
We now consider the targets.
As explained in \cref{units-lax-fun}, since $K\colon \Catperf \to \Spcn$ is a lax symmetric monoidal functor, there is a natural map $\cC^\times \to K(\cC)^\times$ for $\cC \in \CAlg(\Catperf)$.
Moreover, the equivalence 
$\unit_\cC^\times \simeq \Omega\cC^\times$
corresponds to a natural map
$\Sigma \unit_\cC^\times \to  \cC^\times$
of connective spectra.
Composing them, we get a natural map
\[
    \Sigma \unit_\cC^\times \too K(\cC)^\times.
\]
For any $S \in \CAlg_R(\SpTn)$, applying the above to $\cC = \cMod_S^\dbl$ and post-composing with the map $K\to \KTnp$, we obtain a natural map 
\[
    \Sigma S^\times \too \KTnp(S)^\times.
\]
Consider also the two adjunctions
\begin{gather*}
    R[-]\colon \Spcn \adj \calg_R(\SpTn)\noloc (-)^\times
    \\
    \KTnp(R)[-]\colon \Spcn \adj \calg_{\KTnp(R)}(\SpTnp)\noloc (-)^\times
\end{gather*}
These adjunctions are compatible via the map above in the following sense:

\begin{lem}\label{kummer-bottom}
    Let $R \in \CAlg(\SpTn)$ and let $M \in \Spcn$.
    Then, the following diagram commutes naturally in $S \in \CAlg_R(\SpTn)$
    \[\begin{tikzcd}
    	{\Map(R[M],S)} & {\Map(\KTnp(R[M]),\KTnp(S))} & {\Map(\KTnp(R)[\Sigma M],\KTnp(S))} \\
    	{\Map(M,S^\times)} & {\Map(\Sigma M, \Sigma S^\times)} & {\Map(\Sigma M,\KTnp(S)^\times)}
    	\arrow["\wr"', from=1-1, to=2-1]
    	\arrow[from=1-1, to=1-2]
    	\arrow["\sim", from=1-2, to=1-3]
    	\arrow["\wr", from=1-3, to=2-3]
    	\arrow[from=2-1, to=2-2]
    	\arrow[from=2-2, to=2-3]
    \end{tikzcd}\]
    where the upper right isomorphism is the one from \cref{KTnp-R-A-Lnf}.
\end{lem}

\begin{proof}
    Note that the functor
    \[
        \Map(R[M], -)\colon \calg_R(\SpTn) \to \Spc
    \]
    in the upper left corner of the diagram is co-represented by $R[M]$, so it suffices to prove the commutativity in the special case $S=R[M]$, and after evaluating both of the compositions at $\Id_{R[M]}$. 

    The value of the upper-right composite at $\Id_{R[M]}$ is the map
    \[
        \Sigma M \too \KTnp(R)[\Sigma M]^\times  \iso \KTnp(R[M])^\times,
    \]
    while the value of the left-bottom composite is the map 
    \[
        \Sigma M \too \Sigma R[M]^\times \too \KTnp(R[M])^\times.     
    \]
    In light of \cref{KTnp-cMod}, it thus suffices to show that the following diagram commutes:
    \[\begin{tikzcd}
    	&& {\KTnp(\cMod_R^\dbl)[\Sigma M]^\times} \\
    	{\Sigma M} & {(\cMod_R^\dbl[\Sigma M])^\times} & {\KTnp(\cMod_R^\dbl[\Sigma M])^\times} \\
    	{\Sigma R[M]^\times} & {(\cMod_{R[M]}^\dbl)^\times} & {\KTnp(\cMod_{R[M]}^\dbl)^\times}
    	\arrow[from=2-1, to=3-1]
    	\arrow[from=3-1, to=3-2]
    	\arrow[from=2-2, to=3-2]
    	\arrow[from=2-1, to=2-2]
    	\arrow[from=2-2, to=2-3]
    	\arrow[from=3-2, to=3-3]
    	\arrow["\wr", from=1-3, to=2-3]
    	\arrow["\wr", from=2-3, to=3-3]
    	\arrow[curve={height=-12pt}, from=2-1, to=1-3]
    \end{tikzcd}\]
    The commutativity of the upper (triangle-shaped) square follows from \cref{units-group-algebras}.
    The right square commutes by naturality of the map $\cC^\times \to \KTnp(\cC)^\times$, so we are left with the left square.
    Finally, using \cref{Mod-R-A}, we embed this square as the left square in the following diagram:
    \[\begin{tikzcd}
    	{\Sigma M} & {(\cMod_R^\dbl[\Sigma M])^\times} & {(\cMod_R[\Sigma M]^\dbl)^\times} & {(\cMod_R[\Sigma M])^\times} \\
    	{\Sigma R[M]^\times} & {(\cMod_{R[M]}^\dbl)^\times} && {\cMod_{R[M]}^\times}
    	\arrow[from=1-1, to=1-2]
    	\arrow[from=1-1, to=2-1]
    	\arrow[from=2-1, to=2-2]
    	\arrow[from=1-2, to=1-3]
    	\arrow[from=1-3, to=1-4]
    	\arrow["\sim", from=2-2, to=2-4]
    	\arrow[from=1-4, to=2-4]
    	\arrow[from=1-2, to=2-2]
    	\arrow[from=1-3, to=2-2]
    \end{tikzcd}\]
    Recall that every invertible object is dualizable, which shows that the bottom right horizontal map is an isomorphism.
    Therefore, to show that the left square commutes, it suffices to show that the outer rectangle and the right rectangle commute.
    The commutativity of the outer square follows from the construction of the right map as in \cite[Proposition 5.11]{barthel2022chromatic} (see also \cref{variants-remark}).
    The middle triangle commutes since the vertical map is defined as the composition of the other two.
    Finally, the right (trapezoid-shaped) square commutes because of the naturality of the embedding $\cC^\dbl \into \cC$.
\end{proof}

Recall that we have a map $\Sigma R^\times \to \KTnp(R)^\times$.
Post-composing with this map gives a map
\begin{align*}
    \Map(\Dual{M}{n}, R^\times)
    &\too \Map(\Sigma \Dual{M}{n}, \Sigma R^\times)\\
    &\iso \Map(\Dual{M}{n+1}, \Sigma R^\times)\\
    &\too \Map(\Dual{M}{n+1}, \KTnp(R)^\times).
\end{align*}
Also recall that for any $\OR$-orientation $\omega$ of $R$, \cref{fourier-comp} gives an $\OR$-orientation $\KTnp(\omega)$ of $\KTnp(R)$.
We can now compare the Kummer equivalences corresponding to these orientations.

\begin{thm}\label{kummer-comp}
    Let $\OR\in \calg(\Spcnp)$, let $R\in \calg(\Sp_{\Tn})$ and let $\omega$ be an $\OR$-orientation of $\cMod_R$ of height $n+1$.
    Then, the following diagram commutes for any $\pi$-finite $M\in \Mod_{\OR}^{[1,n]}$
    \[\begin{tikzcd}
    	{\GalExt{R; \SpTn}{\Omega^\infty M}} & {\GalExt{\KTnp(R); \SpTnp}{\Omega^\infty M}} \\
    	{\Map(\Dual{M}{n},R^\times)} & {\Map(\Dual{M}{n+1},\KTnp(R)^\times)}
    	\arrow[from=1-1, to=1-2]
    	\arrow[from=2-1, to=2-2]
    	\arrow["\wr"', from=1-1, to=2-1]
    	\arrow["\wr", from=1-2, to=2-2]
    \end{tikzcd}\]
\end{thm} 

\begin{proof}
    We exhibit the required commutative square as the pasting of the following squares:
    \[\begin{tikzcd}
    	{\GalExt{R; \SpTn}{\Omega^\infty M}} & {\GalExt{\KTnp(R); \SpTnp}{\Omega^\infty M}} \\
    	{\Map(R^{\Omega^\infty M},R)} & {\Map(\KTnp(R)^{\Omega^\infty M},\KTnp(R))} \\
    	{\Map(R[\Dual{M}{n}],R)} & {\Map(\KTnp(R)[\Dual{M}{n+1}],\KTnp(R))} \\
    	{\Map(\Dual{M}{n},R^\times)} & {\Map(\Dual{M}{n+1},\KTnp(R)^\times)}
    	\arrow["\wr"', from=2-1, to=3-1]
    	\arrow["\wr", from=2-2, to=3-2]
    	\arrow[from=2-1, to=2-2]
    	\arrow["\wr"', from=3-1, to=4-1]
    	\arrow[from=3-1, to=3-2]
    	\arrow["\wr", from=3-2, to=4-2]
    	\arrow["\wr"', from=1-1, to=2-1]
    	\arrow["\wr", from=1-2, to=2-2]
    	\arrow[from=1-1, to=1-2]
    	\arrow[from=4-1, to=4-2]
    \end{tikzcd}\]
    The bottom square is the commutative square of \cref{kummer-bottom} applied to the ring $R$ and connective spectrum $\Dual{M}{n}$.
    The upper horizontal morphism comes from \cref{galois-comp}.
    The second and third horizontal morphisms are obtained by applying $\KTnp$ and using the isomorphisms of \cref{KTnp-R-A-Lnf} and \cref{KTnp-CAlg-lim}.
    The left upper vertical isomorphism is given by taking the limit over $\Omega^\infty M$ to a Galois extension $R \to S$, giving a map $R^{\Omega^\infty M} \to \lim_{\Omega^\infty M} S \simeq R$, and similarly for the right upper vertical isomorphism.
    The upper square commutes by naturality of limits.
    The middle vertical isomorphisms are the Fourier transforms, which are isomorphisms as in \cite[Theorem 7.33]{barthel2022chromatic}, and the middle square commutes because of \cref{fourier-comp}.
\end{proof}

\section{Cyclotomic Hyperdescent}

In this section we explain the intimate relationship between cyclotomic completion and hyperdescent, and use it to rephrase the results of the previous section in this language.
These should be compared with the general results on hyperdescent for algebraic $K$-theory as in \cite{clausen2021hyperdescent}. We refer the reader also to \cite{mor2023picard} for a discussion of hypersheaves and hyperdescent in the chromatic context. 

The starting point of this section is the observation that for $X \in \SpTn$ the $\ZZ_p^\times$-action on the infinite cyclotomic extension $\cyc{X}{p^\infty}{n}$ is \emph{continuous}.
As a definition of continuous $\ZZ_p^\times$-actions, we consider sheaves on the site $\cT_{\ZZ_p^\times}$ of continuous finite $\ZZ_p^\times$-sets (see for example \cite[Definition 4.1]{clausen2021hyperdescent}).
Indeed, the intermediate cyclotomic extensions $\cyc{X}{p^r}{n}$ arrange into the \emph{cyclotomic sheaf} $\cyc{X}{p^{(-)}}{n}$, whose stalk is $\cyc{X}{p^\infty}{n}$.
In this language, cyclotomic redshift (\cref{cyclo-main}) says that $\KTnp$ takes cyclotomic sheaves to cyclotomic sheaves (see \cref{cyclosheaf-shift}).

The main connection between cyclotomic completion and hyperdescent is established in \cref{cyclo-sheaf-galois}, showing that $X$ is cyclotomically complete if and only if the cyclotomic sheaf $\cyc{X}{p^{(-)}}{n}$ is a hypersheaf.
As explained in \cite[\S7.3]{barthel2022chromatic}, by the Devinatz--Hopkins theorem \cite{devinatz2004homotopy}, all $\Kn$-local spectra are cyclotomically complete.
Thus, we deduce that $\KKnp(\cyc{R}{p^{(-)}}{n})$ is a hypersheaf for any $R \in \Alg(\SpTn)$, namely that $\Knp$-local algebraic $K$-theory satisfies hyperdescent along the cyclotomic tower (see \cref{Knp-hyper}).

As continuous finite $\ZZ_p$-sets are simpler than continuous finite $\ZZ_p^\times$-sets, we consider the corresponding restriction of the cyclotomic sheaf, whose values are the extensions $R\iwa{r}{n}$ of \cref{iwa-def}.
We leverage the connection between cyclotomic completion and hyperdescent to reinterpret and strengthen \cref{iwa-comp} to \cref{K-hyper-cyclo}, stating that the map
\[
    \KTnp(R\iwa{(-)}{n}) \too \KTnp(R\iwa{\infty}{n})^{h(p^{(-)}\ZZ)}
\]
exhibits the target both as the hypersheafification and as the level-wise cyclotomic completion of the source.

\subsection{Continuous Group Actions}

In this subsection we study continuous group actions of profinite groups $G$, and in particular continuous Galois extensions.
To achieve this, we begin by recalling the general setup of (hyper)sheaves on a site, and then specialize to the site $\cT_G$ of continuous finite $G$-sets.
We say that a sheaf $\sR$ of commutative algebras is a continuous $G$-Galois extension if its value at any finite $G$-set is a faithful Galois extension.
The two main results are \cref{Prof_Galois_Descent}, which is a form of Galois descent for continuous $G$-Galois extensions, and \cref{hyper-R-loc-dd}, which shows that for modules over $\sR$, hypersheafification coincides with level-wise localization at the stalk of $\sR$.

\subsubsection{Sheaves and Hypersheaves}

We begin by recalling the relevant setup of (hyper)sheaves and their formal properties.
We refer the reader to \cite[\S6.5]{HA} for a comprehensive study of sheaves and hypersheaves and to the discussion in \cite[\S1]{haine2020homotopy} for a short overview of (hyper)sheaves with coefficients.

For a site $\cT$  consider the $\infty$-topos $\Shv(\cT)$ of sheaves on $\cT$.
For a presentable category $\cC$ we let
\[
    \Shv(\cT; \cC) := \Shv(\cT) \otimes \cC \qin \Prl
\]
denote the category of $\cC$-valued sheaves on $\cT$.
This can alternatively be defined as the full subcategory of presheaves that satisfy descent, namely the sheaf condition (see \cite[Remark 1.3.1.6]{SAG}).

We now wish to define the full subcategory of hypersheaves.
In the case of $\Spc$-valued sheaves, they can be defined intrinsically -- we denote by $\Shv^\hyp(\cT) \subset \Shv(\cT)$ the full subcategory of hypercomplete objects, namely those sheaves that are local with respect to $\infty$-connected morphisms (see \cite[\S6.5.2]{HTT}).
As above, we define
\[
    \Shv^\hyp(\cT; \cC) := \Shv^\hyp(\cT) \otimes \cC \qin \Prl.
\]
This can alternatively be defined as the full subcategory of presheaves that satisfy \emph{hyperdescent}, the analogue of the sheaf condition for hypercovers (this follows from \cite[Corollary 6.5.3.13]{HTT} and the formula for the Lurie tensor product \cite[Proposition 4.8.1.17]{HA}).
The inclusion $\Shv^\hyp(\cT; \cC) \sseq \Shv(\cT; \cC)$ admits a left adjoint
\[
    (-)^\hyp\colon \Shv(\cT; \cC) \too \Shv^\hyp(\cT; \cC)
\]
called \emph{hypersheafification}.
We note that pushforwards of sheaves preserve hypersheaves (see for example \cite[Proposition 6.5.2.13]{HTT}).
Observe that when $\cC$ is presentably symmetric monoidal, then the category of (hyper)sheaves is presentably symmetric monoidal as well, and the hypersheafification functor is symmetric monoidal.

We shall be primarily interested in sites with a finitary Grothendieck topology in the sense of \cite[Definition A.3.1.1]{SAG}.
One of the key properties of hypersheaves is the Deligne completeness theorem saying that equivalences can be checked on stalks, a version of which we now recall:

\begin{prop}\label{deligne}
    Let $\cT$ be a category with pullbacks endowed with a finitary Grothendieck topology, and let $\cC$ be a compactly generated presentable category.
    Then $\Shv^\hyp(\cT; \cC)$ has enough points in the sense that the collection of functors 
    \[
        \cC \otimes f^*\colon \Shv^\hyp(\cT; \cC) \to \cC,
    \]
    where $f^*$ ranges over the points of $\Shv^\hyp(\cT)$, is jointly conservative.
\end{prop}

\begin{proof}
    By \cite[Lemma 2.12]{haine2021nonabelian}, it suffices to prove the result for $\cC = \Spc$.
    By \cite[Proposition A.3.1.3]{SAG}, the $\infty$-topos $\Shv(\cT)$ is locally coherent, whence by \cite[Proposition A.2.2.2]{SAG} so is $\Shv^\hyp(\cT)$.
    The claim then follows from the Deligne completeness theorem \cite[Theorem A.4.0.5]{SAG}.
\end{proof}

By \cite[p.\ 669]{HTT}, hypersheafification is a geometric morphism.
Thus, by composing, any point of $\Shv^\hyp(\cT)$ gives a point of $\Shv(\cT)$.
Therefore, under the assumptions of \cref{deligne}, given a map $\sF \to \sG$ in $\Shv(\cT; \cC)$ from a sheaf to a hypersheaf, to check that it exhibits $\sG$ as the hypersheafification of $\sF$ it suffices to check that it is an isomorphism on all points coming from $\Shv^\hyp(\cT)$.

\subsubsection{Sheaves on Continuous $G$-sets}

We now review the setup of continuous $G$-actions where $G$ is a profinite group, as developed in the $\infty$-categorical setting for example in \cite[\S4.1]{clausen2021hyperdescent}.
We denote by $\cT_G$ the site of continuous finite $G$-sets, endowed with the Grothendieck topology generated by the jointly surjective finite families of maps, which is finitary by construction.
For a sheaf $\sF \in \Shv(\cT_G; \cC)$ and an open subgroup $U \le G$, observe that $\sF(G/U)$ has a residual $N_G(U)/U$-action.
Furthermore, there is an equivalence (see \cite[Construction 4.5]{clausen2021hyperdescent})
\[
    \Shv(\cT_G; \cC) \simeq 
    \colim_{U \unlhd G \text{ open }} \cC^{B(G/U)} \qin \Prl.
\]
In other words, the data of a sheaf $\sF$ is precisely the data of $\sF(G/U) \in \cC^{B(G/U)}$ together with coherent compatibility maps. In particular, by \cite[Theorem 6.3.3.1]{HTT} the $\infty$-topos 
$\Shv(\cT_G; \Spc)$ can be presented as a filtered limit formed in the category of $\infty$-topoi: 
\[
    \Shv(\cT_G; \Spc) \simeq 
    \invlim_{U \unlhd G \text{ open }} \Spc^{B(G/U)}. 
\] 
By \cite[Remark 6.3.5.10]{HTT} applied to the $\infty$-topos of spaces, the space of points of $\Spc^{B(G/U)}$ is $B(G/U)$. Consequently, the space of points of $\Shv(\cT_G; \Spc)$ is the filtered limit $\invlim B(G/U)$, which is a connected space with a canonical basepoint. We denote the distinguished point of this topos by 
\[
e_*\colon \Spc \too \Shv(\cT_G; \Spc).
\] 
For $\cC\in \Prl$, the corresponding stalk functor $e^*\colon \Shv(\cT_G; \cC) \to \cC$ is given by
\[
    e^* \sF \simeq \colim_{U \unlhd G \text{ open }} \sF(G/U).
\]
Since there is a canonical map $BG \to \invlim B(G/U)$, the stalk has a canonical $G$-action (where $G$ is regarded as a discrete group), namely $e^*$ lifts to a functor $\oeu\colon \Shv(\cT_G; \cC) \to \cC^{BG}$.
The right adjoint $\oe_*\colon \cC^{BG} \to \Shv(\cT_G; \cC)$ sends $X \in \cC^{BG}$ to the sheaf whose values are
\[
    (\oe_* X)(G/U) = X^{hU}
    \qin \cC^{B(G/U)}.
\]
Since the right adjoint to the forgetful $\cC^{BG} \to \cC$ is given by $X \mapsto \coind{G}{X}$, we see that the right adjoint of $e^*$ sends $X \in \cC$ to the sheaf whose values are
\[
    e_* X(G/U) = \oe_*\big(\coind{G}{X}\big)(G/U) = \big(\coind{G}{X}\big)^{hU} \simeq \coind{G/U}{X}
    \qin \cC^{B(G/U)}.
\]
Also note that every object of $\cC^{BG}$ is a hypersheaf since $\Spc^{BG}$ is hypercomplete, and since $\oe_*$ is a geometric morphism it sends hypersheaves to hypersheaves (see \cite[Proposition 6.5.2.13]{HTT}).

\subsubsection{Continuous Galois Extensions and Hypersheafification}

We now consider Galois extensions with respect to profinite groups using the above setting.

\begin{defn}
    Let $G$ be a profinite group, let $\cC \in \CAlg(\Prl)$ be semiadditive and let $\sR \in \CAlg(\Shv(\cT_G; \cC))$.
    We say that $\sR$ is a \tdef{continuous $G$-Galois extension} if for every open normal subgroup $U \unlhd G$ the object $\sR(G/U) \in \CAlg(\cC)^{B(G/U)}$ is a \emph{faithful} Galois extension.
\end{defn}

We observe that a collection of compatible faithful Galois extensions automatically assembles into a sheaf in the following sense:

\begin{prop}\label{Galois_Sheaf}
    Let $\sF\colon \cT_G^\op \to \CAlg(\cC)$ be a finite product preserving functor such that $\sF(G/U)$ is a faithful $G/U$-Galois extension for any open normal subgroups $U \unlhd G$.
    Then $\sF$ satisfies the sheaf condition, and hence is a continuous $G$-Galois extension. 
\end{prop}

\begin{proof}
    Since $\sF$ preserves finite products, it remains to show that for every inclusion of open normal subgroups $U \unlhd U' \unlhd G$ the canonical map
    \[
        \sF(G/U') \too \sF(G/U)^{h(U'/U)}
        \qin \CAlg(\cC^{B(G/U')})
    \]
    is an isomorphism.
    This map makes the target into a commutative $\sF(G/U')$-algebra.
    Since $\sF(G/U')$ is a faithful Galois extension, \cref{Galois_Descent} implies that
    \[
        (-)^{h(G/U')}\colon \CAlg_{\sF(G/U')}(\cC^{B(G/U')}) \iso \CAlg(\cC)
    \]
    is an equivalence, and in particular is conservative.
    Therefore, it suffices to check that
    \[
        \sF(G/U')^{h(G/U')} \too (\sF(G/U)^{h(U'/U)})^{h(G/U')} \simeq \sF(G/U)^{h(G/U)}
        \qin \CAlg(\cC)
    \]
    is an isomorphism.
    Indeed, by assumption $\sF(G/U')$ and $\sF(G/U)$ are $G/U'$- and $G/U$-Galois extensions respectively, and in particular satisfy the first Rognes--Galois condition. Namely, the above map identifies with the identity map of the unit $\unit_\cC$ and so, in particular, is an isomorphism. 
\end{proof}

For every continuous $G$-Galois extension $\sR$ and $X \in \cC$, we have the presheaf $X \otimes \sR$ formed by tensoring $\sR$ with $X$ level-wise.
We show that it is in fact a sheaf, and generalize the Galois descent result of \cref{Galois_Descent} to profinite groups.

\begin{prop}\label{Prof_Galois_Descent}
    Let $G$ be a profinite group, let $\cC \in \CAlg(\Prl)$ be semiadditive and let $\sR$ be a continuous $G$-Galois extension.
    Then, there is a symmetric monoidal equivalence
    \[
        - \otimes \sR\colon \cC \adj \Mod_{\sR}(\Shv(\cT_G; \cC))\noloc (-)(G/G)
    \]
\end{prop}

\begin{proof}
    Recall that
    \[
        \Shv(\cT_G; \cC) \simeq 
        \colim_{U \unlhd G \text{ open }} \cC^{B(G/U)} \qin \Prl.
    \]
    By \cref{Galois_Descent}, for every open normal subgroup $U \unlhd G$, taking homotopy fixed points induces a symmetric monoidal equivalence
    \[
        (-)^{h(G/U)}\colon \Mod_{\sR(G/U)}(\cC^{B(G/U)}) \iso \cC.
    \]
    Also, recall that $\Mod_{(-)}(-)$ is a symmetric monoidal left adjoint functor from $\PrlCAlg$ to $\CAlg(\Prl)$
    (\cite[Theorem 4.8.5.11]{HA}) and hence commutes with colimits.
    Finally, we obtain the required equivalence in $\CAlg(\Prl)$:
    \begin{align*}
        \Mod_{\sR}(\Shv(\cT_G; \cC))
        &\simeq \Mod_{\sR}(\colim_{U \unlhd G \text{ open }} \cC^{B(G/U)})\\
        &\simeq \colim_{U \unlhd G \text{ open }} \Mod_{\sR(G/U)}(\cC^{B(G/U)})\\
        &\simeq \colim_{U \unlhd G \text{ open }} \cC\\
        &\simeq \cC.
    \end{align*}
\end{proof}

For a continuous $G$-Galois extension $\sR$ as above, consider the stalk $R := e^* \sR \in \CAlg(\cC)$.
We denote by $L_R\colon \cC \to L_R\cC$ the Bousfield localization with respect to $R$.
Our next goal is to show that, under some hypothesis, the hypersheafification of a sheaf which is a module over $\sR$ is given by applying $L_R$ level-wise.
For that purpose, we begin with the following lemma:

\begin{lem}\label{mod-sheaf}
    Let $G$ be a profinite group, let $F\colon \cC \to \cD \in \CAlg(\Prl)$ with $\cC$ and $\cD$ semiadditive, and let $\sR$ be a continuous $G$-Galois extension.
    For any $\sR$-module sheaf $\sM \in \Mod_{\sR}(\Shv(\cT_G; \cC))$, the presheaf $F(\sM)$ obtained by applying $F$ level-wise is a sheaf, and $F(\sR)$ is a continuous $G$-Galois extension.
    This assembles into a functor
    \[
        F\colon \Mod_{\sR}(\Shv(\cT_G; \cC)) \too \Mod_{F(\sR)}(\Shv(\cT_G; \cD)).
    \]
\end{lem}

\begin{proof}
    Since $F$ is a colimit preserving functor between semiadditive categories, it also preserves finite products, so that $F(\sR)$ preserves products.
    By \Cref{colim_sym_mon_preserve_Gal}, the functor $F$ sends (finite) Galois extensions to Galois extensions, so that $F(\sR(G/U))$ is a $G/U$-Galois extension for every open normal subgroup $U \unlhd G$.
    Thus, \Cref{Galois_Sheaf} implies that $F(\sR)$ is a continuous $G$-Galois extension.

    Now, since the symmetric monoidal structure on $\PSh(\cT_G; \cC)$ is compatible with the sheafification $\PSh(\cT_G; \cC) \to \Shv(\cT_G; \cC)$, and $\sR$ is a sheaf, we get an induced localization $\Mod_\sR(\PSh(\cT_G; \cC)) \to \Mod_\sR(\Shv(\cT_G; \cC))$, and similarly for $F(\sR)$ and $\cD$.
    Post-composition with $F$ induces a symmetric monoidal functor
    \[
        F\colon \PSh(\cT_G; \cC) \too \PSh(\cT_G; \cD),
    \]
    and in particular sends $\sR$-modules to $F(\sR)$-modules.
    Since $F\colon \cC \to \cD$ is symmetric monoidal we see that the outer square in the following diagram commutes:
    \[\begin{tikzcd}
    	\cC && {\Mod_{\sR}(\Shv(\cT_G; \cC))} && {\Mod_{\sR}(\PSh(\cT_G; \cC))} \\
    	\cD && {\Mod_{F(\sR)}(\Shv(\cT_G; \cC))} && {\Mod_{F(\sR)}(\PSh(\cT_G; \cD))}
    	\arrow["F"', from=1-1, to=2-1]
    	\arrow["{- \otimes \sR}", from=1-1, to=1-3]
    	\arrow["{- \otimes F(\sR)}"', from=2-1, to=2-3]
    	\arrow[hook, from=1-3, to=1-5]
    	\arrow[hook, from=2-3, to=2-5]
    	\arrow[dashed, from=1-3, to=2-3]
    	\arrow["F", from=1-5, to=2-5]
    \end{tikzcd}\]
    \cref{Prof_Galois_Descent} shows that the horizontal morphisms factor as depicted in the diagram, thus the right vertical morphism
    \[
        F\colon \Mod_\sR(\PSh(\cT_G; \cC)) \too \Mod_{F(\sR)}(\PSh(\cT_G; \cD))
    \]
    restricts to the full subcategories of sheaves in the source and target, giving the dashed morphism in the diagram.
\end{proof}

\begin{prop}\label{LR-hyper}
    Let $G$ be a profinite group of finite virtual $p$-cohomological dimension, and let $\cC \in \CAlg(\Prlst)$ be $p$-local.
    Let $\sR$ be a continuous $G$-Galois extension with stalk $R := e^*\sR$.
    For every $\sM \in \Mod_{\sR}(\Shv(\cT_G; \cC))$, the presheaf $L_R \sM$ is a hypersheaf and the map
    \[
        \sM \too L_R \sM
    \]
    exhibits the target as the hypersheafification of the source.
\end{prop}

\begin{proof}
    Since the inclusion $L_R\cC \sseq \cC$ is limit preserving, for an $L_R\cC$-valued presheaf, the (hyper)sheaf condition is the same whether we view it as valued in $\cC$ or in $L_R\cC$. Hence, in either interpretation, we get by \cref{mod-sheaf} that $L_R \sM$ is a sheaf.
    
    We now show that $L_R \sM$ is a hypersheaf. We begin by reducing to $\sM = \sR$.
    By \cite[Corollary 4.28]{clausen2021hyperdescent}, hypersheaves form a smashing localization of $\Shv(\cT_G; \Sp_{(p)})$.
    Consequently, the same holds for $\Shv(\cT_G; L_R\cC)$ since smashing localizations are closed under tensor products in $\CAlg(\Prl)$.
    Consequently, since $L_R \sM$ is a module over $L_R \sR$, it suffices to prove that the latter is a hypersheaf. 

    We shall reduce it to showing that the level-wise tensor product $R \otimes \sR$ is a hypersheaf as follows. The sheaf $L_R \sR$ is a hypersheaf if and only if the map
    \[
        L_R \sR \too (L_R \sR)^\hyp 
    \]
    is an isomorphism, where the hypersheafification is formed in $\Shv(\cT_G;L_R\cC)$. 
    Since the functor 
    \[
        R\otimes - \colon L_R \cC \too \Mod_R(\cC)
    \]
    is conservative, we can check this after tensoring level-wise with $R$. 
    We proceed by showing that this operations sends both sides to \textit{sheaves}, and moreover $R \otimes (L_R \sR)^\hyp$ is the hypersheafification of $R \otimes L_R\sR$.
    Recall that the category of $L_R \cC$-valued sheaves is $L_R \cC$-linear, with the action given by tensoring with $X \in L_R \cC$ level-wise and sheafifying, and since hypersheafification is smashing, it commutes with this action.
    That is, for a sheaf $\sF$, the hypersheaf $(X \otimes \sF)^\hyp$ is given by the sheafification of $X \otimes \sF^\hyp$.
    Now, observe that both $L_R \sR$ and $ (L_R \sR)^\hyp$ are $\sR$-modules. Thus, by \cref{mod-sheaf}, applying 
    \[
        R \otimes -\colon 
        L_R \cC \too
        \Mod_R(L_R \cC)
    \] 
    level-wise sends them to sheaves (note that the tensor product in $L_R \cC$ is obtained by applying $L_R$ to the tensor product in $\cC$, but $R$-modules are already $R$-local). So does applying the limit preserving right adjoint of the above displayed functor.
    As a result, both presheaves $R \otimes L_R \sR$ and $R \otimes (L_R \sR)^\hyp$ are sheaves, and the latter is isomorphic to $(R \otimes L_R \sR)^\hyp$.
    Moreover, since the map $\sR(G/U) \to L_R \sR(G/U)$ is an $R$-equivalence for every open normal subgroup $U$ of $G$, we also get that $R \otimes \sR \simeq R \otimes L_R \sR$.
    That is, it suffices to show that $R \otimes \sR$ is a hypersheaf.

    By the $\sR(G/U)$-algebra structure of $R$, we have a map 
    \[
        R \otimes \sR(G/U) \iso  
        R \otimes_{\sR(G/U)} (\sR(G/U)\otimes \sR(G/U)) \too
        R \otimes_{\sR(G/U)} \coind{G/U}{\sR(G/U)} \iso
        \coind{G/U}{R}
    \]
    depending functorially on $U$, which is an isomorphism by virtue of the second Rognes--Galois condition for $\sR(G/U)$.
    That is, we have
    \[
    R \otimes \sR \simeq e_* R \simeq \oe_* \coind{G}{R}.
    \]
    Recalling that $\oe_*$ lands in hypersheaves, we conclude that $R \otimes \sR$ is indeed a hypersheaf, concluding the proof that $L_R \sM$ is a hypersheaf.

    Finally, we need to check that $\sM \to L_R \sM$ exhibits the target as the hypersheafification of the source.
    Since the target is a hypersheaf, by \cref{deligne} it suffices to check that the map is an isomorphism on stalks, namely that $e^* \sM \to e^* L_R \sM$ is an isomorphism, or equivalently that its cofiber is zero.
    On the one hand, since $e^*$ is symmetric monoidal, both the source and the target are modules over $R := e^*\sR$, so in particular they are $R$-local and hence so is the cofiber. On the other hand, for each open normal subgroup $U$ of $G$, the cofiber of $\sM(G/U) \to L_R\sM(G/U)$ is $R$-acyclic, and as the cofiber of $e^*\sM \to e^*L_R\sM$ is given by their colimit, it is itself $R$-acyclic.
\end{proof}

As a sheaf $\sF \in \Shv(\cT_G; \cC)$ can be thought of as an object $X = e^*\sF \in \cC$ endowed with a ``continuous action'' of the profinite group $G$, the value $\cF(G/U)$ is the ``continuous $U$-fixed points'' of said action of $G$ on $X$. We shall hence introduce a suggestive notation, which is less likely to cause confusion now that we are done with the technical discussion above. 

\begin{notation}
    For $X \in \Shv(\cT_G; \cC)$, we shall abuse notation and denote by $X$ also the underlying object $e^*X$ and for every open $U \leq G$, denote by $X^{hU}$ the value of $X$ at $G/U \in \cT_G$.
\end{notation}

\subsubsection{Continuous $\ZZ_p$-actions}

The case $G = \ZZ_p$ is notably easier to analyze than the general case for two reasons.
First, because the transitive continuous finite $\ZZ_p$-sets are the orbits $\ZZ/p^r$ for $r \geq 0$, making (pre)sheaves simpler to handle.
Particularly, we have a sequential colimit presentation of the category of sheaves
\[
    \Shv(\cT_{\ZZ_p}; \cC) \simeq 
    \colim \cC^{B(\ZZ/p^r)} \qin \Prl,
\]
i.e.\ the data of a sheaf is the data of local systems $X^{h(p^r\ZZ)} \in \cC^{B(\ZZ/p^r)}$ together with isomorphisms (and no further coherence data)
\[
    X^{h(p^r\ZZ)} \iso
    (X^{h(p^{r+1}\ZZ)})^{h(p^r\ZZ/p^{r+1}\ZZ)}.
\]
Second, because $\ZZ_p$ is freely topologically generated by a single generator via the dense inclusion $\ZZ \le \ZZ_p$.
We denote by
\[
    d^*\colon \Shv(\cT_{\ZZ_p}; \cC) \adj \cC^{B\ZZ}\noloc d_*
\]
the adjunction obtained from the adjunction $\oeu \dashv \oe_*$ by further restricting and right Kan extending along the inclusion $\ZZ \le \ZZ_p$.
We now show that, under certain hypotheses, a concrete description of the hypersheafification can be given in terms of this adjunction.

\begin{prop}\label{pull-push-hyper}
    Let $\cC$ be a compactly generated $p$-complete stable category.
    Then, for every $X \in \Shv(\cT_{\ZZ_p}; \cC)$ the unit map
    \[
        X \too d_* d^* X
    \]
    exhibits $d_* d^* X$ as the hypersheafification of $X$.
\end{prop}

\begin{proof}
    As mentioned above, $d_* d^* X$ is a hypersheaf.
    Therefore, by \cref{deligne}, it suffices to show that $X \to d_* d^* X$ induces an isomorphism on stalks, namely after applying $e^*$.
    Since $e^*$ is obtained by applying $d^*$ and forgetting the $\ZZ$-action, it suffices to check that $X \to d_* d^* X$ induces an isomorphism after applying $d^*$.
    By the zigzag identity and $2$-out-of-$3$, it suffices to check that the counit map
    \[
        (d^* d_*) d^* X \too d^* X
    \]
    is an isomorphism.
    
    Next, we reduce to the case $\cC = \Spp$.
    Since $\cC$ is compactly generated, the functors
    \[
        \hom(K, -)\colon \cC \too \Spp
    \]
    ranging over all $K \in \cC^\omega$ are jointly conservative, and are colimit and limit preserving.
    Thus, they commute with $d_*$, $d^*$ and hypersheafification, so we are reduced to the case $\cC = \Spp$.
    
    Now, we can take the cofiber of multiplication by $p$ to reduce to the claim that the map
    \[
        d^* d_* d^* (X/p) \too d^* (X/p)
    \]
    is an isomorphism for every $\Sp$-valued sheaf $X$ on $\cT_{\ZZ_p}$.
    Observe that the homotopy groups of the spectrum 
    \[
        Y := d^* (X/p) \simeq \colim_r (X/p)^{h(p^r\ZZ_p)}
    \]
    are $p^2$-torsion and the $\ZZ$-action on them is restricted from a continuous $\ZZ_p$-action.
    We will prove that for any $Y$ with such homotopy groups the map
    \[
        \colim_r Y^{h(p^r \ZZ)} \simeq d^* d_* Y \too Y
    \]
    is an isomorphism by showing that it induces an isomorphism on homotopy groups.
    Note that homotopy groups commute with filtered colimits, so that $\colim_r \pi_i(Y^{h(p^r \ZZ)}) \iso \pi_i(\colim_r Y^{h(p^r \ZZ)})$, and the $i$-th homotopy group of the source of the map above fits into a short exact sequence 
    \[
        0
        \too \colim_r H^1(p^r\ZZ; \pi_{i+1}(Y))
        \too \pi_i(\colim_r Y^{h(p^r\ZZ)})
        \too \colim_r H^0(p^r\ZZ; \pi_i(Y))
        \too 0
    \]
    and the map $\pi_i(\colim_r Y^{h(p^r\ZZ)}) \to \pi_i(Y)$ factors through
    \[
        \colim_r H^0(p^r\ZZ; \pi_i(Y)) \too \pi_i(Y).
    \]
    This map is an isomorphism, since the action on $\pi_i(Y)$ comes from a continuous $\ZZ_p$-action, so that every element is fixed by $p^r\ZZ$ for large enough $r$.
    It remains to show that
    \[
        \colim_r H^1(p^r\ZZ; \pi_{i+1}(Y)) = 0.
    \]
    Since $\pi_{i+1}(Y)$ is $p^2$-torsion and the action comes from a continuous $\ZZ_p$-action, $\pi_{i+1}(Y)$ can be written as a filtered colimit of finite $p^2$-torsion groups $M_j$ with a continuous $\ZZ_p$-action.
    Since $H^1(p^r\ZZ; -)$ commutes with filtered colimits, by exchanging the order of colimits we are reduced to showing that for every $j$ we have
    \[
        \colim_r H^1(p^r\ZZ; M_j) = 0.
    \]
    Indeed, for large enough $r$, the action of $p^r\ZZ$ on $M_j$ is trivial.
    Thus, for large enough $r$, this is the $1$-st cohomology of $S^1 \simeq B(p^r\ZZ)$ with coefficients in $M_j$.
    The transitions maps going from $r$ to $r+1$ correspond to the $p$-fold covering map of the circle, thus induce multiplication by $p$ on $M_j$.
    Since $M$ is $p^2$-torsion, the result follows.
\end{proof}

Combining \cref{LR-hyper} and \cref{pull-push-hyper}, we immediately deduce the following:

\begin{cor}\label{hyper-R-loc-dd}
    Let $\cC \in \CAlg(\Prlst)$ be compactly generated and $p$-complete, and let $R$ be a continuous $\ZZ_p$-Galois extension with stalk $e^* R \in \CAlg(\cC)$.
    For $M \in \Mod_{\sR}(\Shv(\cT_{\ZZ_p}; \cC))$, the map
    \[
        M \too d_* d^* M
    \]
    exhibits the target both as the hypersheafification and as the level-wise $e^* R$-localization of $M$.
\end{cor}

\subsection{Hyperdescent and Cyclotomic Completion}

We begin by constructing the cyclotomic sheaf.
By \cite[Construction 4.5]{clausen2021hyperdescent} and cofinality there is an equivalence
\[
    \Shv(\cT_{\ZZ_p^\times}; \SpTn)
    \simeq \colim_{U \unlhd \ZZ_p^\times \text{ open }} \SpTn^{B(\ZZ_p^\times/U)}
    \simeq \colim \SpTn^{B(\ZZ/p^r)^\times}
    \qin \Prl.
\]

\begin{defn}\label{cyclosheaf}
    For $X \in \SpTn$ we define the \tdef{cyclotomic sheaf}
    \[
        \mdef{\cyc{X}{p^{(-)}}{n}}
        \qin \Shv(\cT_{\ZZ_p^\times}; \SpTn)
    \]
    to be the sheaf determined by the cyclotomic extensions $\cyc{X}{p^r}{n}$ and the isomorphisms 
    \[
        \cyc{X}{p^r}{n} \iso \cyc{X}{p^{r+1}}{n}^{h(\ZZ/p)}.
    \] 
\end{defn}

Note that the stalk of the cyclotomic sheaf is given by
\[
    \oeu \cyc{X}{p^{(-)}}{n} \simeq \colim \cyc{X}{p^r}{n} \simeq \cyc{X}{p^\infty}{n}
    \qin \SpTn^{B\ZZ_p^\times}.
\]

\begin{prop}\label{cyclo-sheaf-galois}
    The cyclotomic sheaf $\cyc{\Sph_{\Tn}}{p^{(-)}}{n}$ is a continuous $\ZZ_p^\times$-Galois extension and there is a symmetric monoidal equivalence
    \[
        \SpTn \iso \Mod_{\cyc{\Sph_{\Tn}}{p^{(-)}}{n}}(\Shv(\cT_{\ZZ_p}; \SpTn)),
        \qquad X \mapsto \cyc{X}{p^{(-)}}{n}
    \]
    Moreover, $X \in \SpTn$ is cyclotomically complete if and only the cyclotomic sheaf $\cyc{X}{p^{(-)}}{n}$ is hypercomplete.
\end{prop}

\begin{proof}
    For the first part, $\cyc{\Sph_{\Tn}}{p^{(-)}}{n}$ is Galois because each finite cyclotomic extension is a faithful Galois extension by \cite[Proposition 5.2]{carmeli2021chromatic}, and the equivalence follows from \cref{Prof_Galois_Descent}.

    For the second part, by the first part $\cyc{X}{p^{(-)}}{n}$ is a module over $\cyc{\Sph_{\Tnp}}{p^{(-)}}{n+1}$, whose stalk is $\cyc{\Sph_{\Tnp}}{p^\infty}{n+1}$.
    Since cyclotomic completion is $\cyc{\Sph_{\Tnp}}{p^\infty}{n+1}$-localization, the result follows from \cref{LR-hyper}.
\end{proof}

For $R \in \Alg(\SpTn)$, consider the cyclotomic sheaf $\cyc{R}{p^{(-)}}{n}$.
By applying $\KTnp$ level-wise we get the $\SpTnp$-valued presheaf $\KTnp(\cyc{R}{p^{(-)}}{n})$.

\begin{prop}\label{cyclosheaf-shift}
    The functor $\KTnp$ sends cyclotomic sheaves to cyclotomic sheaves, that is, there is an isomorphism
    \[
        \cyc{\KTnp(R)}{p^{(-)}}{n+1} \iso \KTnp(\cyc{R}{p^{(-)}}{n}) \qin \Shv(\cT_{\ZZ_p^{\times}}; \SpTnp) 
    \]
\end{prop}

\begin{proof}
    We need to show that the two presheaves on $\cT_{\ZZ_p^\times}$ agree.
    The source is a sheaf and in particular preserves products.
    Since $\KTnp$ preserves products, the target preserves products as well.
    Therefore, it remains to show that they agree on transitive continuous finite $\ZZ_p^\times$-sets, which is the content of \cref{inter-galois}.
\end{proof}

\begin{cor}\label{Knp-hyper}
    $\KKnp(\cyc{R}{p^{(-)}}{n})$ is a hypersheaf of $\Knp$-local spectra for any $R \in \Alg(\SpTn)$.
\end{cor}

\begin{proof}
    By \cref{cyclosheaf-shift} we know that $\KKnp(\cyc{R}{p^{(-)}}{n})$ is a cyclotomic sheaf.
    Since all $\Knp$-local spectra are cyclotomically complete (see the discussion above \cite[Question 7.36]{barthel2022chromatic}) the result follows from \cref{cyclo-sheaf-galois}.
\end{proof}

We now push the cyclotomic sheaf to $\ZZ_p$, giving a sheaf whose values are the extensions from \cref{iwa-def}.
Recall that $\ZZ_p^\times \simeq T_p \times \ZZ_p$ where $T_p$  is $(\ZZ/4)^\times$ for $p = 2$ and $\Fp^\times$ for odd primes.
Let $\pi\colon \ZZ_p^\times \to \ZZ_p$ denote the projection, and consider the functor $\cT_{\ZZ_p} \to \cT_{\ZZ_p^\times}$ given by restriction along $\pi$.
Restricting (pre)sheaves along this functor induces a geometric morphism denoted
\[
    \pi_*\colon \Shv(\cT_{\ZZ_p^\times}; \SpTn) \too \Shv(\cT_{\ZZ_p}; \SpTn).
\]

\begin{defn}\label{iwa-sheaf-def}
    For $X \in \SpTn$ we let
    \[
        \mdef{X\iwa{(-)}{n}} := \pi_*(\cyc{X}{p^{(-)}}{n})
        \qin \Shv(\cT_{\ZZ_p}; \SpTn).
    \]
\end{defn}

\begin{war}
    The notation might be somewhat confusing.
    For odd primes, the value of $X\iwa{(-)}{n}$ at the $\ZZ_p$-set $\ZZ/p^r$ is $X\iwa{r+1}{n}$, and for $p = 2$ the value at the $\ZZ_2$-set $\ZZ/2^r$ is $X\iwat{r+2}{n}$.
    In particular the value at the trivial $\ZZ_p$-set is $X$.
\end{war}

Recall that the functor
\[
    d^*\colon \Shv(\cT_{\ZZ_p}; \Sp_{\Tnp}) \too  (\Sp_{\Tnp})^{B\ZZ}
\]
is computed by a filtered colimit $d^* \sF \simeq \colim \sF^{h(p^r\ZZ_p)}$, and thus commutes with $\KTnp$.
This gives us a map
\[
    \KTnp(R\iwa{(-)}{n})
    \too d_* d^* \KTnp(R\iwa{(-)}{n})
    \iso d_* \KTnp(d^* R\iwa{(-)}{n}).
\]

\begin{thm}\label{K-hyper-cyclo}
    For any $R \in \Alg(\SpTn)$ the map
    \[
        \KTnp(R\iwa{(-)}{n})
        \too d_* \KTnp(d^* R\iwa{(-)}{n})
    \]
    exhibits the target both as the hypersheafification and as the level-wise cyclotomic completion of the source.
\end{thm}

\begin{proof}
    It suffices to show that
    \[
        \KTnp(R\iwa{(-)}{n})
        \too d_* d^* \KTnp(R\iwa{(-)}{n})
    \]
    exhibits the target both as the hypersheafification and as the level-wise cyclotomic completion of the source.
    By \cref{cyclosheaf-shift} we know that $\KTnp(R\iwa{(-)}{n})$ is module over $\Sph_{\Tnp}\iwa{(-)}{n+1}$, whose stalk is $\Sph_{\Tnp}\iwa{\infty}{n+1}$.
    By \cite[Proposition 6.19]{barthel2022chromatic}, being cyclotomically complete is also equivalent to being local with respect to $\Sph_{\Tnp}\iwa{\infty}{n+1} := \cyc{\Sph_{\Tnp}}{p^\infty}{n+1}^{hT_p}$.
    The claim now follows from \cref{hyper-R-loc-dd}.
\end{proof}

\appendix
\section{Group Algebras}

As explained in the introduction, at the core of our study is the map
\[
    \KTnp(R)[A] \too \KTnp(R[\Omega A])
\]
for a ring spectrum $R$ and a pointed connected space $A$.
The construction of this map is the combination of two ingredients.
First, for every functor $F\colon \cC \to \cD$ between categories with $A$-shaped colimits and $X\in \cC$, we have an assembly map 
\[
    F(X)[A] := \colim_AF(X) \too F(\colim_A X) =: F(X[A]).
\] 
Second, for every ring spectrum $R$, the equivalence between local systems over $A$ and representations of the group $\Omega A$ in $R$-modules gives, in the realm of small stable categories, an equivalence 
\[
    \Perf(R[\Omega A]) \simeq \Perf(R)[A]. 
\]
In \cref{asm-maps-sec} and \cref{grp-alg-sec} we review these standard constructions, and verify some of their basic naturality and multiplicativity properties necessary for our applications to $K$-theory.
Finally, for $\Tn$-local commutative rings we shall also employ a variant of the above equivalence, replacing perfect modules by dualizable $\Tn$-local modules, which we discuss in \cref{units-dbl-sec}.

\subsection{Multiplicativity of Assembly Maps}\label{asm-maps-sec}

Recall that given a functor $F \colon \cC \to \cD$, for every $X \in \cC$ and $A \in \Spc$ we have a natural assembly map
\[
    F(X)[A] \too F(X[A]).
\]
Furthermore, if $F$ is lax symmetric monoidal, we shall show that the assembly map is canonically lax symmetric monoidally natural in both $X$ and $A$.

First, we want to exhibit the source and target of the assembly map as symmetric monoidal functors in both variables. 
Let $\Catall$ be the category of cocomplete categories and colimit preserving functors. 
For every $\cC \in \Catall$ we can define the functor $\Spc \times \cC \to \cC$ given by $(A,X) \mapsto X[A]$ as a left Kan extension in the following way. Consider the functor $i: \cC \to \Spc \times \cC$ that is constant on the point in $\Spc$ and is the identity on $\cC$. The left Kan extension along $i$ is a functor of the form
\[
    i_!\colon \Fun(\cC, \cC) \too \Fun(\Spc \times \cC, \cC)
\]
which gives $(i_!\Id_\cC)(X,A) = X[A]$.
Now if $\cC \in \CAlg(\Catall)$, then the functor $i$ is symmetric monoidal, by \cite[Proposition 3.34]{linskens2022global} or \cite[Proposition 3.6]{moshe2021higher}, the functor $i_!$ is symmetric monoidal, with lax symmetric monoidal right adjoint $i^*$, with respect to the Day convolution on the source and target. 
Since commutative algebras in the Day convolution are lax symmetric monoidal functors, we get an induced adjunction
\[
    i_!\colon \Fun^\lax(\cC, \cC) \rightleftarrows \Fun^\lax(\cC \times \Spc, \cC)\noloc i^*
\]
lifting the adjunction on (not lax symmetric monoidal) functors.
Thus, the functor
\[
    \Spc \times \cC \too \cC,
    \qquad (A, X) \mapsto X[A]
\]
given by $i_! \Id_\cC$
acquires a lax symmetric monoidal structure. Since the symmetric monoidal structure of $\cC$ preserves colimits in each variable, the lax symmetric monoidal structure on $i_! \Id_\cC$ is strong.

Now, let $\cC, \cD \in \CAlg(\Catall)$ and let $F\colon \cC \to \cD$ be a lax symmetric monoidal functor.
Denote by $\widetilde{F}$ the functor given by post-composition with $F$, and consider the following commutative square:
\[\begin{tikzcd}
	{\Fun^\lax(\cC, \cC)} & {\Fun^\lax(\cC \times \Spc, \cC)} \\
	{\Fun^\lax(\cC, \cD)} & {\Fun^\lax(\cC \times \Spc, \cD)}
	\arrow["{i^*}"', from=1-2, to=1-1]
	\arrow["{i^*}"', from=2-2, to=2-1]
	\arrow["{\widetilde{F}}", from=1-2, to=2-2]
	\arrow["{\widetilde{F}}"', from=1-1, to=2-1]
\end{tikzcd}\]
Passing to horizontal left adjoints, we obtain a Beck--Chevalley map
\[
    i_! \widetilde{F} \too \widetilde{F} i_!.
\]
Evaluating this map at $\Id_\cC \in \Fun^\lax(\cC, \cC)$, and noting that $\widetilde{F}(\Id_\cC) = F$, gives a map between lax symmetric monoidal functors
\[
    i_! F \too F \circ i_! \Id_\cC.
\]
Unwinding the definitions, this is precisely the assembly map.

\begin{defn}
    For $\cC, \cD \in \CAlg(\Catall)$ and a lax symmetric monoidal functor $F\colon \cC \to \cD$ we define the map
    \[
        F(X)[A] \too F(X[A])
        \qin \cD
    \]
    to be the map $i_! F \to F \circ i_! \Id_\cC$ constructed above, so in particular as a lax symmetric monoidally natural transformation.
\end{defn}

\begin{rem}
    When $F$ is colimit preserving the assembly map defined above is an isomorphism.
\end{rem}

\subsection{Modules over Group Algebras}\label{grp-alg-sec}

We now specialize to the case where the category itself is $\Prl$ (which lives in a large version of $\Catall$) giving us the construction $\cC[A]$ symmetric monoidally naturally in $\cC \in \Prl$ and $A \in \Spc$.
If we further assume that $A$ is pointed connected, we have an identification of $\cC[A]$ with $\LMod_{\Omega A}(\cC)$.
We would like to make this identification symmetric monoidal as well.
The argument will require the following rigidity property of the category of pointed connected spaces:

\begin{lem}\label{auto-equiv-unique}
    The identity functor is the only (symmetric monoidal) auto-equivalence of the category $\Spc^{\ge 1}$, up to isomorphism.
\end{lem}
\begin{proof}
    The category 
    \[
        \Spc^{\ge 1}_* \simeq
        \Grp(\Spc)
    \] 
    is the non-abelian derived category (a.k.a \textit{animation}) of the ordinary category of groups $\Grp(\Set)$ (see, e.g., \cite[Example 5.1.6(1)]{cesnavicius2019purity}), and it is known that the latter has no non-identity auto-equivalences (see \cite[p.31 example F]{freyd1964abelian}). By (the dual version of) \cite[Proposition 2.4.3.8]{HA}, every functor admits a unique oplax symmetric monoidal structure with respect to the Cartesian monoidal structure on the source and the target, which implies the symmetric monoidal version of the claim. 
\end{proof}

\begin{prop}\label{C-A}
    There is an equivalence
    \[
        \LMod_{\Omega A}(\cC) \simeq \cC[A]
        \qin \Prl,
    \]
    symmetric monoidally natural in $A \in \Spc^{\geq 1}_*$ and $\cC \in \Prl$.
\end{prop}

\begin{proof}
    Both sides are symmetric monoidal functors 
    \[
        \Prl \times \Spc^{\ge 1}_* \too \Prl.
    \]
    That is, morphisms in a larger version of $\CMon(\Cat)$.
    As the source is a coproduct of $\Prl$ and $\Spc_*^{\ge 1}$, it suffices to identify the two functors on each coordinate separately.
    On the $\Pr^L$ factor both functors are the identity so the claim holds trivially.
    We are thus reduced to proving the claim for $\cC = \Spc$. Note that $\LMod_{\Omega A}(\Spc)$ is naturally pointed by $\Omega A$ with its regular left action on itself. Also, $\Spc[A]$ is pointed by the map $\Spc \to \Spc[A]$ induced from the basepoint of $A$. Hence, both promote to symmetric monoidal functors 
    \[
        \Spc^{\ge 1}_* \too \Prl_*,
    \]
    where $\Prl_*$ is the category of presentable categories with a chosen object (equivalently, the under category $\Prl_{\Spc/}$).
    It would suffice to show they are naturally isomorphic as such. 

    We begin by showing that both of them are fully faithful. For $\LMod_{\Omega(-)}(\Spc)$ this follows from \cite[Theorem 4.8.5.11]{HA}. For $\Spc[-]$ we argue as follows. First, for all $A,B \in \Spc^{\ge 1}_*$, restriction along the Yoneda embedding $A \into \Spc[A]$ provides an isomorphism
    \[
        \Map_{\Prl_*}(\Spc[A],\Spc[B]) \simeq
        \Map_{\widehat{\Cat}_*}(A,\Spc[B])
    \]
    by the universal property of $\Spc[A]$ as a free co-completion of $A$. Now, since $A$ is \textit{connected}, every \textit{pointed} functor $A \to \Spc[B]$ factors uniquely through the Yoneda embedding $B \into \Spc[B]$, and hence provides an isomorphism
    \[
        \Map_{\widehat{\Cat}_*}(A,\Spc[B]) \simeq
        \Map_{\Spc^{\ge 1}_*}(A,B).
    \]
    It thus follows that $\Spc[-]$ is fully faithful as well.
    Furthermore, the essential images of the two functors are the same (see, e.g.,  \cite[Proposition 4.4]{carmeli2021chromatic}), so the composition of one with the inverse image of the other is an auto-equivalence of $\Spc^{\ge 1}_*$, which by \Cref{auto-equiv-unique} must be the identity.
\end{proof}

For $\cC \in \CAlg(\Catall)$ 
applying $\Alg(-)$ to the symmetric monoidal functor $(A, X) \mapsto X[A]$ and pre-composing with $\Omega\colon \Spc_*^{\ge 1} \to \Mon(\Spc)$ gives a symmetric monoidal functor
\[
    \Spc^{\geq 1}_* \times \Alg(\cC) \too \Alg(\cC),
    \quad (A, R) \mapsto R[\Omega A].
\]
Similarly, applying $\CAlg(-)$ and pre-composing with $\Spcn \to \CMon(\Spc)$ we get a symmetric monoidal functor
\[
    \Spcn \times \CAlg(\cC) \too \CAlg(\cC),
    \quad (M, R) \mapsto R[M].
\]

\begin{cor}\label{Mod-R-A}
    There is an equivalence
    \[
        \LMod_{R[\Omega A]}(\cC) \simeq 
        \LMod_R(\cC)[A]
        \qin \Prl_*,
    \]
    symmetric monoidally natural in $A \in \Spc^{\geq 1}_*$ and $(\cC, R) \in \PrlAlg$. 
    Similarly, there is an equivalence
    \[
        \Mod_{R[M]}(\cC) \simeq \Mod_R(\cC)[\Sigma M]
        \qin \CAlg(\Prl),
    \]
    natural in $M \in \Spcn$ and $(\cC, R) \in \PrlCAlg$.
\end{cor}
\begin{proof}
    By composing (the pointed version of) the equivalence of \Cref{C-A} with the symmetric monoidal functor $\PrlAlg \to \Prl_*$ taking $(\cC, R)$ to $\LMod_R(\cC)$ pointed by $R$ we obtain a symmetric monoidal equivalence 
    \[
        \LMod_{\Omega A}(\LMod_R(\cC)) \simeq 
        \LMod_R(\cC)[A]
        \qin \Prl_*.
    \]
    It thus remains to identify the left-hand side above with $\LMod_{R[\Omega A]}(\cC)$. As in the proof of \Cref{C-A}, these are two symmetric monoidal functors 
    \[
        \PrlAlg \times \Spc_*^{\ge 1} \too \Prl_*,
    \]
    and since the source is a coproduct in the category of symmetric monoidal categories, it suffices to make the identification separately in each coordinate, both of which are obvious. 
    
    The second part follows by applying $\CAlg$ to the above equivalence. 
\end{proof}

\begin{cor}\label{perf-R-A}
    There is an equivalence
    \[
        \Perf(R[\Omega A]) \simeq \Perf(R)[A]
        \qin \Catperf,
    \]
    symmetric monoidally natural in $A \in \Spc^{\geq 1}_*$ and $R \in \Alg(\Sp)$.
\end{cor}

\begin{proof}
    Consider the natural equivalence of functors obtained from the case $\cC = \Sp$ of \cref{Mod-R-A}.
    Observe that by \cref{ind-w}, the symmetric monoidal inclusion $\Prlstw \into \Prlst$ is colimit preserving, and by \cref{LMod}, the two functors factor through it.
    The result then follows by applying the equivalence $(-)^\omega$ of \cref{ind-w}.
\end{proof}

\begin{rem}\label{variants-remark}
    Various versions of \Cref{C-A} and its corollaries were formulated and proved in the literature, including two different papers of the authors (\cite[Proposition 4.4]{carmeli2021chromatic}, \cite[Proposition 5.11]{barthel2022chromatic}).
    We hope that this version, which is compatible with the cited ones, but also
    includes the multiplicative naturality in all variables, will meet all our future requirements, so that we shall not have to prove it again. 
\end{rem}

\subsection{Dualizable Modules}\label{units-dbl-sec}

We first recall some generalities on Beck--Chevalley maps.
Let
\[\begin{tikzcd}
	{\cC_0} && {\cC_1} \\
	{\cD_0} && {\cD_1}
	\arrow["{L^\cC}", from=1-1, to=1-3]
	\arrow["{F_0}"', from=1-1, to=2-1]
	\arrow["{F_1}", from=1-3, to=2-3]
	\arrow["{L^\cD}"', from=2-1, to=2-3]
	\arrow["\alpha", Rightarrow, from=2-1, to=1-3]
\end{tikzcd}\]
be a lax commutative square of categories, and assume that $L^\cC$ admits a right adjoint $R^\cC$ with unit $u^\cC$ and counit $c^\cC$, and similarly for $L^\cD$.
Then, there is a Beck--Chevalley transformation given by the composition
\[
    F_0 R^\cC
    \xRightarrow{u^\cD} R^\cD L^\cD F_0 R^\cC
    \xRightarrow{\alpha} R^\cD F_1 L^\cC R^\cC
    \xRightarrow{c^\cC} R^\cD F_1.
\]
The Beck--Chevalley map is compatible with the unit maps in the sense that the following diagram commutes
\[\begin{tikzcd}
	{F_0} & {R^\cD L^\cD F_0} \\
	{F_0 R^\cC L^\cC} & {R^\cD F_1 L^\cC}
	\arrow["{u^\cD}", Rightarrow, from=1-1, to=1-2]
	\arrow[Rightarrow, from=2-1, to=2-2]
	\arrow["{u^\cC}"', Rightarrow, from=1-1, to=2-1]
	\arrow["\alpha", Rightarrow, from=1-2, to=2-2]
\end{tikzcd}\]

We now apply this in a specific context.
Let $\cC \in \CAlg(\Catall)$.
Recall that there is a group algebra--units adjunction:
\[
    \unit_\cC[-]\colon \Spcn \adj \CAlg(\cC)\noloc (-)^\times.
\]
For $\cC, \cD \in \CAlg(\Catall)$ and a lax symmetric monoidal functor $F\colon \cC \to \cD$ there is a natural assembly map
\[
    F(\unit_\cC)[M] \too F(\unit_\cC[M])
\]
of functors $\Spcn \to \CAlg_{F(\unit_\cC)}(\cD)$.
In other words, this provides the $2$-morphism depicted in the following diagram:
\[\begin{tikzcd}
	\Spcn && {\CAlg(\cC)} \\
	\Spcn && {\CAlg_{F(\unit_\cC)}(\cD)}
	\arrow["F"', from=1-3, to=2-3]
	\arrow["{\unit_\cC[-]}", from=1-1, to=1-3]
	\arrow["{F(\unit_\cC)[-]}"', from=2-1, to=2-3]
	\arrow[Rightarrow, no head, from=1-1, to=2-1]
	\arrow[Rightarrow, from=2-1, to=1-3]
\end{tikzcd}\]

\begin{defn}\label{units-lax-fun}
    Let $\cC, \cD \in \CAlg(\Catall)$ and let $F\colon \cC \to \cD$ be a lax symmetric monoidal functor.
    We define the natural map
    \[
        R^\times \too F(R)^\times
    \]
    of functors $\CAlg(\cC) \to \Spcn$ to be the Beck--Chevalley map associated to the above lax commutative square under the group algebra--units adjunctions.
\end{defn}

Namely, at $R\in \calg(\cC)$, the resulting map $R^\times \to F(R)^\times$ is the composition
\[
    R^\times
     \iso \hom(\unit_\cC[\Sph], R)
     \too \hom(F(\unit_\cC[\Sph]), F(R))
     \too \hom(\unit_\cD[\Sph], F(R))
     \iso F(R)^\times
\]
where the third morphism is obtained by lax unitality of $F$ and the assembly map
\[
    \unit_\cD[\Sph]
    \too F(\unit_\cC)[\Sph]
    \too F(\unit_\cC[\Sph]).
\]

\begin{lem}\label{units-group-algebras}
    Let $\cC, \cD \in \CAlg(\Catall)$ and let $F\colon \cC \to \cD$ be a lax symmetric monoidal functor.
    Then the following diagram commutes naturally in $M \in \Spcn$:
    \[\begin{tikzcd}
    	M & {F(\one_\cC)[M]^\times} \\
    	{\one_\cC[M]^\times} & {F(\one_\cC[M])^\times}
    	\arrow[from=1-1, to=2-1]
    	\arrow[from=2-1, to=2-2]
    	\arrow[from=1-2, to=2-2]
    	\arrow[from=1-1, to=1-2]
    \end{tikzcd}\]
\end{lem}

\begin{proof}
    This is the compatibility of the Beck--Chevalley map with units.
\end{proof}

Consider now the group algebra--unit adjunction applied to the category $\Cat \in \CAlg(\Catall)$:
\[
    \pt[-]\colon \Spcn \adj \CAlg(\Cat)\noloc (-)^\times.
\]
Every invertible object of a symmetric monoidal category is in particular dualizable, that is, the inclusion $\cC^\dbl \into \cC$ induces an isomorphism $(\cC^\dbl)^\times \iso \cC^\times$.
In other words, we have a commutative square of categories:
\[
\begin{tikzcd}
	{\CAlg(\Cat)} & \Spcn \\
	{\CAlg(\Cat)} & \Spcn 
	\arrow["{(-)^\times}", from=1-1, to=1-2]
	\arrow[""{name=1, anchor=center, inner sep=0}, "{(-)^\dbl}"', from=1-1, to=2-1]
	\arrow["{(-)^\times}"', from=2-1, to=2-2]
 \arrow[equal, from=1-2, to=2-2]
\end{tikzcd}
\]
Now, for $\cC\in \calg(\Prlst)$, we have $\cC^\dbl \in \calg(\Catperf)$, and the formation of dualizable objects refines to a functor 
\[
(-)^\dbl \colon \calg_{\cC}(\Prlst) \to \calg_{\cC^\dbl}(\Catperf). 
\]
By restricting the (large categories version) of the commutative square above to presentable stable categories, we obtain the following commutative square:
\[
\begin{tikzcd}
	{\CAlg_\cC(\Prlst)} & \Spcn \\
	{\CAlg_{\cC^\dbl}(\Catperf)} & \Spcn 
	\arrow["{(-)^\times}", from=1-1, to=1-2]
	\arrow[""{name=1, anchor=center, inner sep=0}, "{(-)^\dbl}"', from=1-1, to=2-1]
	\arrow["{(-)^\times}"', from=2-1, to=2-2]
 \arrow[equal, from=1-2, to=2-2]
\end{tikzcd}
\]
Recall that the horizontal maps admit left adjoints $\cC[-]$ and $\cC^\dbl[-]$.

\begin{defn}\label{Cdbl-in-out}
    For $\cC \in \CAlg(\Prlst)$ we define the natural transformation
    \[
    \cC^\dbl[-] \too \cC[-]^\dbl 
    \]
    between functors $\Spcn \to \calg_{\cC^\dbl}(\Catperf)$ to be the Beck--Chevalley transformation corresponding to the square above.
\end{defn}

Observe that for a colimit preserving symmetric monoidal functor $F\colon \cC \to \cD \in \Prlst$ the assembly map is an isomorphism
\[
    F(R)[\Omega A] \iso F(R[\Omega A]).
\]
Namely, the group algebra is unambiguously defined. 

\begin{prop}\label{Mod-R-dbl-F}
    Let $F\colon \cC \to \cD \in \CAlg(\Prlst)$ be a colimit preserving symmetric monoidal functor between presentably symmetric monoidal stable categories.
    Then, there is a commutative diagram in $\Catperf$
    \[\begin{tikzcd}
    	{\Mod_R^\dbl(\cC)[\Sigma M]} & {\Mod_R(\cC)[\Sigma M]^\dbl} & {\Mod_{R[M]}^\dbl(\cC)} \\
    	{\Mod_{F(R)}^\dbl(\cD)[\Sigma M]} & {\Mod_{F(R)}(\cD)[\Sigma M]^\dbl} & {\Mod_{F(R)[M]}^\dbl(\cD)}
    	\arrow[from=1-1, to=2-1]
    	\arrow[from=1-2, to=2-2]
    	\arrow[from=1-3, to=2-3]
    	\arrow[from=1-1, to=1-2]
    	\arrow[from=2-1, to=2-2]
    	\arrow["\sim", from=1-2, to=1-3]
    	\arrow["\sim", from=2-2, to=2-3]
    \end{tikzcd}\]
    lax symmetric monoidally natural in $M \in \Spcn$ and $R \in \CAlg(\cC)$.
\end{prop}

\begin{proof}
    Applying \cref{Mod-R-A} to $F\colon \cC \to \cD$ gives the commutative square in $\CAlg(\Prlst)$
    \[\begin{tikzcd}
    	{\Mod_R(\cC)[\Sigma M]} & {\Mod_{R[M]}(\cC)} \\
    	{\Mod_{F(R)}(\cD)[\Sigma M]} & {\Mod_{F(R)[M]}(\cD)}
    	\arrow[from=1-1, to=2-1]
    	\arrow[from=1-2, to=2-2]
    	\arrow["\sim", from=1-1, to=1-2]
    	\arrow["\sim", from=2-1, to=2-2]
    \end{tikzcd}\]
    lax symmetric monoidally naturally in $M \in \Spcn$ and $R \in \CAlg(\cC)$.
    Taking dualizable objects, we get the right commutative square in question.
    The left commutative square in question is the naturality of the map $\cE^\dbl[\Sigma M] \to \cE[\Sigma M]^\dbl$ in $\cE$, applied to
    \[
        \Mod_R(\cC) \too \Mod_{F(R)}(\cD).
    \]
\end{proof}

\bibliographystyle{alpha}
\phantomsection\addcontentsline{toc}{section}{\refname}
\bibliography{cyclored}

\newcommand{\etalchar}[1]{$^{#1}$}
\begin{thebibliography}{LMMT24}

\bibitem[ABM22]{ausoni2022adjunction}
Christian Ausoni, Haldun~\"Ozg\"ur Bay{\i}nd{\i}r, and Tasos Moulinos.
\newblock {Adjunction of roots, algebraic $K$-theory and chromatic redshift},
  2022.

\bibitem[AR02]{AR02}
Christian Ausoni and John Rognes.
\newblock {Algebraic K-theory of topological K-theory}.
\newblock {\em Acta Mathematica}, 188(1):1 -- 39, 2002.

\bibitem[AR08]{ausoni2008chromatic}
Christian Ausoni and John Rognes.
\newblock {The chromatic red-shift in algebraic K-theory}.
\newblock {\em L`Enseignement Math{\'e}matique}, 54(2):13--15, 2008.

\bibitem[Ban17]{banerjee2017galois}
Romie Banerjee.
\newblock {Galois descent for real spectra}.
\newblock {\em Journal of Homotopy and Related Structures}, 12(2):273--297,
  2017.

\bibitem[BCM20]{bhatt2020remarks}
Bhargav Bhatt, Dustin Clausen, and Akhil Mathew.
\newblock {Remarks on $K(1)$-local $K$-theory}.
\newblock {\em Selecta Mathematica}, 26(3):39, 2020.

\bibitem[BCSY24]{barthel2022chromatic}
Tobias Barthel, Shachar Carmeli, Tomer~M. Schlank, and Lior Yanovski.
\newblock {The Chromatic Fourier Transform}.
\newblock {\em Forum of Mathematics, Pi}, 12:e8, 2024.

\bibitem[BGT13]{blumberg2013universal}
Andrew~J Blumberg, David Gepner, and Gon{\c{c}}alo Tabuada.
\newblock {A universal characterization of higher algebraic K-theory}.
\newblock {\em Geometry \& Topology}, 17(2):733--838, 2013.

\bibitem[BGT14]{blumberg2014uniqueness}
Andrew~J Blumberg, David Gepner, and Gon{\c{c}}alo Tabuada.
\newblock {Uniqueness of the multiplicative cyclotomic trace}.
\newblock {\em Advances in Mathematics}, 260:191--232, 2014.

\bibitem[BHLS23]{Telefalse}
Robert Burklund, Jeremy Hahn, Ishan Levy, and Tomer~M. Schlank.
\newblock {$K$-theoretic counterexamples to Ravenel's telescope conjecture}.
\newblock {\em arXiv preprint arXiv:2310.17459}, 2023.

\bibitem[BMS24]{moshe2021higher}
Shay Ben-Moshe and Tomer~M. Schlank.
\newblock {Higher semiadditive algebraic K-theory and redshift}.
\newblock {\em Compositio Mathematica}, 160(2):237--287, 2024.

\bibitem[Bou01]{bousfield2001telescopic}
A~Bousfield.
\newblock {On the telescopic homotopy theory of spaces}.
\newblock {\em Transactions of the American Mathematical Society},
  353(6):2391--2426, 2001.

\bibitem[BSY22]{Null}
Robert {Burklund}, Tomer~M. {Schlank}, and Allen {Yuan}.
\newblock {The Chromatic Nullstellensatz}.
\newblock {\em arXiv e-prints}, page arXiv:2207.09929, July 2022.

\bibitem[CCRY22]{cnossen2022characters}
Bastiaan Cnossen, Shachar Carmeli, Maxime Ramzi, and Lior Yanovski.
\newblock {Characters and transfer maps via categorified traces}.
\newblock {\em arXiv preprint arXiv:2210.17364}, 2022.

\bibitem[CDH{\etalchar{+}}20]{hermitian-2}
Baptiste Calm\`es, Emanuele Dotto, Yonatan Harpaz, Fabian Hebestreit, Markus
  Land, Kristian Moi, Denis Nardin, Thomas Nikolaus, and Wolfgang Steimle.
\newblock {Hermitian K-theory for stable $\infty$-categories II: Cobordism
  categories and additivity}.
\newblock {\em arXiv preprint arXiv:2009.07224}, 2020.

\bibitem[CM21]{clausen2021hyperdescent}
Dustin Clausen and Akhil Mathew.
\newblock {Hyperdescent and {\'e}tale K-theory}.
\newblock {\em Inventiones mathematicae}, pages 1--96, 2021.

\bibitem[CMNN24]{clausen2020descent}
Dustin Clausen, Akhil Mathew, Niko Naumann, and Justin Noel.
\newblock {Descent and vanishing in chromatic algebraic $K$-theory via group
  actions}.
\newblock {\em Annales Scientifiques de l'{\'E}cole Normale Sup{\'e}rieure},
  57(4):1135--1190, 2024.

\bibitem[{\v{C}}S24]{cesnavicius2019purity}
K{\k{e}}stutis {\v{C}}esnavi{\v{c}}ius and Peter Scholze.
\newblock {Purity for flat cohomology}.
\newblock {\em Annals of Mathematics}, 199(1):51--180, 2024.

\bibitem[CSY21a]{AmbiHeight}
Shachar Carmeli, Tomer~M. Schlank, and Lior Yanovski.
\newblock {Ambidexterity and height}.
\newblock {\em Advances in Mathematics}, 385:107763, 2021.

\bibitem[CSY21b]{carmeli2021chromatic}
Shachar Carmeli, Tomer~M. Schlank, and Lior Yanovski.
\newblock {Chromatic Cyclotomic Extensions}.
\newblock {\em arXiv preprint arXiv:2103.02471}, 2021.

\bibitem[CSY22]{TeleAmbi}
Shachar Carmeli, Tomer~M. Schlank, and Lior Yanovski.
\newblock {Ambidexterity in chromatic homotopy theory}.
\newblock {\em Inventiones mathematicae}, 228(3):1145--1254, 2022.

\bibitem[DH04]{devinatz2004homotopy}
Ethan~S Devinatz and Michael~J Hopkins.
\newblock {Homotopy fixed point spectra for closed subgroups of the Morava
  stabilizer groups}.
\newblock {\em Topology}, 43(1):1--47, 2004.

\bibitem[DM98]{dwyer1998k}
William~G Dwyer and Stephen~A Mitchell.
\newblock {On the K-theory spectrum of a ring of algebraic integers}.
\newblock {\em K-theory}, 14(3):201--264, 1998.

\bibitem[Fre64]{freyd1964abelian}
Peter~J Freyd.
\newblock {\em {Abelian categories}}, volume 1964.
\newblock Harper \& Row New York, 1964.

\bibitem[GL21]{gepner2021brauer}
David Gepner and Tyler Lawson.
\newblock {Brauer groups and Galois cohomology of commutative ring spectra}.
\newblock {\em Compositio Mathematica}, 157(6):1211--1264, 2021.

\bibitem[Hai21]{haine2021nonabelian}
Peter~J Haine.
\newblock {From nonabelian basechange to basechange with coefficients}.
\newblock {\em arXiv preprint arXiv:2108.03545}, 2021.

\bibitem[Har20]{harpaz2020ambidexterity}
Yonatan Harpaz.
\newblock {Ambidexterity and the universality of finite spans}.
\newblock {\em Proceedings of the London Mathematical Society},
  121(5):1121--1170, 2020.

\bibitem[HL13]{AmbiKn}
Michael Hopkins and Jacob Lurie.
\newblock {Ambidexterity in $K(n)$-local stable homotopy theory}.
\newblock {\em preprint}, 2013.

\bibitem[HPT23]{haine2020homotopy}
Peter~J Haine, Mauro Porta, and Jean-Baptiste Teyssier.
\newblock {The homotopy-invariance of constructible sheaves}.
\newblock {\em Homology, Homotopy and Applications}, 25(2):97--128, 2023.

\bibitem[HW22]{hahn-wilson}
Jeremy Hahn and Dylan Wilson.
\newblock {Redshift and multiplication for truncated Brown--Peterson spectra}.
\newblock {\em Annals of Mathematics}, 196(3):1277 -- 1351, 2022.

\bibitem[LMMT24]{land2020purity}
Markus Land, Akhil Mathew, Lennart Meier, and Georg Tamme.
\newblock {Purity in chromatically localized algebraic $K$-theory}.
\newblock {\em Journal of the American Mathematical Society}, 37(4):1011--1040,
  2024.

\bibitem[LNP22]{linskens2022global}
Sil Linskens, Denis Nardin, and Luca Pol.
\newblock {Global homotopy theory via partially lax limits}.
\newblock {\em arXiv preprint arXiv:2206.01556}, 2022.

\bibitem[LT19]{land2019k}
Markus Land and Georg Tamme.
\newblock {On the K-theory of pullbacks}.
\newblock {\em Annals of Mathematics}, 190(3):877--930, 2019.

\bibitem[Lura]{HA}
Jacob Lurie.
\newblock {Higher Algebra}.
\newblock {http://www.math.harvard.edu/~lurie/}.

\bibitem[Lurb]{SAG}
Jacob Lurie.
\newblock {Spectral Algebraic Geometry}.
\newblock {https://www.math.ias.edu/~lurie/papers/SAG-rootfile.pdf}.

\bibitem[Lur09]{HTT}
Jacob Lurie.
\newblock {\em {Higher topos theory}}, volume 170 of {\em Annals of Mathematics
  Studies}.
\newblock Princeton University Press, Princeton, NJ, 2009.

\bibitem[Mat16]{AkhilGalois}
Akhil Mathew.
\newblock {The Galois group of a stable homotopy theory}.
\newblock {\em Advances in Mathematics}, 291:403--541, 2016.

\bibitem[Mil92]{miller1992finite}
Haynes Miller.
\newblock {Finite localizations}.
\newblock {\em Bol. Soc. Mat. Mexicana (2)}, 37(1-2):383--389, 1992.

\bibitem[Mit00]{mitchell2000topological}
Stephen~A Mitchell.
\newblock {Topological K-theory of algebraic K-theory spectra}.
\newblock {\em K-theory}, 21(3):229--248, 2000.

\bibitem[Mor23]{mor2023picard}
Itamar Mor.
\newblock {Picard and Brauer groups of $K(n)$-local spectra via profinite
  Galois descent}.
\newblock {\em arXiv preprint arXiv:2306.05393}, 2023.

\bibitem[Rog08]{RognesGal}
John Rognes.
\newblock {\em {Galois Extensions of Structured Ring Spectra/Stably Dualizable
  Groups: Stably Dualizable Groups}}, volume 192.
\newblock American Mathematical Soc., 2008.

\bibitem[SS03]{schwede2003stable}
Stefan Schwede and Brooke Shipley.
\newblock {Stable model categories are categories of modules}.
\newblock {\em Topology}, 42(1):103--153, 2003.

\bibitem[Wal84]{waldhausen1984algebraic}
Friedhelm Waldhausen.
\newblock {Algebraic K-theory of spaces, localization, and the chromatic
  filtration of stable homotopy}.
\newblock In {\em Algebraic Topology Aarhus 1982: Proceedings of a conference
  held in Aarhus, Denmark, August 1--7, 1982}, pages 173--195. Springer, 1984.

\bibitem[Wes17]{Westerland}
Craig Westerland.
\newblock {A higher chromatic analogue of the image of J}.
\newblock {\em Geometry \& Topology}, 21(2):1033--1093, 2017.

\bibitem[Yua24]{yuan2021examples}
Allen Yuan.
\newblock {Examples of chromatic redshift in algebraic $K$-theory}.
\newblock {\em Journal of the European Mathematical Society}, 2024.

\end{thebibliography}

\end{document}